\documentclass[letterpaper, 11pt,  reqno]{amsart}

\usepackage[margin=1.1in,marginparwidth=1.5cm, marginparsep=0.5cm]{geometry}

\usepackage{amsmath,amssymb,amscd,amsthm,amsxtra, esint, comment}

\usepackage{adjustbox}

\usepackage{kotex}
\usepackage{mathrsfs} 

\usepackage[implicit=true]{hyperref}

\usepackage{color} 
\usepackage{ulem}
\usepackage[makeroom]{cancel}

\allowdisplaybreaks[2]

\definecolor{gr}{rgb}   {0.,   0.69,   0.23 }
\definecolor{bl}{rgb}   {0.,   0.5,   1. }
\definecolor{mg}{rgb}   {0.85,  0.,    0.85}
\definecolor{yl}{rgb}   {0.8,  0.7,   0.}
\definecolor{or}{rgb}  {0.7,0.2,0.2}

\newtheorem{theorem}{Theorem} [section]

\newtheorem{lemma}[theorem]{Lemma}
\newtheorem{proposition}[theorem]{Proposition}
\newtheorem{remark}[theorem]{Remark}

\newtheorem{corollary}[theorem]{Corollary}

\DeclareMathOperator*{\supp}{supp}

\DeclareMathOperator{\Id}{Id}

\newcommand{\I}{\hspace{0.5mm}\text{I}\hspace{0.5mm}}
\newcommand{\noi}{\noindent}
\newcommand{\Z}{\mathbb{Z}}
\newcommand{\R}{\mathbb{R}}

\newcommand{\T}{\mathbb{T}}

\let\P= \undefined
\newcommand{\P}{\mathbf{P}}

\newcommand{\E}{\mathbb{E}}

\newcommand{\be}{\beta}
\newcommand{\dl}{\delta}

\newcommand{\nb}{\nabla}

\newcommand{\Dl}{\Delta}
\newcommand{\eps}{\varepsilon}

\newcommand{\g}{\gamma}

\newcommand{\ld}{\lambda}

\newcommand{\s}{\sigma}

\newcommand{\wt}{\widetilde}
\newcommand{\cj}{\overline}
\newcommand{\dx}{\partial_x}

\newcommand{\dd}{\partial}

\newcommand{\ta}{\theta}

\newcommand{\les}{\lesssim}
\newcommand{\ges}{\gtrsim}

\newcommand{\jb}[1]
{\langle #1 \rangle}

\newcommand{\ind}{\mathbf 1}

\newcommand{\M}{\mathcal{M}}

\newcommand{\N}{\mathbb{N}}

\newcommand{\J}{\mathcal{J}}

\newtheorem*{ackno}{Acknowledgements}

\numberwithin{equation}{section}
\numberwithin{theorem}{section}

\newcommand{\PP}{\mathbb{P}}

\DeclareMathOperator{\Law}{Law}

\newcommand{\dr}{\theta}
\newcommand{\Dr}{\Theta}

\newcommand{\Ha}{\mathbb{H}_a}

\usepackage{tikz}

\usetikzlibrary{shapes.misc}
\usetikzlibrary{shapes.symbols}
\usetikzlibrary{shapes.geometric}
\tikzset{
	dot/.style={circle,fill=black,draw=black,inner sep=1pt,minimum size=0.5mm},
	>=stealth,
	}
\tikzset{
	ddot/.style={circle,fill=white,draw=black,inner sep=2pt,minimum size=0.8mm},
	>=stealth,
	}

\tikzset{decision/.style={ % requires library shapes.geometric
        draw,
        diamond,
        aspect=1.5
    }}

\tikzset{dia2/.style
={diamond,fill=white,draw=black,inner sep=0pt,minimum size=1mm},
	>=stealth,
	}

\tikzset{dia/.style
={star,fill=black,draw=black,inner sep=0pt,minimum size=1mm},
	>=stealth,
	}

\colorlet{symbols}{black}
\colorlet{testcolor}{green!60!black}

\def\1{\mathbf{{1}}}

\usetikzlibrary{shapes.misc}
\usetikzlibrary{shapes.symbols}
\usetikzlibrary{snakes}
\usetikzlibrary{decorations}
\usetikzlibrary{decorations.markings}
\definecolor{dblue}{rgb}{0.1, 0.1, 0.9}

\tikzset{
	root/.style={circle,fill=testcolor,inner sep=0pt, minimum size=2mm},		
	dot/.style={circle,fill=black,draw=black, solid,inner sep=0pt,minimum size=0.75mm},
	bdot/.style={circle,fill=blue,draw=dblue, solid,inner sep=0pt,minimum size=0.75mm},
		}
\colorlet{symbols}{blue!90!black}

\makeatletter
\def\DeclareSymbol#1#2#3{\expandafter\gdef\csname MH@symb@#1\endcsname{\tikz[baseline=#2,scale=0.15]{#3}}%
\expandafter\gdef\csname MH@symb@#1s\endcsname{\scalebox{0.6}{\tikz[baseline=#2,scale=0.15]{#3}}}}
\def\<#1>{\csname MH@symb@#1\endcsname}
\makeatother

\DeclareSymbol{sunset}{-3}{\draw (0,0) circle [x radius = 1.5, y radius = 1]; \draw (-1.5,0) node[dot] {} -- (1.5,0) node[dot] {};}
\DeclareSymbol{tadpole}{-3}{\draw (0,0) circle [x radius = 0.75, y radius = 1]; \draw (0,-1) node[dot] {};}

\DeclareSymbol{1}{0}{\draw[white] (-.4,0) -- (.4,0); \draw (0,0)  -- (0,1.2) node[dot] {};}
\DeclareSymbol{2}{0}{\draw (-0.5,1.2) node[dot] {} -- (0,0) -- (0.5,1.2) node[dot] {};}
\DeclareSymbol{3}{0}{\draw (0,0) -- (0,1.2) node[dot] {}; \draw (-.7,1) node[dot] {} -- (0,0) -- (.7,1) node[dot] {};}
\DeclareSymbol{4}{0}{\draw (-0.36,1.2) node[dot] {} -- (0,0) -- (0.36,1.2) node[dot] {}; \draw (-1,1) node[dot] {} -- (0,0) -- (1,1) node[dot] {};}
\DeclareSymbol{30}{-3}{\draw (0,0) -- (0,-1); \draw (0,0) -- (0,1.2) node[dot] {}; \draw (-.7,1) node[dot] {} -- (0,0) -- (.7,1) node[dot] {};}
\DeclareSymbol{31}{-3}{\draw (0,0) -- (0,-1) -- (1,0) node[dot] {}; \draw (0,0) -- (0,1.2) node[dot] {}; \draw (-.7,1) node[dot] {} -- (0,0) -- (.7,1) node[dot] {};}
\DeclareSymbol{32}{-3}{\draw (0,0) -- (0,-1) -- (1,0) node[dot] {}; \draw (0,0) -- (0,-1) -- (-1,0) node[dot] {}; \draw (0,0) -- (0,1.2) node[dot] {}; \draw (-.7,1) node[dot] {} -- (0,0) -- (.7,1) node[dot] {};}
\DeclareSymbol{20}{-3}{\draw (0,0) -- (0,-1);\draw (-.7,1) node[dot] {} -- (0,0) -- (.7,1) node[dot] {};}
\DeclareSymbol{22}{-3}{\draw (0,0.3) -- (0,-1) -- (1,0) node[dot] {}; \draw (0,0.3) -- (0,-1) -- (-1,0) node[dot] {};\draw (-.7,1) node[dot] {} -- (0,0.3) -- (.7,1) node[dot] {};}
\DeclareSymbol{202}{-3}{\draw (-.6625,0) -- (0,-1) -- (.6625,0)  {}; \draw (-1.1,1) node[dot] {} -- (-.6625,0) -- (-0.365,1) node[dot] {}; \draw (0.365,1) node[dot] {} -- (.6625,0) -- (1.1,1) node[dot] {};}

\makeatletter
\def\DeclareSymbol#1#2#3{\expandafter\gdef\csname MH@symb@#1\endcsname{\tikz[baseline=#2,scale=0.15]{#3}}}
\def\<#1>{\csname MH@symb@#1\endcsname}
\makeatother

\DeclareSymbol{X}{-2.4}{\node[dot] {};}
\DeclareSymbol{1}{0}{\draw[white] (-.4,0) -- (.4,0); \draw (0,0)  -- (0,1.2) node[dot] {};}
\DeclareSymbol{2}{0}{\draw (-0.5,1.2) node[dot] {} -- (0,0) -- (0.5,1.2) node[dot] {};}

\DeclareSymbol{sunset}{-3}{\draw (0,0) circle [x radius = 1.5, y radius = 1]; \draw (-1.5,0) node[dot] {} -- (1.5,0) node[dot] {};}
\DeclareSymbol{tadpole}{-3}{\draw (0,0) circle [x radius = 0.75, y radius = 1]; \draw (0,-1) node[dot] {};}

\DeclareSymbol{T0}{-2.7}
 { \draw (0,0) node[ddot]{};}

\DeclareSymbol{T1}{0}
 {  %\coordinate [label=above:$j=1$] (A) at (0,0);
 \draw (0,0) node[dot]{} -- (0,4) node[ddot] {}; 
 \draw (-3,0) node[dot] {} -- (0,4)node[ddot] {} -- (3,0) node[dot] {};
\draw[line width=1pt] (0, 4)node[ddot, label=above:$r_1$]{} 
.. controls(3,6) and (5,6) ..
(8, 4)node[ddot, label=above:$r_2$]{}
(4, 12) node[label ] {$\underline{j = 1}$};
 }

\DeclareSymbol{T2}{0}
 {\draw (0,0) node[dot]{} -- (0,4) node[ddot] {}; 
 \draw (-3,0) node[dot] {} -- (0,4)node[ddot] {} -- (3,0) node[dot] {};
 \draw (0,-4) node[dot]{} -- (0,0) node[dot] {}; 
 \draw (-3,-4) node[dot] {} -- (0,0)node[dot] {} -- (3,-4) node[dot] {};
\draw[line width=1pt] (0, 4)node[ddot, label=above:$r_1$]{} 
.. controls(3,6) and (5,6) ..
(8, 4)node[ddot, label=above:$r_2$]{}
(4, 12) node[label ] {$\underline{j = 2}$};
 }

\DeclareSymbol{T3}{0}
 {\draw (0,0) node[dot]{} -- (0,4) node[ddot] {}; 
 \draw (-3,0) node[dot] {} -- (0,4)node[ddot] {} -- (3,0) node[dot] {};
 \draw (0,-4) node[dot]{} -- (0,0) node[dot] {}; 
 \draw (-3,-4) node[dot] {} -- (0,0)node[dot] {} -- (3,-4) node[dot] {};
\draw (10,0) node[dot]{} -- (10,4) node[ddot] {}; 
 \draw (7,0) node[dot] {} -- (10,4)node[ddot] {} -- (13,0) node[dot] {};
\draw[line width=1pt] (0, 4)node[ddot, label=above:$r_1$]{} 
.. controls(4,6) and (6,6) ..
(10, 4)node[ddot, label=above:$r_2$]{}
(5, 12) node[label ] {$\underline{j = 3}$};
 }

\tikzstyle{dot1} = [ draw=  gray!00, 
 rectangle, rounded corners, fill=gray!00,  inner sep=1pt, inner ysep=3pt]

\tikzstyle{dot2} = [ draw=  black, 
ellipse, fill=gray!00,  inner sep=1pt, inner ysep=3pt]

\tikzstyle{dot3} = [ draw=  gray!00, 
ellipse, fill=gray!00,  inner sep=1pt, inner ysep=3pt]

\DeclareSymbol{T4}{0}
 {\draw (0,0) node[dot1]{$\phantom{n_2^{(1)} = n^{(2)}}$} -- (0,15) node[dot1] {$\phantom{n^{(1)}}$}; 
 \draw (-13,0) node[dot1] {$n_1^{(1)}$} -- (0,15)node[ddot] {$\phantom{n^{(1)}}$} 
 -- (13,0) node[dot1] {$n_3^{(1)}$};
 \draw (0,-15) node[dot1]{$n_2^{(2)}$} -- (0,0) node[dot1] {$\phantom{n_2^{(1)} = n^{(2)}}$}; 
 \draw (-13,-15) node[dot1] {$n_1^{(2)}$} -- (0,0)node[dot1] {$n_2^{(1)} = n^{(2)}$} -- (13,-15) node[dot1] {$n_3^{(2)}$};
\draw (40,0) node[dot1]{{$n_2^{(3)}$}} -- (40,15) node[dot3]  {$\phantom{n^{(1)} = n^{(3)}}$} ; 
 \draw (27,0) node[dot1] {$n_1^{(3)}$} -- (40,15)node[dot3]  {$\phantom{n^{(1)} = n^{(3)}}$}  -- (53,0) node[dot1] {$n_3^{(3)}$};
\draw[line width=1pt] (0, 15)node[ddot, label=above:$r_1$]{$n^{(1)}$} 
.. controls(10,23) and (27,23) ..
(40, 15)node[dot2, label=above:$r_2$] {$n^{(1)} = n^{(3)}$} ;
 }

\makeatletter
\def\DeclareSymbol#1#2#3{\expandafter\gdef\csname MH@symb@#1\endcsname{\tikz[baseline=#2,scale=0.15]{#3}}}
\def\<#1>{\csname MH@symb@#1\endcsname}
\makeatother

\newcommand{\hao}[1]{\marginpar{\color{red} $\Leftarrow\Leftarrow\Leftarrow$}{\smallskip \noi\color{blue}Hao: #1}\smallskip}

\newcommand{\kihoon}[1]{\marginpar{\color{red} $\Leftarrow\Leftarrow\Leftarrow$}{\smallskip \noi\color{brown}Kihoon: #1}\smallskip}

\begin{document}
\baselineskip = 14pt

\title[sine–Gordon measure and multi-solitons]
{Concentration and fluctuations of sine–Gordon measure around topological multi-soliton manifold}

%Central limit theorem for sine–Gordon measures around multi-topological solitons

\author[K.~Seong, H.~Shen, and P.~Sosoe]
{Kihoon Seong, Hao Shen, and Philippe Sosoe}

%and Simons Laufer Mathematical Sciences Institute (SLMath)\\
%17 Gauss Way\\
%Berkeley, CA 94720\\
%USA

\address{Kihoon Seong\\
École polytechnique fédérale de Lausanne (EPFL)\\
Rte Cantonale, 1015 Lausanne\\
Switzerland}

\email{kihoon.seong@epfl.ch}

\address{Hao Shen\\ Department of Mathematics\\
University of Wisconsin–Madison\\
480 Lincoln Drive\\
Madison, WI 53706\\
USA}

\email{pkushenhao@gmail.com}

\address{Philippe Sosoe\\
Department of Mathematics\\
Cornell University\\ 
310 Malott Hall\\ 
Cornell University\\
Ithaca\\ New York 14853\\ 
USA }

\email{ps934@cornell.edu}



%the Gibbs measure defined on each class still concentrates around the multi-soliton manifold (superpositions of topological solitons) and exhibits Ornstein–Uhlenbeck fluctuations near 

\begin{abstract}

%Each topological degree corresponds to a family of topological solitons, known as kinks and antikinks. 

We study the sine–Gordon measure defined on each homotopy class. The energy space decomposes into infinitely many such classes indexed by the topological degree $Q \in \Z$.  Even though the sine–Gordon action admits no minimizer in homotopy classes with $|Q| \ge 2$, we prove that 
the Gibbs measure on each class nevertheless concentrates and exhibits Ornstein–Uhlenbeck fluctuations near the multi-soliton manifold in the joint low-temperature and infinite-volume limit.  
Moreover, we show that soliton collisions are unlikely events, so that typical states consist of solitons separated at an appropriate scale.
Finally, we identify the joint distribution of the multi-soliton centers as the ordered statistics of independent uniform random variables, so that each soliton’s location follows a Beta distribution.

%Finally, we show that the joint distribution of the multi-soliton centers coincides with the ordered statistics of independent uniform random variables, so that each soliton’s location follows a Beta distribution. 

%Moreover, multi-soliton configurations with colliding solitons form a large deviation event; in particular, typical configurations consist of well-separated solitons.

%On this manifold where the measure concentrates, we also precisely characterize the distance scale between solitons, which determines the typical geometry of multi-soliton configurations.

%In particular, we determine an explicit minimal distance scale between solitons, which guarantees both concentration and a central limit theorem for the Gibbs measure.

%As a consequence, the expected positions of the topological solitons are evenly spaced along the spatial domain, while the gaps between them follow a Dirichlet distribution.

%As a consequence, the expected positions of the topological solitons become evenly spaced along the spatial domain.

\end{abstract}

\maketitle

\tableofcontents

\setlength{\parindent}{0mm}
\setlength{\parskip}{6pt}

\section{Introduction}

%Sine–Gordon measures on homotopy classes

\subsection{Motivations and implications of the main results}

The sine–Gordon model plays a central role as a fundamental example of a nonlinear scalar field theory admitting {\it topological solitons}. In this paper, we specifically study the massless sine–Gordon field theory, with action
\begin{align}
E(\phi)=\frac 12 \int_{\R} |\dx \phi|^2 dx +\int_{\R} (1-\cos \phi ) dx.
\label{Ham1}
\end{align}

%\noi
%which arises as a canonical model in quantum field theory and statistical mechanics.

This field theory allows the energy space to be classified into infinitely many disjoint homotopy classes according to the topological invariant $Q$, defined in \eqref{topd}.   For any finite-energy configuration $E(\phi)<\infty$, the field must satisfy
\begin{align*}
\phi(\infty):=\lim_{x \to \infty} \phi(x)   \in 2\pi \Z, \qquad \phi(-\infty):= \lim_{x \to -\infty} \phi(x)   \in 2\pi \Z.  
\end{align*}

\noi
These boundary conditions imply that the map $x \mapsto e^{i \phi(x)}$ winds around the target circle an integer number of times as $x$ runs from $-\infty$ to $\infty$. This integer $Q$ defines the topological degree/charge or winding number 
\begin{align}
Q=\frac{1}{2\pi} (\phi(\infty)-\phi(-\infty)) \in \Z.
\label{topd}
\end{align}

\noi
According to the winding number $Q$, the energy space $\mathcal{C}$ decomposes into disjoint connected components, referred to as homotopy classes or topological sectors $\mathcal{C}_Q$
\begin{align*}
\mathcal{C}:=\{\phi \in H_{\textup{loc}}^1(\R): E(\phi)<\infty\}
=\bigsqcup_{Q \in \Z} \mathcal{C}_{Q},
\end{align*}

\noi
where
\begin{align}
\mathcal{C}_{Q}=\{ \phi \in \mathcal{C}: (\phi(\infty)-\phi(-\infty) )/ 2\pi =Q  \}.
\label{HT1}
\end{align}

%Any two finite-energy configurations with the same topological degree $Q$ can be joined by a continuous path within the sector $\mathcal{C}_Q$.

%because this would require changing the boundary values at infinity and would therefore cost infinite energy.

\noi
Within each sector $\mathcal{C}_Q$, fields can be continuously deformed into one another. However, a configuration in $\mathcal{C}_Q$ cannot be continuously deformed into a configuration in $\mathcal{C}_{Q'}$ when $Q\neq Q'$. Therefore, when studying the minimization problem for the energy functional $E(\phi)$, we fix a topological sector $\mathcal{C}_Q$ and consider 
\begin{align*}
\inf_{\phi \in \mathcal{C}_Q} E(\phi).
\end{align*}

\noi 
The following facts are well known:
\begin{itemize}
\item[(i)]  $Q=0$ (vacuum sector): the minimizers are the vacuum states
\begin{align*}
\phi(x) =2 \pi k, \quad k\in \Z.
\end{align*}

\medskip 

\item[(ii)] $|Q|=1$ (kink/antikink sector): the minimizers are the kinks when $Q=1$ and the anti-kinks when $Q=-1$, unique up to translation symmetry
\begin{align*}
\{ m(\cdot -\xi) \}_{\xi \in \R} \quad \textup{and} \quad \{ m^{-}(\cdot -\xi) \}_{\xi \in \R},
\end{align*}

\noi
where 
\begin{align*}
m(x)=4 \arctan(e^{ x }) \quad \textup{and} \quad m^{-}(x)=4 \arctan(e^{ -x }).
\end{align*}

\noi
Thus, the family of minimizers forms a one-dimensional soliton manifold.
These topological solitons, kinks and antikinks, interpolate between the vacua $0$ and $2\pi$.

%Notice that the topological solitons  are the kinks ($Q = 1$, that is, $m^{+}$) and antikinks ($Q = -1$, that is, $m^{-}$), which interpolate between different vacua $0$ and $2\pi$.

%A topological soliton is a stable, localized solution whose existence and stability are guaranteed by a topological invariant (e.g., winding number,

\medskip 

\item[(iii)] $|Q|\ge 2$ (higher-charge sectors): no minimizer exists. For higher charge, the lack of compactness prevents the existence of a minimizer: the energy infimum is approached only by a ``runaway'' configuration of $|Q|$ widely separated kinks/antikinks.

%Configurations whose net interpolation between vacua is $2\pi k$ for $Q=k>0$ or $- 2 \pi k$ for $Q=-k$.

%See Theorem~\ref{THM1} and Subsection \ref{SUBSEC:motiv} for details.
 
\end{itemize}

%\footnote{Here $Z_\eps^Q$ denotes partition functions that may differ from line to line.} 

Although the higher-charge sectors $|Q| \ge 2$ admit no minimizer, we nevertheless study, for each $Q\in \Z$, the concentration and fluctuation behavior of the Gibbs measure $\rho_\eps^Q$ on the homotopy class $\mathcal{C}_Q$ 
\begin{align*}
\rho_\eps^Q(d\phi)= (Z_{\eps}^Q)^{-1} \exp\Big\{-\frac 1\eps \int_{-L_\eps}^{L_\eps} ( 1-\cos \phi(x) ) dx -\frac 1{2\eps} \int_{-L_\eps}^{L_\eps} | \dx \phi|^2  dx   \Big\} \prod_{x \in  [-L_\eps, L_\eps] } d\phi(x)
\end{align*}

\noi
in the joint low-temperature $\eps \to 0$ and infinite-volume $L_\eps \to \infty$ limits. For the precise definition of the Gibbs measure $\rho_\eps^Q$ on each homotopy class $\mathcal{C}_Q$, see the next subsection.

%although Noting in particular that no minimizer exists in the higher-charge sectors $|Q| \ge 2$.

%around the multi-soliton manifold
%\begin{align*}
%\bigg\{ \sum_{j=1}^{|Q|} m^\s(\cdot-\xi_j): \xi_j \in \R, \; \s=\pm 1 \bigg\},
%\end{align*}

\noi

%We also study the joint and marginal distributions of the soliton centers $(\xi_1,\dots, \xi_{|Q|})$, which describe where the topological solitons are likely to appear in the limits. 

%\textcolor{blue}{Topological invariant,  A topological soliton is a stable solution whose existence and stability are guaranteed by a topological invariant. }

\begin{comment}
\begin{align*}
\{ Q_{\xi_1,\dots, \xi_k}^{+}  \}_{\xi_1,\dots, \xi_k \in \R }, \qquad  \{ Q_{\xi_1,\dots, \xi_k}^{-}  \}_{\xi_1,\dots, \xi_k \in \R },
\end{align*}

\noi
Here, $Q^{+}_{\xi_1,\dots, \xi_k }$ denotes the superposition of kinks for $\xi_1,\dots,\xi_k \in \R$
\begin{align*}
Q^{+}_{\xi_1,\dots, \xi_k }(x)=Q^{+}(x-\xi_1)+Q^{+}(x-\xi_2)+\dots +Q^{+}(x-\xi_k),
\end{align*}

\noi
and $Q^{-}_{\xi_1,\dots, \xi_k }$ denotes the superposition of antikinks 
\begin{align*}
Q^{-}_{\xi_1,\dots, \xi_k }(x)=Q^{-}(x-\xi_1)+Q^{-}(x-\xi_2)+\dots +Q^{-}(x-\xi_k).
\end{align*}
\end{comment}

We first state our main results in a somewhat informal manner; see Theorems \ref{THM:1}, \ref{THM:2}, and \ref{THM:3} for the precise statements. In the following, we only consider the nontrivial topological sector $Q \neq 0$, where solitons appear.

\begin{theorem}
Let $Q \in \Z$ with $Q \neq 0$. 

\begin{itemize}
\item[(1)] Under the ensemble $\rho^Q_\eps$ with $Q>0$, the field $\phi$ exhibits the typical behavior
\begin{align*}
\phi(x) \approx \sum_{j=1}^{Q} m(x-\xi_j) +\eps^{\frac 12} \cdot \textup{Ornstein–Uhlenbeck}
\end{align*}

\noi
with 
\begin{align*}
\min_{ i \neq j} |\xi_j-\xi_i| \ge \big|\log (\eps \log \tfrac 1\eps) \big| \to \infty
\end{align*}

\noi 
as $\eps \to 0$ and $L_\eps \to \infty$. When $Q<0$ the soliton $m$ is replaced by $m^{-}$.

%Here, $\s=+$ when $Q>0$ and $\s=-$ when $Q<0$. 

\medskip 

\item[(2)] 
As $\eps \to 0$, the joint distribution of the centers $(\xi_1,\dots, \xi_{|Q|} )$ is the ordered statistics of $|Q|$ independent uniform random variables. In particular, each marginal $\xi_{(j)}$ has a Beta distribution. Consequently, the expected position and the expected gap are 
\begin{align*}
\E[\xi_{(j)}]\approx-L_\eps+\frac{2L_\eps j}{|Q|+1}, \qquad  \E[\xi_{(j)}-\xi_{(j-1)}]\approx \frac{2L_\eps}{|Q|+1},
\end{align*}

\noi 
where $\xi_{(j)}$ denotes the $j$-th ordered center in increasing rearrangement $\xi_{(1)}\le \cdots \le \xi_{(|Q|)}$.

\end{itemize}

\end{theorem}

The main results have the following implications:

\begin{itemize}
\item[(i)] The main theorem provides the first result on the concentration and fluctuation behavior of Gibbs measures around multi-solitons $\sum_{j=1}^Q m(\cdot-\xi_j)$, rather than a single soliton. In particular, our result shows that even though the higher-charge sector $\mathcal{C}_Q$, $|Q| \ge 2$, does not contain minimizers of the action on $\mathbb{R}$, the Gibbs measure over $\mathcal{C}_Q$ still exhibits concentration and fluctuation around multi-soliton configurations.

%For references on the behavior of Gibbs measures around a single soliton, see Subsection. 

\medskip 

\item[(ii)] 
At leading order, typical configurations under the Gibbs ensemble consist of exactly $|Q|$ solitons whose mutual separations are of order $\big|\log (\eps \log \tfrac 1\eps)\big|$.  Consequently, configurations in which the solitons collide are unlikely, and the solitons behave as effectively non-interacting objects.

%and configurations involving their interaction correspond to large deviation events with exponentially small probability.

\medskip 

\item[(iii)] Our base measure is the Brownian bridge \eqref{BB} without a mass term, which lacks correlation decay, whereas the fluctuation measure is the Ornstein–Uhlenbeck measure, exhibiting strong correlation decay. This contrast is rather striking, since in many quantum field and statistical physics models, the fluctuations are typically governed by the underlying base field, rather than having a completely different covariance structure.

\medskip 

\item[(iv)]  
The expected soliton centers $\xi_1,\dots,\xi_{|Q|}$ are evenly spaced, dividing the interval $[-L_\eps, L_\eps]$ into $|Q| + 1$ equal parts of length $\tfrac{2L_\eps}{|Q| + 1}$. Furthermore, each individual center $\xi_j$ exhibits a Beta-type fluctuation around its expected position.

\end{itemize}

%around the multi-soliton manifold.

\begin{figure}[h]
\includegraphics[scale=0.7]{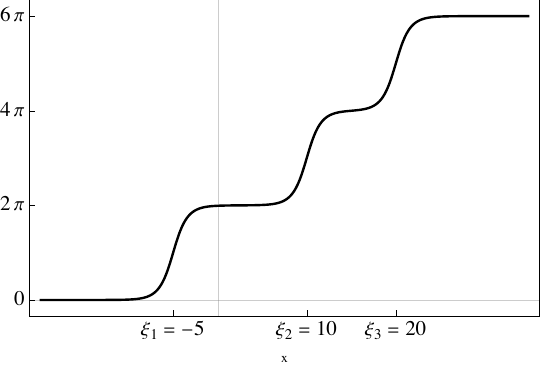}
\caption{a multi-soliton  $\sum_{j=1}^Q m(\cdot-\xi_j)$ with $Q=3$ and $(\xi_1,\xi_2,\xi_3)=(-5,10,20)$}
\end{figure}

The remarkable point is that in previous works (see Subsections \ref{SUBSUBSEC:NLSC}, \ref{SUBSUBSEC:NLSd}, \ref{SUBSUBSEC:ALLCAHN}, and \ref{SUBSUBSEC:TOPSOL}), the behavior of Gibbs measures was mainly studied around a single soliton, and most of the analysis in those works focused primarily on the concentration of the measure. On the other hand, our result is the first study of the Gibbs measure around \textit{multi-soliton configurations}, involving not only concentration but also a central limit theorem behavior around the multi-soliton manifold.  Furthermore, we provide a concrete description of the soliton locations and gaps.

Moreover, the geometry of the multi-soliton manifold $\{\sum_{j=1}^Q m(\cdot-\xi_j): \xi_j \in \R \}$ is a central object of interest in this work. In contrast to the single-soliton manifold $\{m(\cdot-\xi): \xi \in \R\}$, the multi-soliton manifold fails to be differentiable and becomes merely a topological manifold when solitons collide, that is, when $|\xi_i-\xi_j| \les 1$. As a result, fundamental geometric objects such as tangent and normal vectors are no longer well defined in the collision regime. This geometric degeneracy requires a careful analysis and a delicate decomposition of the multi-soliton manifold. See Remark~\ref{REM:strata}.

%In particular, understanding the structure of the collision manifold $\{\sum_{j=1}^Q m(\cdot-\xi_j): \xi_j \in \R \; \textup{and} \; \min_{i \neq j}|\xi_i-\xi_j|<d \}$ is of interest from a geometric perspective.

We finally remark that, unlike in the one-dimensional setting where topological solitons are well understood (see Subsection \ref{SUBSUBSEC:SG}), the situation in two dimensions is quite different.
The sine–Gordon equation in 2D also admits soliton-like solutions, often called kink walls, obtained by extending the one-dimensional kink uniformly in another spatial direction.
These configurations solve the equation but necessarily have infinite energy, and thus are not finite-energy solitons in the usual sense. To the best of our knowledge, such infinite-energy kink-wall solutions are far less understood and have not been studied as systematically as their one-dimensional counterparts.

%We remark that, unlike in the one–dimensional setting where topological solitons are well understood (see Subsection \ref{SUBSUBSEC:SG}), it is not known whether the two–dimensional sine–Gordon energy supports finite energy soliton solutions. In particular, the existence (or non–existence) of topological solitons in 2D remains unknown.

%under the Gibbs ensemble.

%The main theorem reveals a generic statistical behavior of solutions to the sine–Gordon equation. This is consistent with the result that global solutions to  asymptotically decompose into multi-solitons and a small radiation term \cite{}. The key difference from the deterministic setting is that the radiation part is characterized by a Gaussian field rather than by a linear profile.

\subsection{Main results}

%Energy and entropy competition explanation about $L_\eps$.

In this subsection, we present the three main theorems, \ref{THM:1}, \ref{THM:2}, and \ref{THM:3}. Before stating the theorems, we first study Gibbs measures corresponding to each topological degree $Q \in \Z$.

Based on the definition of the topological degree $Q$ in \eqref{topd}, 
when $\phi(\infty)=2\pi n^{+}$ and $\phi(-\infty)=2 \pi n^{-}$, with $n^{+}, n^{-} \in \Z$, the homotopy class $\mathcal{C}_Q$ depends only on the difference 
\begin{align*}
Q=n_{+}-n_{-}, 
\end{align*}

%For example, for any $\phi$ with limits $\phi(-\infty)=2\pi n_{-}$, $\phi(\infty)=2 \pi n_{+}$, define $\wt \phi(x):=\phi(x)-2\pi n_{-}$.  Then $\wt \phi$ has limits $0$ and $2 \pi (n_{+}-n_{-} )$. 

\noi
not on the individual values of $n_{+}$ and $n_{-}$. Therefore, when describing each sector $\mathcal{C}_Q$ with $Q>0$, we fix a representative in the equivalence class by choosing the left boundary value $\phi(-L_\eps)=0$  as the base point, so that $\phi(L_\eps)=2\pi Q$, where $L_\eps \to \infty$. A similar convention applies for $Q<0$ by reversing the orientation.   By symmetry, we only consider the case $Q>0$ throughout the paper, unless specified otherwise.

%Notice that $2\pi$-shifts also leave the energy invariant, $E(\phi+2\pi k)=E(\phi)$. This means that we can freely shift $\phi$ by $2\pi k$, $k \in \Z$, without affecting the energy, and only the difference $Q$ has physical/topological meaning.

%Notice that two configurations with boundary conditions satisfying $2\pi(n_{+}-n_{-})=2\pi (n'_{+}-n'_{-})=2\pi Q$ belong to the same topological sector $\mathcal{C}_Q$. 

We now introduce the base measure, namely the Brownian bridge,
\begin{align}
\mu_\eps^Q(d\phi)=\frac{1}{Z_\eps^{Q,\textup{BB}}} \exp \Big\{-\frac 1\eps \int_{-L_\eps}^{L_\eps} |\dx \phi|^2 dx \Big\} \prod_{x \in [-L_\eps, L_\eps ] }d\phi(x),
\label{BB}
\end{align}

\noi
which is the Gaussian measure conditioned on $\phi(-L_\eps)=0$ and $\phi(L_\eps)=2\pi Q$. This Gaussian measure describes fluctuations around the affine line connecting the boundary values $0$ and $2\pi Q$. See Subsection \ref{SUBSEC:BB}. This choice of base measure pins down a representative within the equivalence class $\mathcal{C}_Q$ and yields a unique Brownian bridge measure $\mu^Q_{\eps}$.

%For the explicit expression of the measure, see Subsection \ref{SUBSEC:BB}.

%The Gaussian measure describes fluctuations around the affine line connecting the boundary values $0$ and $2\pi Q$. 

\noi 
For each topological degree $Q \in \Z$, we now define the Gibbs measure, using the Brownian bridge 
\begin{align}
\rho_\eps^Q(d\phi)=(Z_{\eps}^Q)^{-1} \exp\Big\{-\frac 1\eps \int_{-L_\eps}^{L_\eps} ( 1-\cos \phi(x) ) dx \Big\} \mu_\eps^Q (d\phi).
\label{Gibbs1} 
\end{align}

\noi
In the following, any field $\phi$ distributed according to the Gibbs measure $\rho_\eps^Q$ is viewed as a function on $\R$, extended trivially by $0$ and $2\pi |Q|$ outside the interval $[-L_\eps, L_\eps]$. We now state, for each $Q\in \Z$, how the Gibbs measure $\rho^Q_{\eps}$, associated with the homotopy class $\mathcal{C}_Q$, concentrates around the multi-soliton manifold.

%Before stating the first theorem, we introduce the approximating soliton $m^\eps(\cdot-\xi_j)$ on the interval $[-L_\eps, L_\eps]$, defined in \eqref{MGS1}, with extensions $0$ and $2\pi |Q|$ outside $[-L_\eps, L_\eps]$. As $\eps \to 0$, $m^\eps(\cdot-\xi_j)$ becomes a more and more precise approximation of the topological soliton $m(\cdot-\xi_j)$ on $\R$ (see \eqref{APM0}).   

%In this way, the measure is interpreted accordingly.

%Every random function distributed according to $\mu^Q_\eps$ can be considered as a function in $m_{\xi_1,\dots, \xi_{|Q|}} +L^2(\R) $ by trivial extension 

%For $Q<0$, we define the corresponding multi-soliton manifold with anti-kinks. 

\begin{theorem}\label{THM:1}
Let $Q \in \Z$ with $Q \neq 0$, and $L_\eps =\eps^{-\frac 12+\eta}$ with $\eta>0$ arbitrarily small but fixed.

\begin{itemize}
\item[(i)] There exists $c>0$ such that for any $\dl>0$ 
\begin{align}
\limsup_{\eps \to 0} \eps \log \rho_\eps^Q  \big(\{   \textup{dist}(\phi, \M_Q ) \ge \dl   \} \big) \le -c \dl^2,
\label{ldp1} 
\end{align}

%\begin{align*}
%\rho_\eps^Q  \big(\{   \textup{dist}(\phi, \M_Q) \ge \dl   \} \big) \les \exp\Big\{ - \frac{c \dl^2}{\eps } \Big\},
%\end{align*}

%m_{\xi_1,\dots,\xi_Q}

\noi
where $\textup{dist}$ denotes the $L^2(\R)$-distance and the multi-soliton manifold\footnote{The multi-soliton profile $m(\cdot-\xi_1)+\dots +m(\cdot-\xi_Q)$ is invariant under permutation of the labels $i=1,\ldots, |Q|$. Each unordered configuration corresponds to $Q!$ identical ordered configurations. Because of the indistinguishable nature, we work with the ordered set.} is defined as
\begin{align}
\M_Q:=\Big\{\sum_{j=1}^Q m(\cdot-\xi_j): -\infty<\xi_1\le \dots \le \xi_Q<\infty  \Big\}, \quad Q>0.
\label{man0}
\end{align}

\noi 
For $Q<0$, we define the corresponding multi-soliton manifold with anti-kinks.

\medskip

\item[(ii)] There exists $c>0$ such that  
\begin{align} 
\limsup_{\eps \to 0} \eps \log \rho_\eps^Q  \big(\{   \textup{dist}(\phi, \M_Q^{<d} ) < \dl  \big) \le - c e^{- d},
\label{ldp2}
\end{align}

%$\pi^\eps$ denotes the projection\footnote{If $\textup{dist}(\phi, \M_Q^\eps)<\dl$ with $\dl$ sufficiently small, then there exists a unique $(\xi_1,\dots,\xi_Q)$ such that $\| \phi-\sum_{j=1}^Qm^\eps(\cdot-\xi_j)\|_{L^2(\R)}$ is minimized. Accordingly, we define the projection $\pi^\eps(\phi)$  as the closest point on  $\M_Q^\eps$. } onto $\M_Q^\eps$, defined in \eqref{APM5}, and  

\noi
for any $\dl>0$, $d>0$ satisfying\footnote{ Later, we prove that the collision manifold $\M_Q^{<d}$ is an unlikely event. Therefore, when the field $\phi$ is sufficiently close to $\M_Q^{<d}$, in the sense that $ c e^{-d } \ge \dl^2$, we obtain the same result. See Lemmas~\ref{LEM:GAP2} and~\ref{lem:gap-multikink}.} $ c  e^{-d } \ge \dl^2$, where the collision manifold $\M_Q^{<d}$, $Q>0$, is defined as 
\begin{align}
\M^{<d}_{Q}:=\Big\{\sum_{j=1}^Q m(\cdot-\xi_j) : -\infty<\xi_1\le \dots \le \xi_{Q}<\infty \; \; \textup{and} \; \; \min_{i \neq j} |\xi_i-\xi_j|<d \Big\}.
\label{man1}
\end{align}

\noi
For $Q<0$, we define the corresponding manifold with anti-kinks.

%\noi
%For $Q<0$, we define the corresponding multi-soliton manifold with anti-kinks. 

\end{itemize}

\end{theorem}

The first part of Theorem \ref{THM:1} shows that when a field $\phi$ is far from the multi-soliton manifold $\M_Q$, the event is exponentially unlikely with rate $\dl^2$. In particular, the manifold consists of exactly $|Q|$ solitons, without any mixture of kinks and anti-kinks.
Furthermore, the second part shows that although a field $\phi$ is close to the multi-soliton manifold, when the solitons collide so that
$\min_{i \neq j}|\xi_i-\xi_j|<d$, the collision region becomes a large-deviation event.
From \eqref{ldp1}, \eqref{ldp2}, and the condition $c e^{-d}\ge \dl^2$, we may choose the distance and collision scales as $\dl_\eps=\eta\sqrt{\eps \log \frac 1\eps}$ and $d_\eps = \big|\log (\eps \log \frac 1\eps) \big|$. With these choices, we obtain
\begin{align*}
&\rho_\eps^Q  \big(\{   \textup{dist}(\phi, \M_Q ) \ge \dl_\eps   \} \big) \le e^{-c \log \frac 1\eps}\\
&\rho_\eps^Q  \big(\{   \textup{dist}(\phi, \M_Q^{<d_\eps} ) < \dl_\eps   \} \big) \le e^{-c \log \frac 1\eps }
\end{align*}

\noi
for some $c>0$ as $\eps \to 0$. An interesting fact is that we identify the collision scale $\big|\log (\eps \log \frac 1\eps) \big|$ that determines the typical behavior of solitons. Hence, most of the probability mass is concentrated in the well-separated (non-collision) region
\begin{align*}
\{   \textup{dist}(\phi, \M_Q^{\ge d_\eps} ) < \dl_\eps   \},
\end{align*}

\noi
where the non-collision manifold $
\M^{ \ge d_\eps}_Q$ is defined as 
\begin{align*}
\M^{ \ge d_\eps}_{Q}:=\Big\{\sum_{j=1}^Q m(\cdot-\xi_j) : -\infty<\xi_1\le \dots \le \xi_{Q}<\infty \; \; \textup{and} \; \; \min_{i \neq j} |\xi_i-\xi_j|\ge d_{\eps} \Big\}.
\end{align*}

\noi 
In the proof of Theorem \ref{THM:1}, the key aspect is to understand how the solitons interact with each other on the collision scale $d_\eps=\big|\log (\eps \log \tfrac 1\eps)\big|$, and how the energy behaves even though these configurations are not minimizers.

\begin{remark}\rm 
In Theorem \ref{THM:1}, there is a competition between the vanishing energy scale $\eps \to 0$ and the entropic effects arising from the growing interval $L_\eps=\eps^{-\frac 12+\eta} \to \infty$.  Under our method, the interval size $L_\eps=\eps^{-\frac 12+\eta}$ is optimal. See %Remark \ref{REM:scale-int}
Section \ref{SEC:LDP}
for an explanation of this scaling. 
\end{remark}

As a consequence of Theorem \ref{THM:1}, under the measure $\rho^Q_\eps$,
the leading-order behavior is described by $\sum_{j=1}^{Q}m(\cdot-\xi_j)$
with almost no collision $\min_{i\neq j}|\xi_i-\xi_j|\ge |\log(\eps \log \tfrac 1\eps)| \to \infty$. In the following theorem, we investigate the next-order fluctuation behavior around the multi-solitons. To state our next theorem, we first introduce $\pi^\eps$, the projection onto the (approximating)\footnote{Since we work on the finite volume $[-L_\eps, L_\eps]$, we need to define an approximating multi–soliton profile
$\sum_{j=1}^Q m^\eps(\cdot-\xi_j)$ so that the entire transition from $0$ to $2\pi Q$ occurs inside this interval. As $\eps \to 0$, $m^\eps(\cdot-\xi_j) $ becomes a more and more precise approximation of the topological soliton $m(\cdot-\xi_j)$  on $\R$ (see \eqref{APM0}) } multi-soliton manifold $\M_Q^{\eps, \ge d_\eps}$ defined in \eqref{APM5}, where the measure $\rho^Q_\eps$ concentrates (see Lemma \ref{LEM:partition}).

%To state our next theorem, we first introduce $\pi^\eps$, the projection onto the (approximating) multi-soliton manifold defined in \eqref{APM5}, where the measure $\rho_\eps^Q$ concentrates (see Remark~\ref{REM:LDP}).

%$\pi^\eps$ is the projection onto $\M_Q^\eps$, defined in \eqref{APM5}, and 

\begin{theorem}\label{THM:2}

Let $Q \in \Z$ with $Q\neq 0$, $L_\eps=\eps^{-\frac 12+\eta}$, and let $F$ be a bounded and continuous function. Then 
\begin{align*}
\lim_{\eps \to 0} \int F( \sqrt{\eps}^{-1}(\phi- \pi^\eps(\phi) ) ) \rho_\eps^Q (d\phi)=\int F(\phi) \mu_{\text{OU}} (d\phi), 
\end{align*}

\noi
where $\mu_{\text{OU}}$ is the Ornstein–Uhlenbeck measure 
\begin{align*} 
\mu_{\text{OU}} (d\phi)=Z^{-1}\exp\Big\{ -\frac 12 \jb{ \phi, (-\dx^2+1)\phi }_{L^2(\R)}  \Big\} \prod_{x \in \R} d\phi(x).
\end{align*}

\noi 

%Here,  with the separation condition
%\begin{align*}
%\{ m_{\xi_1,\dots, \xi_k}^{\s} : \xi_1,\dots, \xi_k \in \R \quad \textup{and} \quad \min_{ i \neq j} |\xi_j-\xi_i| \ge d_{\min}^\eps    \},
%\end{align*}

\end{theorem}

The fluctuations described in Theorem \ref{THM:2} exhibit behavior that is different from the classical result of Ellis–Rosen \cite[Theorem 4]{ER2}, where Ellis–Rosen studied the central limit theorem for general Gibbs measures in the low–temperature limit. In the usual low-temperature setting $\eps \to 0$,  the fluctuation behavior is determined by the second variation $\nb^2 E$ of the energy evaluated at the minimizers. In contrast, Theorem \ref{THM:2} concerns a joint limit in which $\eps \to 0$ and $L_\eps \to \infty$ simultaneously.
This introduces a competition between energy and entropic effects, leading to a fluctuation behavior that differs markedly from that of Ellis–Rosen \cite[Theorem 4]{ER2}. Furthermore, in our case, the energy $E$ does not possess minimizers in the homotopy classes with $|Q| \ge 2$. Consequently, the second–variation approach used in Ellis–Rosen \cite[Theorem 4]{ER2} cannot be applied here.

%As mentioned above, this fluctuation behavior is interesting. Although the base measure is the Brownian bridge \eqref{BB}, which has no decay of correlations, the fluctuations are governed by an Ornstein–Uhlenbeck law with strong correlation decay.  This is in sharp contrast with many quantum–field and statistical models, where the fluctuations are usually induced by the base field, rather than by a measure with a different covariance. 

To prove Theorem \ref{THM:2}, although a minimizer does not exist when $|Q| \ge 2$, we analyze the second variation of the energy at the multi–soliton configuration $m_{\xi_1,\dots,\xi_Q}=\sum_{j=1}^Q m(\cdot-\xi_j)$
\begin{align*}
\nb^2 E(m_{\xi_1,\dots,\xi_{Q}})=-\dx^2+\cos(m_{\xi_1,\dots,\xi_Q})
%-\dx^2+1-2\sum_{j=1}^{Q} \text{sech}^2(\cdot-\xi_j)+O(e^{-\min_{i\neq j} |\xi_i-\xi_j| })
\end{align*}

\noi
under the separation scale $\min_{i \neq j} |\xi_i-\xi_j|\ge \big|\log (\eps \log \frac 1\eps)\big|$ (see also \eqref{2ndvar}). It allows to study the Gaussian measure
$\exp\big\{-\frac 12 \jb{ \nb^2 E(m_{\xi_1,\dots,\xi_Q} )v,v  } \big\}$ whose covariance structure (Lemma \ref{LEM:cov}) and correlation decay (Proposition \ref{PROP:dec}) are crucial ingredients in the proof of the central limit theorem.

In particular, the spectral analysis of $\nb^2 E(m_{\xi_1,\dots,\xi_{Q}})$  is closely linked to an understanding of the geometry of the multi-soliton manifold $\M_Q$ through quadratic forms such as 
\begin{align*}
\jb{ \nb^2 E(m_{\xi_1,\dots,\xi_Q} )v,v  }, 
\end{align*}

\noi 
where  $v$ is taken in either the tangential or the normal direction to the manifold $\M_Q$. As discussed in Remark~\ref{REM:strata}, the multi-soliton manifold $\M_Q$ and, in particular, the collision manifold $\M_Q^{<d}$ fail to be differentiable and are only a topological manifold. This lack of smooth geometric structure prevents the use of standard tools such as tangent and normal decompositions, which are essential for performing a second-order expansion and identifying Gaussian fluctuations in Ellis–Rosen \cite[Theorem 4]{ER2}.
To overcome this issue, by proving the large-deviation theorem \ref{THM:2}, we exclude the collision manifold $\M_Q^{<d}$. On the resulting non-collision manifold $\M_Q^{\ge d}$, the manifold is smooth and admits well-defined tangent and normal directions.  This allows us to carry out a geometric decomposition for studying $\nb^2 E(m_{\xi_1,\dots,\xi_{Q}})$.

%and to establish the central limit theorem.

%Note that this is also of independent interest as a study of the linearized operator $\nb^2 E(m_{\xi_1,\dots,\xi_{Q}})$ for the sine–Gordon action at a multi-soliton configuration.

%Crucially, the spectral analysis of $\nb^2 E(m_{\xi_1,\dots,\xi_{Q}})$ is tightly linked to the geometry of the underlying multi-soliton manifold

%a precise understanding of the geometry of the multi-soliton manifold %$\M_Q$ plays a crucial role. As discussed in Remark~\ref{REM:strata}, the multi-soliton manifold—and in particular the collision manifold—fails to be differentiable and is only a topological manifold. This lack of smooth geometric structure prevents the use of standard tools such as tangent and normal decompositions, which are essential for performing a second-order expansion and identifying Gaussian fluctuations.

%Using the large-deviation estimates, we exclude neighborhoods of the collision manifold with high probability. On the resulting non-collision manifold, the multi-soliton manifold is smooth and admits well-defined tangent and normal directions. This allows us to carry out a precise geometric decomposition and to establish the central limit theorem.

\begin{remark}\rm 
In Theorem \ref{THM:2}, the restriction $L_\eps=\eps^{-\frac 12+\eta}$ follows from Theorem~\ref{THM:1}. If one could enlarge the admissible range of $L_\eps$ in Theorem \ref{THM:1}, then the fluctuation result in Theorem \ref{THM:2} continues to hold on a much larger scale.
\end{remark}

We now state the final theorem. The infinite separation condition $|\xi_i-\xi_j|\to \infty$ implies that the interactions between solitons are negligible. However, it does not provide any information about the locations of the solitons. In the following, we analyze the joint and marginal distribution of the soliton locations $(\xi_1,\dots, \xi_{|Q|})$,
which describes their expected positions and the gaps.

%Notice  the multi-soliton profile $m^\eps(\cdot-\xi_1)+\dots +m^\eps(\cdot-\xi_Q)$ is invariant under permutation of the labels $i=1,\ldots, |Q|$. Each unordered configuration corresponds to $Q!$ identical ordered configurations. Because of the indistinguishable nature, we take the set\footnote{when working on the finite volume $[-L_\eps, L_\eps]$, we need to define an approximating multi–soliton profile
%$\sum_{j=1}^Q m^\eps(\cdot-\xi_j)$ so that the entire transition from $0$ to $2\pi Q$ occurs inside this interval. For this reason, we introduce a slightly smaller interval, denoted by $\cj L_\eps \sim L_\eps$, on which the translation $\xi_j \in [-\cj L_\eps, \cj L_\eps]$ modes are defined (see \eqref{cjL}).} $A \in [-\cj L_\eps, \cj L_\eps]^Q$ to be symmetric under permutations: 
%\begin{align}
%\ind_{ \{(\xi_1,\dots,\xi_Q) \in A \} }=\ind_{\{ (\xi_{\s(1)},\dots,\xi_{\s(Q)} )  \in A\} }
%\label{perA}
%\end{align}

%\noi 
%for any permutation $\s$.  In the following, we consider the measure
%$\rho_\eps^Q \{ \pi_\eps^{T}(\phi) \in A \}$. 

Before stating the final theorem, we first present some preliminaries.
Under the coordinate representation\footnote{Thanks to the large deviation results in Theorem \ref{THM:1}, we can write the field $\phi$ as a multi–soliton configuration plus a small perturbation.} $\phi=\sum_{j=1}^Q m^\eps(\cdot-\xi_j)+v$, where $\|v \|_{L^2}<\dl$, defined in \eqref{APM5}, we consider the projection
$\pi_\eps^T(\phi)=(\xi_1,\dots,\xi_Q)$ onto the coordinate variables, where\footnote{when working on the finite volume $[-L_\eps, L_\eps]$, we need to define an approximating multi–soliton profile
$\sum_{j=1}^Q m^\eps(\cdot-\xi_j)$ so that the entire transition from $0$ to $2\pi Q$ occurs inside this interval. For this reason, we introduce a slightly smaller interval, denoted by $\cj L_\eps \sim L_\eps$, on which the translation $\xi_j \in [-\cj L_\eps, \cj L_\eps]$ modes are defined (see \eqref{cjL}).} $-\cj L_\eps \le \xi_1\le \cdots \le \xi_Q \le \cj L_\eps$. Furthermore, we define the (marginal) tangential projection $\pi^T_{j}(\phi)=\xi_{j}$, where $\xi_{j}$ denotes the $j$-th ordered center in the increasing rearrangement $\xi_{1}\le \cdots \le \xi_{Q}$.

\begin{theorem}\label{THM:3}
Let $Q \in \Z$ with $Q\neq 0$, $\cj L_\eps\sim L_\eps=\eps^{-\frac 12+\eta}$.

\begin{itemize}

\item[(i)] The joint distribution of the centers $(\xi_1,\dots, \xi_{|Q|})$ is given by the ordered statistics of $|Q|$ independent uniform random variables on $[-\cj L_\eps, \cj L_\eps]$
\begin{align*}
\rho_\eps^Q \{ \pi_\eps^{T}(\phi) \in A \}
= \frac{\big|A \cap \{-\cj L_\eps \le \xi_1 \le \dots \le \xi_{|Q|} \le \cj L_\eps    \} \big|}{ \big|\{-\cj L_\eps \le \xi_1 \le \dots \le \xi_{|Q|} \le \cj L_\eps    \} \big|  }(1+O(\eps^{0+}) )
\end{align*}

\noi 
as $\eps \to 0$, where $A\subset  \{ -\cj L_\eps \le \xi_1\le \cdots \le \xi_Q \le \cj L_\eps   \}  $ is a measurable subset.

%where $A$ is a measurable subset of $\{ (\xi_1,\dots, \xi_{|Q|} ): \xi_j \in [-L_\eps, L_\eps] \}$. 

\medskip

\item[(ii)] The marginal distribution, that is, each center $\xi_j$, has a Beta-shaped fluctuation
\begin{align*}
\rho_\eps^Q\big\{ \pi^T_{j}(\phi) \in B  \big\}= \int_{B} f_j(x) dx \cdot (1+O(\eps^{0+}))
\end{align*}

\noi
as $\eps \to 0$, where $B\subset [-\cj L_\eps, \cj L_\eps] $ is a measurable subset  and
\begin{align*}
f_j(x)= \frac{|Q|!}{ (2L_\eps)^{|Q|} }\frac{(x+\cj L_\eps)^{j-1}}{(j-1)!} \frac{ (\cj L_\eps-x)^{|Q|-j}  }{(|Q|-j)! }, \; \; -\cj L_\eps<x<\cj L_\eps.
\end{align*}

%where $Y_j$ follows a Beta distribution $\sim \textup{Beta}(j,|Q|+1-j)$. 

\medskip

\item[(iii)] The expected location of each soliton is given by
\begin{align*}
\E_{\rho_\eps^Q } \big[  \pi^T_{j}(\phi)  \big]=\Big(-\cj L_\eps+\frac{2\cj L_\eps j}{|Q|+1} \Big) \cdot (1+O(\eps^{0+}) ).
\end{align*}

\end{itemize}

\end{theorem}

In Theorem \ref{THM:3}, we obtain explicit formulas for the joint and marginal distributions, thereby describing quantitatively how the solitons are arranged over the entire interval. In particular, Theorem \ref{THM:3} shows that the expected centers $\xi_j$ are evenly spaced, dividing the interval $[-\cj L_\eps, \cj L_\eps]$ into $|Q|+1$ equal parts, each of length $\frac{2\cj L_\eps}{|Q|+1}$.

Notice that, in contrast with the soliton resolution result of Chen–Liu–Lu \cite[Theorem~1.1]{CLL20}, which describes the asymptotic decomposition of solutions under deterministic dynamics, our results are probabilistic in nature. Under the Gibbs measure, we identify the typical locations of individual solitons and precisely characterize the gaps between neighboring solitons.

%We mention the works \cite{LS23, AMP23, CLL20, CL} on soliton solutions of the sine–Gordon equation. In particular, \cite[Theorem 1.1]{CLL20} shows that solutions can be written as a superposition of translated topological solitons (kinks and anti-kinks). While \cite{CLL20} establishes the existence of such a superposition, it does not describe the soliton centers or their gaps. In contrast, Theorem \ref{THM:3} provides an explicit characterization of the locations $\xi_j$ and the gaps $\xi_{j}-\xi_{j-1}$ under the Gibbs measure. 

%The leftmost point (the first soliton $m(\cdot-\xi_1)$) tends to be located near the left boundary $-L_\eps$.  The rightmost point (the last soliton $m(\cdot-\xi_{|Q|})$) tends to lie near the right boundary $L_\eps$. The middle solitons are approximately evenly spaced.

\subsection{Related literature}\label{SUBSEC:RLIT}

\subsubsection{Sine–Gordon field theory}\label{SUBSUBSEC:SG}

McKean-Vaninsky \cite{McKV} studied the construction of the one–dimensional sine Gordon measure. More recently, Lacoin–Rhodes–Vargas \cite{LRV23} studied the one–dimensional sine–Gordon measure with a log–correlated base field in the full subcritical regime on a bounded domain.

From the PDE perspective, the one-dimensional sine–Gordon equation and its soliton solutions have been widely studied in recent years. In particular, the asymptotic stability and long–time behavior of soliton solutions have attracted considerable attention. We refer to the works of McKean \cite{McKean}, L\"uhrmann–Schlag \cite{LS23}, Alejo–Muñoz–Palacios \cite{AMP23}, Chen–Liu–Lu \cite{CLL20}, and Chen–Lührmann \cite{CL}.

The two-dimensional sine–Gordon theory has connections to various problems in statistical physics, such as the Coulomb gas and the XY model. We refer to \cite{FS76, BK87, DH93, BB21, BH22, B22,  GS0} for the study of the 2D sine–Gordon measures and to \cite{HS,CHS, BC24} for the two-dimensional sine–Gordon equation with stochastic forcing. 
In two dimensions, the sine–Gordon equation also admits soliton-like solutions, often called kink walls, obtained by extending the one-dimensional kink uniformly in another spatial direction.
Although these configurations solve the 2D sine–Gordon equation, they necessarily have infinite energy and hence fall outside the standard class of finite-energy solitons (see \cite[Chapter 5.4]{MSTOPSOL}).
To the best of our knowledge, these infinite-energy kink-wall solutions are considerably less explored and have not been developed to the same extent as their one-dimensional counterparts.

%These configurations solve the 2D sine–Gordon equation but necessarily have infinite energy, and thus are not finite-energy solitons in the usual sense. To the best of our knowledge, such infinite-energy kink-wall solutions are far less understood and have not been studied as systematically as the classical one-dimensional sine–Gordon solitons.

%Note that as mentioned above, it is not known whether the two-dimensional sine–Gordon energy admits solitons.

%\subsubsection{Gibbs measures around a single-soliton manifold}\label{SUBSUBSEC:soliton}

%The behavior of (invariant) Gibbs measures around a single soliton manifold has been studied for the focusing nonlinear Schr\"odinger (NLS) equation and the stochastic Allen–Cahn equation.

\subsubsection{Continuum focusing NLS Gibbs measure and single-soliton manifold}\label{SUBSUBSEC:NLSC}
The behavior of (invariant) Gibbs measures around a single soliton manifold has been studied for the focusing nonlinear Schr\"odinger (NLS) equation. For the continuum focusing NLS Gibbs measure, McKean \cite{McK1} initiated the study of the infinite-volume limit. Later, Rider \cite{Rider} and Tolomeo–Weber \cite{TW} proved that, on the 1D torus, the measure concentrates around the single-soliton manifold in the infinite-volume limit. In particular, Tolomeo–Weber \cite{TW} identified a critical regime for the strength of the coupling: either the measure strongly concentrates around the single-soliton manifold, or the limiting measure reduces to the underlying Gaussian field. Recently, in \cite{SS24}, the first and third authors proved a central limit theorem for the Gibbs measure around a single-soliton manifold.  In this paper we establish the corresponding result for multi-soliton manifolds in the sine-Gordon setting. The two situations exhibit different fluctuation behavior:  For the focusing NLS Gibbs measure, the fluctuations are white noise near a single soliton, but in the sine–Gordon case studied here, the system exhibits Ornstein–Uhlenbeck fluctuations near the multi-soliton manifold.

%Moreover, compared to the previous work \cite{SS24}, where the first and third authors established a central limit theorem for the Gibbs measure around a single–soliton manifold for the nonlinear Schrödinger equation (NLS), the present paper proves a central limit theorem for the Gibbs measure around a multi–soliton manifold. The two results show fundamentally different fluctuation behavior: in the single–soliton case for the NLS Gibbs measure one obtains white–noise fluctuations, whereas in the multi–soliton case considered here the fluctuations are given by an Ornstein–Uhlenbeck measure.

%in a joint infinite-volume and continuum limit,
%This is possibly because the effect of recession `outruns' the thermodynamic convergence to equilibrium in the infinite volume setting

\subsubsection{Discrete focusing NLS Gibbs measure and single-soliton manifold}\label{SUBSUBSEC:NLSd}
The discretized focusing Gibbs measure and the nonlinear Schrödinger (NLS) equation on a lattice have been studied. When $d \ge 3$, Chatterjee–Kirkpatrick \cite{ChaKirk} initiated the study of the discretized focusing NLS Gibbs measure, identified a critical temperature, and showed that below this threshold the system exhibits striking single-soliton–like behavior. In \cite{Chatt} Chatterjee used microcanonical invariant measures and showed that a typical function in the ensemble decomposes into a “visible” part, which is close to a single soliton, and an “invisible” part that is small in the $L^\infty$ norm.  Notice that, in particular, regarding the reason why a single-soliton profile appears, Chatterjee mentioned in \cite{Chatt} that `` whereas multisoliton solutions eventually merge into a single soliton on the finite discrete torus considered in Theorem 1". In contrast to that situation, in our sine–Gordon model, the boundary conditions imposed by a fixed homotopy class \eqref{HT1} enforce the presence of multiple solitons, and therefore prevent the multi-soliton configuration from collapsing into a single soliton. Again, as emphasized above, to the best of our knowledge, our results are the first study of the concentration and fluctuations of Gibbs measures around multi-soliton manifold. In particular, a crucial aspect is to understand how solitons interact with each other on the collision scale $\big|\log (\eps \log \tfrac 1\eps)\big|$, and how the energy behaves at this scale, even though such configurations are not minimizers.

%\textcolor{red}{
%In other words, the finite-volume nature of the model does not allow solitons to recede to spatial infinity, and this lack of separation forces multi-soliton configurations to collapse into a single soliton profile under the measure. By contrast, in the true infinite-volume regime, such recession dominates any thermodynamic tendency toward collapse. In our setting, where we work on sufficiently large intervals and thus allow well-separated solitons to lie far apart, this separation is accurately captured by the Gibbs measure, which concentrates on genuine multi-soliton configurations.}

Regarding the phase transition of the discrete focusing NLS Gibbs measure, in \cite{DKK}, for $d \ge 3$, Dey-Kirkpatrick-Krishnan identified a phase transition, analogous to the one found by Tolomeo–Weber \cite{TW}, on the lattice. Using two parameters, temperature and the strength of the nonlinearity, they proved the existence of a continuous phase transition curve that divides the parameter plane into two regions, the appearance or non-appearance of (single) solitons.  In the recent work \cite{KR25}, Krishnan and Ray further investigated the model using the two parameters. They proved that the three regions in the phase diagram lead to three distinct limits. A natural question is whether the discrete (or even continuous) sine–Gordon model, in one dimension or higher, also exhibits a phase transition depending on the temperature and the strength of the coupling constant.

%In the weak nonlinear regime, referred to as the dispersive phase, the limiting free energy density coincides with that of a Gaussian field. In the strongly nonlinear regime, the solitonic phase, the free energy acquires a nontrivial contribution from a single soliton, 

%: a massive Gaussian free field; a massless Gaussian free field plus a random constant; and a mixture of massive Gaussian free fields, where part of the mass is “lost’’ to the region of concentration.

\subsubsection{Gibbs measure for stochastic Allen–Cahn equation and single-soliton manifold}\label{SUBSUBSEC:ALLCAHN}
For the (invariant) Gibbs measure of the stochastic one-dimensional Allen–Cahn equation, Weber \cite{WEB10} proved that, in a joint low-temperature and infinite-volume limit, the measure concentrates on the single soliton manifold. Subsequently, Otto–Weber–Westdickenberg \cite{OWW14} studied the same limits, with $\eps \to 0$ and $L_\eps \to \infty$, but identified the optimal scaling of the interval length $L_\eps$ by analyzing the competition between energy $\eps \to 0$ and entropy $L_\eps \to \infty$. Under this optimal scaling, they again observed concentration of the measure around the single soliton manifold. Recently, Bertini–Buttà–Di Gesù \cite{BBG} showed that, beyond the optimal length of the interval $L_\eps$, the measure no longer concentrates, and the interfaces (the soliton centers $\xi_j$) become asymptotically distributed according to a Poisson point process. 
It would be interesting to investigate the optimal length scale for which Theorems \ref{THM:1}, \ref{THM:2}, and \ref{THM:3} remain valid, and to understand what kinds of behavior occur beyond this optimal scale.

%\subsubsection{Gibbs measure and topological solitons}
%In \cite{}, Fr\"ohlich and Marchetti studied topological solitons 
%Bringmann \cite{Bring} recently studied exterior equivariant wave maps with spatial domain $\R^3 \setminus B(0,1)$, which admit topological solitons. In contrast to our setting, where the high charge sector $|Q| \ge 2$ has no minimizer, the wave map system admits infinitely many topological sectors, and each sector possesses a unique minimizer without symmetry. In \cite{Bring}, Bringmann constructed the Gibbs measure on each homotopy class and proved the invariance of this measure under the corresponding PDE flow.

\subsubsection{Gibbs measure and topological solitons}\label{SUBSUBSEC:TOPSOL}
We refer to the work of Bringmann \cite{Bring} on topological solitons and Gibbs measures.  Bringmann recently studied exterior equivariant wave maps with spatial domain $\R^3 \setminus B(0,1)$, which admit topological solitons. In contrast to our setting, where the high charge sector $|Q| \ge 2$ has no minimizer, the wave map system admits infinitely many topological sectors, and each sector possesses a unique minimizer. In \cite{Bring}, Bringmann constructed the Gibbs measure on each homotopy class and proved the invariance of this measure under the corresponding PDE flow. Notice that the measure considered by Bringmann \cite[(1.10)]{Bring} takes a form similar to that of the sine–Gordon measure considered here \eqref{Gibbs1}, \eqref{BB1}, where the base point is fixed, and the measure is defined on the fluctuation coordinate.

Finally, we remark that, as discussed by Manton–Sutcliffe \cite{MSTOPSOL}, there exist many models admitting infinitely many disjoint topological sectors, each supporting topological solitons, including the Abelian Higgs, Ginzburg–Landau, and Yang–Mills models. It would be natural to investigate whether the methods developed in this paper can be applied to study the concentration and fluctuations of Gibbs measures in these settings. The present work may be viewed as a first step toward such a program.

%showed that soliton resolution fails for Gibbs measure initial data.

%Since every sector has a unique minimizer, we can easily apply our method to obtain the concentration and fluctuations around the single soliton.

%\subsubsection{The dynamical sine--Gordon model}

\section{Notations and preliminaries}\label{SEC:Notation}

\subsection{Notations}

Throughout the paper, we fix $L_\eps=\eps^{-\frac 12+\eta}$, where $\eta>0$
is an arbitrarily small but fixed number, and we denote by $\jb{\cdot,\cdot}$ the $L^2$ inner product on the interval $[-L_\eps,  L_\eps ]$:
\begin{align*}
\jb{\phi, \psi }:=\jb{\phi,\psi}_{L^2([-L_\eps,L_\eps])}=\int_{-L_\eps}^{L_\eps} \phi \psi dx, 
\end{align*}

\noi
where $\phi$ and $\psi$ are real-valued functions. All $L^p$-norms appearing below are understood to be over the interval $[-L_\eps, L_\eps]$, and we suppress the domain from the notation for $1\le p \le \infty$
\begin{align*}
\|\phi\|_{L^p}^p=\|\phi\|_{L^p([-L_\eps,L_\eps]) }^p=\int_{-L_\eps}^{L_\eps}|\phi|^p dx, \quad  \| \phi\|_{L^\infty}=\| \phi\|_{L^\infty([-L_\eps, L_\eps])}.
\end{align*}

\noi
When we use $L^p$-norms on the real line $\R$, we explicitly write $L^p(\R)$.

Regarding the topological degree $Q \in \Z$ defined in \eqref{topd}, through the paper we mainly consider the case $Q>0$. 
When $Q<0$, the same results follow by replacing the multi-soliton
$m_{\xi_1,\dots,\xi_k}=m_{\xi_1}+\dots+m_{\xi_k} $, consisting of kinks, with $m_{\xi_1,\dots,\xi_k}^{-}=m_{\xi_1}^{-}+\dots+m_{\xi_k}^{-} $, consisting of anti-kinks.

Let $A_1,\dots, A_k$ be measurable sets. We use the notation 
\begin{align*}
\E \Big[F(\phi), A_1,\dots A_k \Big]=\E\bigg[F(\phi)\prod_{j=1}^k \ind_{A_j} \bigg],
\end{align*}

\noi
where $\E$ stands for the expectation with respect to the probability distribution of $\phi$ under consideration.

We use $c>0$ to denote an unimportant positive constant whose value may change from line to line. We write $ A \les B$ to indicate an estimate of the form $ A \leq CB$ for some $C> 0$.
We also write  $ A \approx B $ to denote $ A \les B $ and $ B \les A $ and use $ A \ll B $ 
when we have $A \leq \zeta B$ for some small $\zeta > 0$.
We may include subscripts to show dependence on external parameters; for example, $A\les_{p} B$ means $A\le C(p) B$, where the constant $C(p)$ depends on a parameter $p$. In addition, we use $a-$ 
and $a+$ to denote $a- \eta$ and $a+ \eta$, respectively for arbitrarily small $\eta > 0$.

\subsection{Brownian bridges and homotopy-class representatives}\label{SUBSEC:BB}

Recall that, in defining the base Gaussian measure corresponding to the homotopy class $\mathcal{C}_Q$, we introduce the Gaussian measure $\mu^Q_\eps$ in \eqref{BB}, conditioned on $\phi(-L_\eps)=0$ and $\phi(L_\eps)=2\pi Q$. This measure is precisely the law induced by the Brownian bridge $B_{\eps, (-L_\eps, L_\eps)}^{0,2\pi Q}$   
\begin{align}
B_{\eps, (-L_\eps, L_\eps)}^{0,2\pi Q}(x)=\frac{\pi Q}{L_\eps}(x+L_\eps) +\sqrt{\eps} B_{(-L_\eps, L_\eps)}^{0,0}(x),
\label{BB1}
\end{align}

\noi 
where $B_{(-L_\eps, L_\eps)}^{0,0}$ is the mean zero Brownian bridge pinned at $0$ at both ends $-L_\eps, L_\eps$, and its covariance is given by 
\begin{align}
\E\Big[\sqrt{\eps} B_{(-L_\eps, L_\eps)}^{0,0}(x_1) \sqrt{\eps} B_{(-L_\eps, L_\eps)}^{0,0}(x_2) \Big]=\frac{\eps}{2L_\eps} \big( (x_1+L_\eps)(L_\eps-x_2) \wedge (x_2+L_\eps)(L_\eps-x_1)   \big)
\label{BB2}
\end{align}

\noi
for all $x_1,x_2 \in [-L_\eps, L_\eps]$. In \eqref{BB1}, $\ell^Q(x)=\frac{\pi Q}{L_\eps}(x+L_\eps)$ is the affine function interpolating between $0$ and $2\pi Q$. Thus, the fluctuations are of order
$\sqrt{\eps}$ around this straight line connecting the boundary values.
Recall that we interpret a field $\phi$ distributed according to $\mu^Q_\eps$ as a function on $\R$, extended by $0$ and $2\pi Q$ outside $[-L_\eps, L_\eps]$. Hence, the reference profile $\ell^Q(x)=\frac{ \pi Q }{ L_\eps}(x+L_\eps)$ is extended by the same boundary values, while the fluctuation coordinate $B^{0,0}_{(-L_\eps, L_\eps) }$ is extended by $0$ outside $[-L_\eps, L_\eps]$.

%We now study the Fourier series representation of the Brownian bridge
%$B_{\eps, (-L_\eps, L_\eps)}^{0,2\pi Q}$ in \eqref{BB3}. 

Let $\{e_n\}_{n\ge 1}$ be the $L^2$-orthonormal eigenfunctions of $-\dx^2$ on $[-L_\eps, L_\eps]$ with Dirichlet boundary conditions
\begin{align}
e_n(x)= \frac{1}{ \sqrt{L_\eps} } \sin\Big(  \frac{n\pi(x+L_\eps)  }{2L_\eps}   \Big)
\end{align}

\noi
for $n \in \N$. The corresponding eigenvalues are $\ld_n=\Big( \frac{n \pi}{2L_\eps} \Big)^2$.  Then, for $x\in [-L_\eps, L_\eps]$, the Brownian bridge admits the Fourier series representation
\begin{align}
B_{\eps, (-L_\eps, L_\eps)}^{0,2\pi Q}(x)=\frac{\pi Q}{L_\eps}(x+L_\eps) +\sqrt{\eps} \sum_{n=1}^\infty   \frac{2 \sqrt{L_\eps} g_n  }{\pi n}   \sin\Big(  \frac{n\pi(x+L_\eps)  }{2L_\eps}   \Big),
\label{BB0}
\end{align}

\noi
where $\{g_n\}_{n \ge 1}$ is a family of independent standard Gaussian random variables.

%\hao{I guess if we want to shorten the paper in the end we could shorten this subsection by a half..

%Also, I feel that the notation $m^Q(x)$ might be confused with the solitons $m_{...}(x)$, especially in places like Sec 6 where they show up together}

\subsection{Bou\'e-Dupuis formula}

In this subsection, we express Gaussian functional integrals with respect to the Brownian bridge measure $\mu^Q_\eps=\Law\big(B_{\eps, (-L_\eps, L_\eps)}^{0,2\pi Q} \big)$ in \eqref{BB1} in terms of an optimal control problem. We first define a centered Gaussian process at each scale $t \in [0,1]$ as follows
\begin{align}
Y(t,x)=\sum_{n \ge 1} \frac{B_n(t)}{\sqrt{\ld_n}} e_n(x) =\sum_{n=1}^\infty   \frac{2 \sqrt{L_\eps} B_n(t)  }{\pi n}   \sin\Big(  \frac{n\pi(x+L_\eps)  }{2L_\eps}   \Big),
\label{CY0} 
\end{align}

\noi
where $\{B_n\}_{n \ge 1}$ is a sequence of independent Brownian motions. Then we have
\begin{align*}
\mu^Q_\eps=\Law\big(B_{\eps, (-L_\eps, L_\eps)}^{0,2\pi Q} \big)=\Law\big((\tfrac{\pi Q}{L_\eps}(x+L_\eps) +  \sqrt{\eps}Y(1) \big).
\end{align*}

\noi 
Next, let $\mathbb{H}_a$ denote the space of drifts, which consists of mean-zero progressively measurable processes belonging to $L^2([0,1]; L^2([-L_\eps, L_\eps]))$, $\PP$-almost surely. We are now ready to state the  Bou\'e-Dupuis variational formula \cite{BD, Ust}; in particular, see Theorem 7 in~\cite{Ust}. See also Theorem 2 in \cite{BG}.

\begin{lemma}\label{LEM:BD}
Let $Q\in \Z$ and let $\eps>0$. Suppose that  $F$ is a measurable real-valued functional such that
$\E\big[|F(Y(1))|^p\big] < \infty$ and $\E\big[|e^{-F(Y(1))}|^q \big] < \infty$ for some $1 < p, q < \infty$ with $\frac 1p + \frac 1q = 1$.
Then, we have
\begin{align*}
-\log \E_{\mu^Q_\eps}\Big[ e^{-F(\phi)}   \Big]=\inf_{\dr \in \Ha}  \E\bigg[ F\big(
\ell^Q +  \sqrt{\eps}Y(1)
+\sqrt{\eps}\Dr(1) \big)  +\frac 12 \int_0^1 \| \dr(t) \|_{L^2_x}^2 dt  \bigg], 
\end{align*}

%$\mu^Q_\eps$ is the Brownian bridge measure in \eqref{BB4},

\noi
where  $\ell^Q(x)=\frac{\pi Q}{L_\eps}(x+L_\eps)$ is the line connecting $0$ and $2\pi Q$ in \eqref{BB1}, and
\begin{align}
\Dr(t):=\int_0^t (-\dx^2)^{-\frac 12} \dr(s) ds.
\label{CY1}
\end{align}

\noi 
Here the expectation $\E = \E_\PP$ is an  expectation with respect to the underlying probability measure~$\PP$.

\end{lemma}

In the following, we use the shorthand notations $Y:=Y(1)$ and $\Dr:=\Dr(1)$ for convenience.

In the large deviation estimates (Section \ref{SEC:LDP}), we need moment estimates for $Y(t)$ and a pathwise estimate for the drift term.
\begin{lemma}\label{LEM:Moment}
Let $Y(t)$ and $\Dr(t)$ be as in \eqref{CY0} and \eqref{CY1}.

\textup{(i)} 
For any $t\in [0,1]$, we have 
\begin{align}
\E \bigg[ \int_{-L_\eps}^{L_\eps} |\sqrt{\eps} Y(t)|^2 dx \bigg]
& = \frac 23 t \cdot \eps L_\eps^2, \label{2M}\\
\E \bigg[ \int_{-L_\eps}^{L_\eps} |\sqrt{\eps} Y(t)| dx \bigg]
& = \frac{\sqrt{\pi t}}{2} \cdot  \eps^\frac 12  L_\eps^{\frac 32}.
\label{1M}
\end{align}

\smallskip

\noi
\textup{(ii)}
The drift term $ \Dr(t) $ has the regularity
of the Cameron-Martin space, that is, for any $\dr\in \mathbb{H}_a $, we have
\begin{align}
\| \Dr(1)\|_{  \dot{H}^{1}_x}^2 \leq \int_0^1 \|  \dot{\Dr}(t) \|_{\dot{H}^1_x}^2dt=\int_0^1 \| \dr(t) \|_{L^2_x}^2 dt,
\label{DRE}
\end{align}

\noi
where $\dot{\Dr}(t)=(-\dx^2)^{-\frac 12}\dr(t)$.
\end{lemma}

\begin{proof}
For Part (i),
\eqref{2M} follows immediately from \eqref{CY0},
Parseval's identity, and $\sum_{n\ge 1}\frac{1}{n^{2}}=\frac{\pi^2}6$.
To prove \eqref{1M}, by \eqref{CY0},
\eqref{BB1}, \eqref{BB0}, and \eqref{BB2}, we have  
\begin{align*}
\E|Y(t,x)|^2=t \sum_{n \ge 1} \frac{1}{\ld_n} e_n(x)^2
=t\E\Big[ (B_{(-L_\eps, L_\eps)}^{0,0}(x) )^2 \Big]=t \frac{(x+L_\eps)(L_\eps-x)}{2L_\eps}.
\end{align*}
Namely $\sqrt{\eps}Y(t,x) \sim \mathcal{N}(0, \eps t  \tfrac{(x+L_\eps)(L_\eps-x)}{2L_\eps}  )$. Recall that for a centered Gaussian $Z \sim \mathcal{N}(0,\s^2)$, $\E|Z|=\s \sqrt{ \frac{2}{\pi}} $. Therefore,
\eqref{1M} follows by computing $\frac{\sqrt{\eps t} }{\sqrt{\pi L_\eps}  }  \int_{-L_\eps}^{L_\eps} \sqrt{L_\eps^2-x^2}\, dx$.

Part (ii) follows from Minkowski and Cauchy-Schwarz inequalities. 

\end{proof}

\section{Structure of multi–topological solitons}\label{SEC:TOPSOL}

In this section, we study the properties of topological solitons and their superpositions forming multi-soliton configurations.
The key points are to understand how solitons interact with each other on the collision scale, and how the energy behaves even though these configurations are not minimizers.

\subsection{Topological solitons}

In this subsection, we investigate the minimizers of the Hamiltonian 
\begin{align*}
\inf_{\phi \in \mathcal{C}_Q} E(\phi).
\end{align*}

\noi
within the homotopy class $\mathcal{C}_Q$ for $|Q|=1$, referred to as topological solitons. The topological solitons, namely the kink ($Q=1$) and anti-kink ($Q=-1$) centered at $\xi \in \R$
\begin{align*}
m_{\xi}(x)=m(x-\xi)=4 \arctan(e^{x-\xi}), \qquad  m_{\xi}^{-}(x)=m^{-}(x-\xi)=4 \arctan(e^{-(x-\xi)})
\end{align*}

\noi 
are localized transition layers connecting the distinct vacua $0$ and $2\pi$, satisfying  the Euler–Lagrange equation
\begin{align*}
-\dx^2 \phi+\sin \phi=0. 
\end{align*}

\noi
The kink represents an increasing transition from $0$ to $2\pi$, whereas the anti-kink corresponds to the decreasing transition from $2\pi $ to $0$. These configurations $\{m_\xi \}_{\xi \in \R}$ and $\{m_{\xi}^{-}\}_{\xi \in \R}$ minimize the Hamiltonian within their respective topological sectors $\mathcal{C}_Q$
\begin{align*}
E_{\textup{kink}}=\inf_{\phi \in \mathcal{C}_1} E(\phi)=E(m_{\xi})=8, \qquad  E_{ \textup{anti-kink} }=\inf_{\phi \in \mathcal{C}_{-1}} E(\phi)=E(m_{\xi}^{-})=8.
\end{align*}

In the following lemma, we show that each $m_{\xi}(x)$ is exponentially close to a vacuum (0 or $2\pi$), that is, nearly constant, when $x$ is far from its center $\xi$.

%the kink profile stays near one of the vacua ($0$ or $2\pi$) for most values of $x$, and undergoes a rapid transition between them around its center $\xi_j$

\begin{lemma}\label{LEM:sol}
Let $\xi \in \R$.

\begin{itemize}
\item[(i)] When $x>\xi$,
\begin{align*}
2e^{-|x-\xi|} \le |m_{\xi }(x) -2\pi  | \le 4 e^{-|x-\xi|}, \quad 2 e^{-|x-\xi|}  \le |m_{\xi }^{-}(x)  | \le 4 e^{-|x-\xi|},
\end{align*}

\noi
uniformly in all centers $\xi$ and all $x\in \R$.

\medskip

\item[(ii)] When $x<\xi$,
\begin{align*}
2 e^{-|x-\xi|} \le |m_{\xi }(x)  | \le 4 e^{-|x-\xi|}, \quad 2 e^{-|x-\xi|}   \le |m_{\xi }^{-}(x)-2\pi  | \le 4 e^{-|x-\xi|},
\end{align*}

\noi
uniformly in all centers $\xi$ and all $x\in \R$.

\medskip 

\item[(iii)] %Let $\xi \in \R$. Then we have 
We have
\begin{align*}
|\partial_{\xi} m_{\xi}(x) |
+|\partial_{\xi}^2 m_{\xi}(x) | 
\le 4 e^{-|x-\xi|},  
\quad |\partial_{\xi} m_{\xi}^{-}(x) |
+|\partial_{\xi}^2 m_{\xi}^{-}(x)| 
\le 4 e^{-|x-\xi|},  
\end{align*} 

\noi
uniformly in all centers $\xi$ and all $x\in \R$. The same holds with $\xi$-derivatives replaced by $x$-derivatives. 

\end{itemize}

\begin{proof}
The parts (i) and (ii) follow from direct computation using $m_{\xi}(x)=4\arctan (e^{x-\xi})$ and $m^{-}_{\xi}(x)=4\arctan(e^{-(x-\xi)})$.
The part (iii) follows from the fact that $\partial_{\xi}m_{\xi}(x)=-2\textup{sech}(x-\xi)$,
and $\partial_{\xi}^2  m_{\xi}(x) =2 \textup{tanh}(x-\xi) \textup{sech}(x-\xi)$ and direct computations.
These derivatives are highly localized profile around $\xi$ with an exponentially decaying tail.
\end{proof}

%\begin{align*}
%|m_{\xi }(x) -(2\pi) \cdot \ind_{ \{ x>\xi   \} } |+ |\partial_{\xi} m_{\xi}(x) | \les e^{-|x-\xi|},  
%\end{align*}

\end{lemma}

In the following lemma, we show that (1) translation is the only symmetry of the minimizer, and (2) if a field $\phi \in \mathcal{C}_Q$ is far away from the family of minimizers, then its energy is also far away from the minimal energy.
\begin{lemma}\label{LEM:SEC1}
Let $|Q|=1$.

\begin{itemize}

\item[(1)] If $ G \in \mathcal{C}_Q$ satisfies $E(G)=\inf_{\phi \in \mathcal{C}_Q}E(\phi) $, then there exists $\xi\in \R$ such that 
\begin{align*}
G(x)=m(x-\xi) \quad \textup{when $Q=1$}, \quad G(x)=m^{-}(x-\xi) \quad \textup{when $Q=-1$}.
\end{align*}

\medskip  

\item[(2)] 
Let $\textup{dist}(\phi, \M_1):=\inf_{\xi \in \R} \| \phi-m_{\xi} \|_{L^2(\R)}$ and $\textup{dist}(\phi, \M_{-1}):=\inf_{\xi \in \R} \| \phi-m_{\xi}^{-} \|_{L^2(\R)}$. There exists $c>0$ such that if $\phi \in \mathcal{C}_Q$ satisfies 
\begin{align*}
\textup{dist}(\phi, \M_Q)\ge \dl>0, 
\end{align*}

\noi
then
\begin{align*}
E(\phi) \ge  \inf_{\phi \in \mathcal{C}_Q } E(\phi)+c \cdot \textup{dist}(\phi, \M_Q)^2 \ge E_{\textup{kink}}+c \cdot \dl^2.
\end{align*}

\end{itemize}

\end{lemma}

\begin{proof}
The parts (1) and (2) follow from the concentration compactness argument together with 
\begin{align*}
E(\phi)&=E(m_\xi)
+ \jb{\nb E(m_\xi), \phi-m_\xi}_{L^2(\R)} 
\\
&\qquad + \frac12\jb{\phi-m_\xi, 
        \nb^2 E(m_\xi) (\phi-m_\xi)}_{L^2(\R)} 
+O( \| \phi-m_\xi \|_{L^2(\R)}^3 )
%\\
%& \ge E_{ \textup{kink}}+O( \| \phi-m_\xi \|_{L^2(\R)}^2 )
\end{align*}

\noi 
since $\nb E(m_\xi)=0$. For details, see, for example, \cite[Lemma 2.4]{TW}, \cite[Proposition 2.2]{WEB10}, and \cite[Lemma 6.5]{Frank}.
\end{proof}

\subsection{Multi-topological solitons}

Unlike the class $\mathcal{C}_Q$ for $|Q|=1$, it is well known that there is no minimizer in the homotopy class $\mathcal{C}_Q$ when $|Q| \ge 2$. 
In this subsection, we therefore investigate the properties of superpositions of topological solitons under appropriate separation conditions, which ``almost'' act as  minimizers in this class (see Remark \ref{REM:AM}). Furthermore, we analyze how the energy behaves when a field is far from the multi-soliton manifold (Lemma \ref{LEM:GAP1}) and when solitons collide on the collision scale (Lemmas \ref{LEM:GAP2} and \ref{lem:gap-multikink}).

First, we study the Bogomolny lower bound on the homotopy class $\mathcal{C}_Q$.

\begin{lemma}[Bogomolny lower bound]\label{LEM:BLW}
Let $Q \in \Z$. For any $\phi \in \mathcal{C}_Q$,
\begin{align*}
E(\phi) \ge E_{\textup{kink} } |Q|=8 |Q|,
\end{align*}

\noi
where $E_{\textup{kink} }=E(m)=E(m^{-})=8$.

\end{lemma}

\begin{proof}
Note that
\begin{align*}
E(\phi)&=\frac 12 \int_{\R} |\dx \phi|^2 dx + \int_{\R} 2 \sin^2 \frac \phi2 dx
\ge 
2\int_{\R}  |\dx \phi \sin \frac \phi2| dx
\\
& \ge 2\int_C f
= \int_0^{2\pi |Q|} |\sin(z/2)|dz=8|Q|,
\end{align*}
where $\int_C f$ is a line integral, $C$ is
the curve in $\R$ going straight from $0$ to $2\pi Q$,
and $f(z)=|\sin(z/2)|$ is a function along $C$.
In the second line we have replaced the parametrization $\phi$ of $C$ by the identity parametrization $[0,2\pi Q] \to [0,2\pi Q]$ and used the independence of line integrals on parametrizations.
\end{proof}

For $\xi_1\le \cdots \le \xi_k$, where $k=|Q|$, define the superposition of topological solitons, that is, the multi-soliton by
\begin{align*}
m_{\xi_1,\dots,\xi_k}=\sum_{j=1}^k m(\cdot-\xi_j)=\sum_{j=1}^k m_{\xi_j},\qquad
m_{\xi_1,\dots,\xi_k}^{-}=\sum_{j=1}^k m^{-}(\cdot-\xi_j)=\sum_{j=1}^k m^{-}_{\xi_j}.
\end{align*}
In the following lemmas, we state the results only for the case $Q>0$, by symmetry.

We first prove an elementary inequality: for $\eta \in (0,1)$,
\begin{align}\label{e:conv}
\int_{\R} e^{-(|y-x|+|z-x|)} dx 
= (1+|y-z|) e^{-|y-z|}
\le 
\frac{1}{\eta} e^{\eta-1}  e^{-(1-\eta)|y-z|}.
\end{align}
The equality follows by observing that the integral between $y$ and $z$ is equal to $|y-z| e^{-|y-z|}$, and the integral outside 
is equal to $e^{-|y-z|}$.
For the inequality, we have
\begin{align*}
(1+|y-z|) e^{-|y-z|}
\le 
\Big( \sup_{r>0} (1+r) e^{-\eta r}\Big) e^{-(1-\eta)|y-z|}
\end{align*}
and the function $(1+r) e^{-\eta r}$ reaches maximum at $r=\frac{1}{\eta}-1$.

When the centers $\xi_1,\dots,\xi_k$ are well separated, each kink $m_{\xi_j}$ contributes its own
$E_{\textup{kink}}$, and the overlap between kinks yields only exponentially small corrections, because each kink remains nearly constant (0 or $2\pi$) outside its center $\xi_j$. This is shown in the next lemma.

\begin{lemma}\label{LEM:MUS}
Let $Q\in \Z$ with $|Q|=k$. Then we have 
\begin{align*}
E(m_{\xi_1,\dots,\xi_k})&=\sum_{j=1}^k E(m_{\xi_j})+O(e^{-c \min_{i \neq j}|\xi_i-\xi_j|  })
\end{align*}

\noi
as $\min_{i \neq j}|\xi_i-\xi_j| \to \infty$, where $c>0$ only depends on $|Q|$.

\end{lemma}

\begin{proof}
The kinetic energy part is 
\begin{align*}
\frac 12\int_{\R} |\dx m_{\xi_1,\dots,\xi_k}|^2 dx&=\frac 12 \sum_{j=1}^k\int_{\R}|\dx m_{\xi_j}|^2 dx +\sum_{i<j} \int_{\R} \dx m_{\xi_i}  \dx m_{\xi_j} dx\\
&=\frac 12 \sum_{j=1}^k\int_{\R}|\dx m_{\xi_j}|^2 dx+\sum_{i<j} O(e^{-c|\xi_i-\xi_j|})
\end{align*}

\noi
%since $\dx m_{\xi_j}=2\textup{sech}(\cdot-\xi_j)$ is highly localized around
%$\xi_j$ with an exponentially decaying tail,
by Lemma~\ref{LEM:sol}(iii) and \eqref{e:conv}.
%the fact that convolutions  of exponentially decaying functions  are exponentially decaying.

We now study the potential energy part. 
Let $U(z)=1-\cos(z)$. 
One has 
\[
|U(a+b)-U(a)-U(b)| = 
\Big|\int_0^a \int_0^b U''(s+t) dsdt\Big|
\le |a||b|
\]
and by induction in $k$ one has 
\begin{align}\label{e:nice-ineq}
\Big|U\Big(\sum_{j=1}^k a_j\Big)-\sum_{j=1}^k U(a_j)\Big| 
\le \sum_{i < j}|a_i||a_j|.
\end{align}
Indeed, suppose that this holds for $k-1$. Then  the left-hand side of 
\eqref{e:nice-ineq} is bounded by
\[
\Big|U\Big(\sum_{j=1}^k a_j\Big) -U\Big(\sum_{j=1}^{k-1} a_j\Big) - U(a_k)\Big|
+
\Big|U\Big(\sum_{j=1}^{k-1} a_j\Big) + U(a_k)-\sum_{j=1}^k U(a_j)\Big| 
\]
and \eqref{e:nice-ineq} follows by using the induction assumption.

Set $a_j= m_{\xi_j}(x) - 2\pi 1_{x>\xi_j}$.
(Note that the shift $2\pi 1_{x>\xi_j}$ does not change the value of $\cos$.)
By Lemma~\ref{LEM:sol}(i)(ii), $|a_j| \le  4 e^{-|x-\xi_j|}$. The desired bound follows upon integrating over $x$ 
%which is again a convolution.
again by \eqref{e:conv}.
\end{proof}

\begin{remark}\rm \label{REM:AM}

According to the Bogomolny structure in Lemmas \ref{LEM:BLW} and \ref{LEM:MUS}, by choosing a minimizing sequence with infinite separation between centers, we can show that the minimal energy in each homotopy class $\mathcal{C}_Q$ is 
\begin{align*}
\inf_{\phi \in \mathcal{C}_Q} E(\phi)= |Q|  E_{\textup{kink}}.
\end{align*}

\noi
This shows that the infimum of the energy in the topological sector
$\mathcal{C}_Q$ is attained asymptotically by configurations consisting of
$Q$ kinks when $Q>0$ or $|Q|$ anti-kinks when $Q<0$, with infinite mutual separation $\min_{i \neq j}|\xi_i-\xi_j| \to \infty$.
However, the minimizing sequence has no convergent subsequence due to the infinite separation, which cannot be realized in practice. This is one of the reasons why there is no actual minimizer for $|Q| \ge 2$.

\end{remark}

From Remark \ref{REM:AM}, under the separation condition $m_{\xi_1,\dots, \xi_k}$ is an almost minimizer on the homotopy class $\mathcal{C}_Q$ (though it never attains the actual minimum).  In the following lemma, we show that the first variation operator is close to zero.

\begin{lemma}\label{LEM:gra}
Let $Q \in \Z$ with $|Q|=k$. For $0<\eta \ll 1$ sufficiently small, 
\begin{align*}
\| \nb E(m_{\xi_1,\dots,\xi_k}) \|_{L^2(\R) } 
%\le  C^{\nb} \frac{ e^{-(1-\eta)d}}{1-e^{-(1-\eta)d}}
%e^{-(1-\eta) \min_{i \neq j}|\xi_i-\xi_j| }
\le C_\eta e^{- (1-\eta) d},
\end{align*}
uniformly in $\xi_1,\dots,\xi_k$ 
%satisfying the separation condition
%$\min_{i \neq j}|\xi_i-\xi_j| =: d  \to \infty$. 
where 
$d:= \min_{i \neq j}|\xi_i-\xi_j| $ and 
%$C^{\nb} = (\frac{1}{2\eta} e^{\eta-1})^{1/2} Q$.
$C_\eta = O(\eta^{-1/2})$.
\end{lemma}

\begin{proof}

Note that 
\begin{align*}
\nb E(m_{\xi_1,\dots,\xi_k})(x)=-\dx^2 m_{\xi_1,\dots,\xi_k}(x)+ \sin m_{\xi_1,\dots,\xi_k}(x)
=
 \sin m_{\xi_1,\dots,\xi_k}(x)
 -\sum_{j=1}^k \sin m_{\xi_j}(x)
\end{align*}
where we used $\nb E(m_{\xi_j})=-\dx^2 m_{\xi_j} + \sin m_{\xi_j}=0$.
Note that \eqref{e:nice-ineq} again holds with 
$U(z)=\sin(z)$. So
\begin{align}
|\nb E(m_{\xi_1,\dots,\xi_k})(x)|
\le 
\sum_{i < j} e^{- |x-\xi_i|}e^{- |x-\xi_j|}.
\label{nbE}
\end{align}
Regarding the $L^2$ bound, note that  by \eqref{e:conv},
\[
\int_{\R} e^{- 2 \,(|x-\xi_i|+|x-\xi_j|)} dx
%\lesssim |\xi_i-\xi_j| e^{-2 \,|\xi_i-\xi_j|}
%\les e^{-2(1-\eta)|\xi_i-\xi_j| }
\le \frac{1}{2\eta} e^{\eta-1} e^{-2(1-\eta)|\xi_i-\xi_j| }
\]
for  small $\eta>0$. This implies 
\[
\| \nb E(m_{\xi_1,\dots,\xi_k})\|_{L^2(\R)}
\le
\Big(\frac{1}{2\eta} e^{\eta-1} \Big)^{1/2}
\sum_{i<j}
e^{-(1-\eta)|\xi_i-\xi_j| }.
\]
% One could estimate $\sum_{i<j}$ by $Q^2$, but to be slightly sharper,
% we assume without loss of generality that 
% $\xi_1 \le \cdots \le \xi_k$.
% Then there are $Q-1$ terms of the form 
% $e^{-(1-\eta)|\xi_i-\xi_{i+1}| }$,
% and $Q-2$ terms of the form
% $e^{-(1-\eta)|\xi_i-\xi_{i+2}| }$ which can be written as
% $e^{-(1-\eta)|\xi_i-\xi_{i+1}| -(1-\eta)|\xi_{i+1}-\xi_{i+2}| }$, etc.
% Thus
% \[
% \| \nb E(m_{\xi_1,\dots,\xi_k})\|_{L^2(\R)}
% \le
% (\frac{1}{2\eta} e^{\eta-1})^{1/2}
% \sum_{k=1}^{Q-1}
% (Q-k)
% e^{-(1-\eta)k d }
% \le 
% (\frac{1}{2\eta} e^{\eta-1})^{1/2}
% Q \frac{ e^{-(1-\eta)d}}{1-e^{-(1-\eta)d}}.
% \]
Thus the bound follows.
%with perhaps a different $c>0$.
%Regarding the $L^\infty$ bound, note that 
%$\inf_{x\in \R} (|x-\xi_i|+|x-\xi_j|) = |\xi_i-\xi_j|$.
\end{proof}

The following lemma shows that if a field $\phi \in \mathcal{C}_Q$ is far away from the multi-soliton manifold, then its energy is also far away from the minimal energy.

\begin{lemma}\label{LEM:GAP1}
Let $Q \in \Z$ with $Q\ge 2$ and let $\dl>0$. Then there exists $c>0$, depending only on $Q$, such that if $\phi \in \mathcal{C}_Q$ with $\phi(-\infty)=0$ and $\phi(\infty)=2\pi Q$ satisfies 
\begin{align}
\textup{dist}(\phi, \M^Q):=\inf_{\xi_1,\dots, \xi_Q \in \R} \| \phi-m_{\xi_1,\dots,\xi_k} \|_{L^2(\R)} \ge \dl>0,
\label{GAP1}
\end{align}

\noi
then
\begin{align*}
E(\phi) \ge \inf_{\phi \in \mathcal{C}_Q } E(\phi)   +c \inf_{\xi_1,\dots, \xi_Q \in \R} \| \phi-m_{\xi_1,\dots,\xi_k} \|_{L^2(\R)}^2 \ge \inf_{\phi \in \mathcal{C}_Q } E(\phi)+  c \dl^2.
\end{align*}

\end{lemma}

%+c \inf_{\xi_1,\dots, \xi_Q \in \R} \| \phi-m_{\xi_1,\dots,\xi_k} \|_{L^2(\R)}^2 \ge |Q| E_{\textup{kink}}+ 

\begin{proof}

Fix $\phi \in \mathcal{C}_Q$ with $\phi(-\infty)=0$ and $\phi(\infty)=2\pi Q$.
Define
\[
t_0=-\infty, \quad t_Q= \infty, \quad t_j=\inf \{x:\phi(x)=2\pi j\}
\]
for $1\le j \le Q-1$. 
 Then, define on each block $1\le j \le Q$ 
\begin{align}
\phi_j(x)=
\begin{cases}
\phi(x) \quad &x\in (t_{j-1}, t_j) \\
2\pi(j-1) \quad &x<t_{j-1}  \\
2\pi j \quad &x>t_{j}.
\end{cases}
\label{ST11}
\end{align}

\noi
Since $2\pi \Z$ extension does not change the energy, we have, for $1\le j \le Q$, 
\begin{align}
E(\phi_j)=\int_{t_{j-1}}^{t_j} |\dx \phi|^2 dx+ \int_{ t_{j-1}  }^{t_j} (1-\cos \phi)  dx.
\label{ST2}
\end{align}

\noi
Therefore, using %\eqref{ST1}, \eqref{ST2}, \eqref{ST3} 
\eqref{ST2}
and Lemma \ref{LEM:SEC1}-(2), we have  
\begin{align}
E(\phi)=\sum_{j=1}^Q E(\phi_j)&=\sum_{j=1}^Q E(\phi_j-2\pi(j-1) ) \notag \\
& \ge |Q| E_{\textup{kink}} +c \sum_{j=1}^Q \inf_{\xi \in \R}\| (\phi_j-2\pi(j-1) )-    m_{\xi} \|_{L^2(\R)}^2.
\label{ST6} 
\end{align}

\noi
In the following, we show that 
\begin{align}
\inf_{\xi_1,\dots, \xi_Q \in \R} \| \phi-m_{\xi_1,\dots,\xi_k} \|_{L^2(\R)}^2 \le c \sum_{j=1}^Q d_j^2,
\label{ST7}
\end{align}

\noi
where $d_j:=\inf_{\xi \in \R}\| (\phi_j-2\pi(j-1) )-    m_{\xi} \|_{L^2(\R)}$.
By the definition of $d_j$, for each $j$ we can pick $\xi_j^* \in \R$ so that 
\begin{align}
\| (\phi_j-2\pi(j-1) )-    m_{\xi_j^*} \|_{L^2(\R)}^2 <d_j^2+\eps 
\label{ST4}
\end{align}

\noi
for some small $\eps>0$ (later we let $\eps \to 0$). Fix such a choice $(\xi_1^*,\dots,\xi_Q^*)$ from now on. By the definition of $\phi_j$, one can easily check that for all $x\in \R$
\begin{align*}
\phi(x)=\sum_{j=1}^Q ( \phi_j(x) -2 \pi  (j-1)     ).
\end{align*}

\noi
This, together with \eqref{ST4}, implies that
\begin{align*}
\inf_{\xi_1,\dots, \xi_Q \in \R} \| \phi-m_{\xi_1,\dots,\xi_k} \|_{L^2(\R)}^2 &\le 2^{Q-1} \sum_{j=1}^Q \| (\phi_j-2\pi(j-1) )-    m_{\xi_j^*} \|_{L^2(\R)}^2\\
&\le 2^{Q-1} \sum_{j=1}^Q d_j^2  +2^{Q-1}|Q| \eps. 
\end{align*}

\begin{comment}
For the block $j=1$, we define $\psi_1(x)=m(x-\xi_1^*)$ when $x<t_1$ and $\psi_1(x)=2\pi$ when $x\ge t_1$.  
As in the construction of $\phi_j$ on the $j$-th block \eqref{ST11},  define the kink profile  on block $2\le j \le Q-1$ 
\begin{align*}
\psi_j(x):=
\begin{cases}
2\pi(j-1), \quad  &x \le t_{j-1}\\
2\pi(j-1)+m(x-\xi_j^*), \quad &x\in (t_{j-1}, t_{j})\\
2\pi j, \qquad   &x \ge t_j.
\end{cases}
\end{align*}

\noi
For the block $j=Q$, define $\psi_Q(x)=2 \pi (Q-1)$ when $x \le t_{Q-1}$ and $\psi_Q(x)=2\pi(Q-1)+m(x-\xi_Q^*)$ when $x>t_{Q-1}$.
Now define $\Psi$ globally 
\begin{align*}
\Psi(x):=\sum_{j=1}^Q( \psi_j(x)-2\pi(j-1) ), \quad x\in \R.
\end{align*}

\noi
Then on each block, one can easily check that $\Psi(x)=\psi_j(x)$ for $ x\in (t_{j-1}, t_j)$, where $t_0:=-\infty$ and $t_Q:=\infty$.
This implies, along with \eqref{ST11}, that 
\begin{align}
\| \phi- \Psi \|_{L^2(\R)}^2=\sum_{j=1}^{Q} \int_{t_{j-1}}^{t_j} | \phi_j(x)- \psi_j(x) |^2dx &\le \sum_{j=1}^Q  \| (\phi_j-2\pi(j-1) )-    m_{\xi_j^*} \|_{L^2(\R)}^2  \notag \\
& \le \sum_{j=1}^Q d_j^2+ Q \eps,
\label{ST5}
\end{align}

\noi
where in the last line we used \eqref{ST4}. By the construction of $\Psi$, $\| \Psi-m_{\xi_1^*,  \dots, \xi_Q^*} \|_{L^2(\R)}$ is exponentially small. Hence, using \eqref{ST5}, we have 
\begin{align*}
\inf_{\xi_1,\dots, \xi_Q \in \R} \| \phi-m_{\xi_1,\dots,\xi_k} \|_{L^2(\R)}^2 &\le 2 \| \phi-\Psi \|_{L^2(\R)}^2+2 \| \Psi-m_{\xi_1^*,\dots, \xi_Q^* }  \|_{L^2(\R)}^2\\
&\le  c \sum_{j=1}^Q d_j^2+ cQ \eps.
\end{align*}
\end{comment}

\noi
Letting $\eps \to 0$, we obtain \eqref{ST7}. Combining \eqref{ST6} and \eqref{ST7} yields 
\begin{align*}
E(\phi) \ge |Q| E_{\textup{kink}}+c \inf_{\xi_1,\dots, \xi_Q \in \R} \| \phi-m_{\xi_1,\dots,\xi_k} \|_{L^2(\R)}^2.
\end{align*}

\noi 
Since $\inf_{\phi \in \mathcal{C}_Q}E(\phi)=|Q|E_{\textup{kink}}$ from Remark \ref{REM:AM}, we obtain the desired result.

%Using the fact that $\inf_{\xi_1,\dots, \xi_Q \in \R} \| \phi-m_{\xi_1,\dots,\xi_Q} \|_{L^2(\R)} \ge \dl>0$, \textcolor{blue}{there exists at least one piece
%$\phi_j$ such that  
%\begin{align*}
%\inf_{\xi_j \in \R}\| (\phi_j-2\pi(j-1) )-    m_{\xi_j}  \|_{L^2(\R)}^2 \ge c(\dl).
%\end{align*}
%}

%\noi
%Therefore, we obtain the desired result. 

\end{proof}

\begin{remark}\rm 
In Lemma \ref{LEM:GAP1}, we obtain an energy gap estimate when the field is far from the multi-soliton manifold $\M^Q$. The main point is to quantify how much the energy exceeds the minimal energy. Specifically, the energy is higher than the minimum by an amount of order $\dl^2$ when $\textup{dist}(\phi, \M^Q) \ge \dl$.
\end{remark}

\begin{remark}\rm 
In Lemma \ref{LEM:GAP1}, when $Q\in \Z$ is negative with $Q \le -2$, the same result holds by replacing the condition \eqref{GAP1} with the multi-antikink manifold, as follows:
\begin{align*}
\textup{dist}(\phi, \M^Q):=\inf_{\xi_1,\dots, \xi_{|Q|} \in \R} \| \phi-m_{\xi_1,\dots,\xi_k}^{-} \|_{L^2(\R)} \ge \dl>0.
\end{align*}
\end{remark}

\subsection{Collision regime}

Recall from Remark \ref{REM:AM} that the minimal energy $\inf_{\phi \in \mathcal{C}_Q} E(\phi)$ is achieved in the limit of multi-solitons with infinite separation, that is, when $\min_{i \neq j} |\xi_i-\xi_j|\to \infty$.
In the following lemma, we show that although a field $\phi$ may be close to the multi-soliton manifold $\textup{dist}(\phi, \M^Q)< \dl$, if the solitons are not well separated $\min_{i \neq j} |\xi_i-\xi_j|<d$, then its energy remains far from the minimal energy.

%and $\pi_Q$ is the projection onto the manifold $\M_Q$, defined in Theorem \ref{THM:1}-(ii), as the closest point to $\M_Q$.

Before beginning the proof, we recall that $\M^{<d}_Q$ is the collision manifold introduced in \eqref{man1}.  The main part is to quantify the error, namely, how much the energy exceeds the minimal energy $|Q|E_{\textup{kink}}$ when the solitons collide on the scale $\min_{i \neq j} |\xi_i-\xi_j|<d$. In this regime, the energy is higher than the minimum by an amount of order $e^{-d}$.

%up to an exponentially small error depending on the distances between the centers.

%\kihoon{I thought about that again. If we use the approximating soliton manifold $\M_Q^{\eps, \le d}$, since it is compact, there is no issue. But if it is not compact, it seems to me that $\dl$ depends on each local point. I have checked another paper. Other than Hendrick's CPAM paper, every paper states theorems with an approximating soliton manifold. }

\begin{lemma}\label{LEM:GAP2}
Let $Q \in \Z$ with $Q\ge 2$, and let $d>0$ be a large constant. Then there exist $c>0$ and $C>0$  such that if $\phi \in \mathcal{C}_Q$ with $\phi(-\infty)=0$ and $\phi(\infty)=2\pi Q$ satisfies 
\begin{align}
\textup{dist}(\phi, \M_Q^{<d})=\inf_{\substack{\xi_1,\dots,\xi_Q \in \R \\ \min_{i \neq j}|\xi_i-\xi_j|<d } }\| \phi-m_{\xi_1,\dots,\xi_Q}   \|_{L^2(\R)}<\dl
\label{GAP2}
\end{align}
where $\delta\le c \,e^{-d/2}$, then
\begin{align*}
E(\phi) \ge |Q| E_{\textup{kink}}+ C\cdot e^{-d}.
\end{align*}

\end{lemma}

\begin{proof}

Since $\textup{dist}(\phi, \M_Q^{<d})<\dl $, we can find %$\xi_1^*,\dots, \xi_Q^*$ 
$m^* := m_{\xi_1^*,\dots,\xi_Q^*}$
such that 
\begin{align}
s:=\min_{i \neq j}|\xi_i^*-\xi_j^*|<d \quad 
\textup{and} 
\quad \| \phi-m^* \|_{L^2(\R)} \le \dl.
\label{middle}
\end{align}

% \noi
% Therefore, from Lemma \ref{lem:gap-multikink}, we obtain 
% \begin{align}
% E(\phi)&=\big(E(\phi)-E(m^*) \big) +E(m^*) \notag \\
% & \ge \big(E(\phi)-E(m^*) \big) +|Q|E_{\textup{kink}}+c\cdot e^{- d}.
% \label{coll5}
% \end{align}

Denoting $\eta(x):=\phi(x)-m^*(x)$, we write
\begin{align*}
E(\phi)-E(m^*)
&=\int \dx  m^* \dx \eta dx+\frac 12 \int_{\R} |\dx \eta|^2 dx  
\\
&\hphantom{X}+\int_{\R} ( 1-\cos(m^*+\eta ) ) dx-\int_{\R}  (1-\cos m^*  ) dx \notag
\\
&=\int \dx  m^* \dx \eta dx+\frac 12 \int_{\R} |\dx \eta|^2 dx 
\\
&\hphantom{X}+\int_{\R} \sin m^* \cdot \eta +\frac 12  \cos\big( m^* +\dr \eta   \big)  \cdot \eta^2 \, dx 
%\label{coll3}
\end{align*}
where we Taylor expanded $\cos(m^*(x) +\eta(x))$ in $\eta(x)$ and 
$\dr(x) \in (0,1)$.

\noi
Integrating by parts,
\begin{align*}
E(\phi)&-E(m^*)
=\int_{\R} \nb E(m^*) \cdot \eta   dx +\frac 12 \int_{\R} |\dx \eta|^2 dx+\frac 12 \int_{\R} \cos\big( m^* +\dr \eta   \big)  \cdot \eta^2 dx,
\end{align*}

\noi
where $\nb E(m^*)=-\dx^2 m^*+\sin m^*  $. Using 
$\| \eta\|_{L^2(\R)}=\| \phi-m^* \|_{L^2(\R)}\le \dl$
and $\cos \ge -1$,
\begin{align}
E(\phi)-E(m^*) 
\ge  -\| \nb E(m^*) \|_{L^2(\R)} \cdot \dl  - \frac 12 \dl^2.
\label{middle2}
\end{align}

% From \eqref{nbE}, we have  $\| \nb E(m^*) \|_{L^2(\R)}  \les e^{-(1-\eta) s  }$. This implies that 
% \begin{align}
% E(\phi)-E(m^*) \ge -c \dl -\frac 12 \dl^2 \ge -c_1 \dl
% \label{coll4}
% \end{align}

% \noi
% for all $0<\dl \le \dl_0$, where $\dl_0>0$ is sufficiently small and $c_1>0$. Combining \eqref{coll4} and \eqref{coll5} yields 
% \begin{align}
% E(\phi) \ge |Q| E_{\textup{kink}}+c \cdot e^{-d}- c_1\dl  \ge |Q| E_{\textup{kink}}+\frac c2 \cdot e^{-d}
% \label{room1}
% \end{align}

% \noi
% under the condition $\frac c2 e^{-d} \ge c_1 \dl $.

By Lemma~\ref{lem:gap-multikink} below, 
%recalling $E_{\textup{kink}}=8$,
\begin{align}\label{eq:new1}
 E(m^* ) \ge  |Q| E_{\textup{kink}} +C_1 e^{-s},
\end{align}

\noi
where $s=\min_{i \neq j}|\xi_i^*-\xi_j^*|$ is defined in \eqref{middle}. By Lemma~\ref{LEM:gra}, 
\begin{align}\label{eq:new2}
\| \nb E(m^* ) \|_{L^2(\R) } 
\le C_\eta e^{- (1-\eta) s}.
\end{align}

Recall that $d$ is fixed, while $s$ depends on $\phi$.
In the following we consider two cases.

Case 1: ``Strong Collision'' regime $s \le d/2$. 
In this regime, the solitons are very close, leading to a large energy surplus. \eqref{eq:new1} implies
\begin{align*}
 E(m^* ) \ge  |Q|E_{\textup{kink}} +C_1 e^{-d/2}.
\end{align*}
For the gradient term we simply bound
$\| \nb E(m^* ) \|_{L^2(\R) } \le C_3$.
Substituting these into 
\eqref{middle2}, we have
\[
E(\phi)-|Q| E_{\textup{kink}} \ge C_1 e^{-d/2} - C_3 \delta - \frac12 \delta^2.
\]
Since $\delta \le c\,e^{-d/2}$, for $c>0$ sufficiently small,
the term $C_1 e^{-d/2}$ dominates the other two terms on the right-hand side, so we have
\[
E(\phi)-|Q| E_{\textup{kink}} \ge \frac{C_1}2 e^{-d/2} \ge C e^{-d}.
\]

Case 2: ``Weak collision'' regime $\frac{d}{2}<s<d$. 
In this regime, the separation is larger, so the energy gap $e^{-s}$ is smaller, but the gradient is also exponentially small.
From \eqref{middle2}, \eqref{eq:new1}, and \eqref{eq:new2} we have
\begin{align}
E(\phi)-|Q| E_{\textup{kink}} \ge
C_1 e^{-s} - C_\eta e^{-(1-\eta) s} \delta -\frac12 \delta^2.
\label{RHS1}
\end{align}
Note that for $c>0$ small enough,
\[
\frac12 \delta^2 \le \frac12 c^2 e^{-d} \le \frac{C_1}{8} e^{-s},
\]
where we used $\dl \le c e^{-d/2}$, and by choosing for instance $\eta=1/4$,
\[
C_\eta e^{-(1-\eta) s} \delta
\le
C_\eta  e^{-s+\eta d} \cdot c\,e^{-d/2} \le \frac{C_1}{8} e^{-s}
\]
so
the first term $C_1 e^{-s}$ in \eqref{RHS1} again dominates the other two terms on the right-hand side, and 
\[
E(\phi)-|Q| E_{\textup{kink}} \ge  C e^{-s} \ge  C e^{-d}
\]
for some $C>0$.

\end{proof}

\begin{comment}
\hao{As discussed today, if we can prove the above result only assuming something like $e^{-d} \ge \delta^2$, it might be good. Maybe we can try this: let $s=\min_{i\neq j} |\xi_i^*-\xi_j^*|$, then consider two cases:

(1) ``strong collision'': $s<d/2$. Here we have a much larger energy gap from Lemma~\ref{lem:gap-multikink}
\begin{align*}
E(m_{\xi_1^*,\dots,\xi_Q^*} ) \ge |Q| E_{\textup{kink}}+c\cdot e^{- d/2}.
\end{align*}

(2) ``weak collision'': $d/2 \le s < d$.
Here we have a much better bound on $\nabla E$ from Lemma~\ref{LEM:gra}
\begin{align*}
\| \nb E(m_{\xi_1^*,\dots,\xi_k^*}) \|_{L^2(\R) } 
\lesssim
e^{-(1-\eta)s} 
\lesssim
e^{-(1-\eta)d/2} 
\end{align*}
So I feel that we should carefully compare the constants $c$, $C^\nabla$ etc. in these two cases....}
\end{comment}

%\begin{remark}\rm 

%In Lemma \ref{LEM:GAP2} we require the condition $\frac 12 e^{-d} \ge \dl$, which is used in \eqref{room1}. Although this condition is needed for Lemma \ref{LEM:GAP2}, it does not appear in the statement of Theorem \ref{THM:1} (ii) on large deviations. Indeed, we can remove this assumption in the proof of Theorem \ref{THM:1} by combining Lemma \ref{LEM:GAP1}. See the arguments in \eqref{ldp3}, \eqref{lldp4}, and \eqref{ldp4}.

%\end{remark}

%\begin{align*}
%E(\phi)-E(m_{\xi_1^*,\dots, \xi_Q^* })= \frac 12 \int_{\R} \cos(m_{\xi_1^*,\dots, \xi_Q^* }(x) + \dr(x) \eta(x) ) |\eta(x)|^2 dx
%\end{align*}

\begin{remark}\rm 
In Lemma \ref{LEM:GAP2}, when $Q\in \Z$ is negative with $Q \le -2$, the same result holds by replacing the condition \eqref{GAP2} with the multi-antikink manifold.
\end{remark}

It remains to prove Lemma \ref{lem:gap-multikink}, which was used in the proof of Lemma \ref{LEM:GAP2}. As alluded above, the key is to quantify the error, namely, how much the energy exceeds the minimal energy 
%when the solitons collide on the scale $\min_{i \neq j} |\xi_i-\xi_j|<d$. 
in the collision regime. Unlike the previous results such as 
Lemma~\ref{LEM:MUS} and Lemma~\ref{LEM:gra} where kink interactions are weak when well-separated, here in the collision regime their interactions are stronger.
In order to better understand their interactions in this regime, our analysis will have a more ``pointwise'' (rather than $L^2$) flavor, and we will prove an exact 
pointwise representation for the energy density gap.
From this representation we will be able to identify the leading order contribution to the energy gap.

%In this case, the energy is higher than the minimum by an amount of order $e^{-d}$.

\begin{lemma}\label{lem:gap-multikink}
Let $Q \in \Z$ with $Q\ge 2$.
For $d>0$ sufficiently large, there exists $c>0$ such that
for all $\xi_1,\dots,\xi_Q \in \R $ with  $\min_{i \neq j} |\xi_i-\xi_j|=d $, 
one has 
\begin{align*}
E(m_{\xi_1,\dots,\xi_Q} ) \ge |Q| E_{\textup{kink}}+c\cdot e^{- d}.
\end{align*}
\end{lemma}

\begin{proof}
From Lemma~\ref{LEM:sol}, we have two-sided exponential asymptotic tail:
for $x<\xi$,
\begin{align}\label{e:kink-upper-lower}
2 e^{-|x-\xi|} 
\le m_{\xi }(x) =4 \arctan(e^{x-\xi}) 
 \le 4 e^{-|x-\xi|}.
\end{align}
Similar lower bound holds for $x>\xi$. Recall from the proof of Lemma~\ref{LEM:MUS}
that
\begin{align*}
E& (m_{\xi_1,\dots,\xi_Q} )  - |Q| E_{\textup{kink}}
\\
&=
\int_{\R} \sum_{i<j}  \dx m_{\xi_i}  \dx m_{\xi_j} 
+ 1-Q -
\cos\Big(\sum_{j=1}^Q m_{\xi_j}\Big)
+\sum_{j=1}^Q \cos(m_{\xi_j}) 
\; dx.
\end{align*}
It is elementary to check that  
(let $\tan\theta=e^{x-\xi_i}$
and use $\sin(2\theta)=\frac{2\tan\theta}{1+(\tan\theta)^2}$)
\[
 \dx m_{\xi_i} 
 =
 \frac{4e^{x-\xi_i}}{1+e^{2(x-\xi_i)}}
 =
2\sin ( 2\arctan(e^{x-\xi_i}))
 = 2\sin(m_{\xi_i}(x) /2).
 \]
We claim that 
\begin{align*}
F_Q & (\xi_1,\cdots, \xi_Q , x)
\\
&:=
4\sum_{i<j}  
\sin(m_{\xi_i}(x) /2) \sin(m_{\xi_j}(x) /2) 
+ 1-Q-
\cos\Big(\sum_{j=1}^Q m_{\xi_j}(x)\Big)
+\sum_{j=1}^Q \cos(m_{\xi_j}(x)) \ge 0
\end{align*}
for all $x\in \R$ and all $\xi_1\le \cdots \le \xi_Q$.
In other words the energy density of 
$m_{\xi_1,\dots,\xi_Q} $ is pointwisely greater than the sum of the energy densities of the corresponding single kinks.
%(Of course this would not be true if $m_{\xi_1,\dots,\xi_Q}$ was replaced by a general profile $\phi$.)

To prove the claim, we first note that $F_1=0$. 
Suppose  we have proved that $F_{Q-1}(\xi_1,\cdots, \xi_{Q-1},x)\ge 0$ for all $x\in \R$
and all $\xi_1\le \cdots \le \xi_{Q-1}$,
and we now prove this for $F_Q$. 
Consider the case $x \le \xi_Q$ (the other possibility is $x\ge \xi_1$ which is similar).
It suffices to prove
\[
F_{Q}(\xi_1,\cdots, \xi_{Q},x) \ge
 F_{Q}(\xi_1,\cdots, \xi_{Q-1},\infty,x)
\]
since the right-hand side is equal to $F_{Q-1}(\xi_1,\cdots, \xi_{Q-1},x)$,
noting that $m_{\infty}(x)=0$.
To this end, 
writing $m_j = m_{\xi_j}(x)/2$, one has
\begin{align*}
F_{Q} & (\xi_1,\cdots, \xi_Q,x) 
-
 F_{Q}(\xi_1,\cdots, \xi_{Q-1},\infty,x)
\\
& =
\Big( 
4 \sin m_Q \sum_{j=1}^{Q-1}  \sin m_j
-
\cos (2m_Q +2 A )
+ \cos (2m_Q)
\Big)
-\Big( \cos 2A -1 \Big)
\\
& =
4 \sin m_Q \sum_{j=1}^{Q-1}  \sin m_j
-2\sin^2 m_Q
+2\sin(m_Q+2A)\sin(m_Q)
\\
& =
4 \sin m_Q
\Big(
 \sum_{j=1}^{Q-1}  \sin m_j+ \cos(m_Q+A)\sin A
\Big)
\end{align*}
where $A:= \sum_{j=1}^{Q-1} m_j$ and we have used sum-to-product trigonometric identities. Note that $ \sum_{j=1}^{Q-1}  \sin m_j \ge |\sin ( \sum_{j=1}^{Q-1} m_j)|$
for any $m_1,\cdots,m_{Q-1} \in [0,\pi]$, which holds by the induction.
Since $ \cos(m_Q+A) \in [-1,1]$, the above expression is indeed non-negative.
Thus we have proved $F_Q  (\xi_1,\cdots, \xi_Q , x)\ge 0$.

\begin{figure}[h]
\includegraphics[scale=0.7]{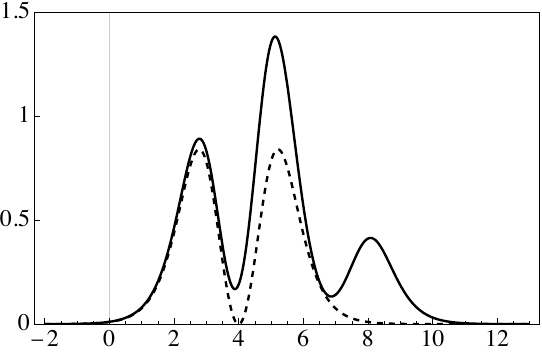}
\caption{The solid curve plots $F_3(\xi_1,\xi_2,\xi_3,\cdot)$ with $(\xi_1,\xi_2,\xi_3)=(3,5,8)$. Clearly, the main contribution to the energy gap is in the neighborhoods of $\xi_j$. The dashed curve plots $F_3(\xi_1,\xi_2,\infty,\cdot)$.}
\end{figure}

Now we need to prove that $\int_{\R} F_Q (\xi_1,\cdots, \xi_Q , x) dx \ge c\cdot e^{- d}$. To this end 
we show the following representation of $F_Q$
\begin{align}
{} & F_Q  (\xi_1,\cdots, \xi_Q , x)		\notag
\\		
&:=
4\sum_{i<j}  
\sin m_i \sin m_j
+ 1-Q 
-\cos\Big(2 \sum_{j=1}^Q m_j \Big)
+\sum_{j=1}^Q \cos(2m_j)			\notag
\\
&=
4\sum_{i<j}  
\sin m_i \sin m_j (\cos(m_i+m_j)+1)	\label{e:identityF}
\\							
&\qquad +
\sum_{n=3}^Q
2^n \sum_{i_1<\cdots<i_n}
\sin m_{i_1} \cdots \sin m_{i_n} 		\notag
\\
&
\qquad\qquad \times
\Big(
(-1)^{\frac{n-1}2} \sin( m_{i_1} + \cdots + m_{i_n} ) 1_{n\in 2\Z-1}
+
(-1)^{\frac{n-2}2} \cos( m_{i_1} + \cdots + m_{i_n} ) 1_{n\in 2\Z}
\Big).						\notag
\end{align}

\noi 
This representation may be of independent interest;\footnote{It is somehow reminiscent with Brydges--Kennedy expansion \cite{BK87} for 2D sine-Gordon.} however below we will mainly use the fact that the ``2-body interaction''  terms (which have good positive sign since $m_i\in [0,\pi]$) exhibit the desired lower bound whereas the ``$\ge 3$-body interaction'' terms
(which do not have good sign) will have much smaller absolute values   since they have more $\sin$ factors, if we observe in the right regime.

Assume that the above identity holds. 
Suppose that $|\xi_{\ell}-\xi_{\ell+1}|=d$ is the smallest distance.
Thanks to the above pointwise positivity,
it suffices to prove the lower bound for the integral  of $F_Q$ over $x\in [\xi_\ell - 1 , \xi_\ell + 1]$. Fix $x$ in this interval.

By the lower bound in \eqref{e:kink-upper-lower}, $\sin m_\ell = \sin(m_{\xi_\ell }(x)/2) \ge \sin(e^{-1}) \ge \frac1{2e}$. Also, for every $j\neq \ell$, again by  the lower bound in \eqref{e:kink-upper-lower}, 
\[
\sin m_j = \sin(m_{\xi_j }(x)/2) \ge \sin(e^{-|x-\xi_\ell|})
\ge \frac12 e^{-|\xi_j-\xi_\ell|}
\]
and 
since $|\xi_j -x| > d-1$,  $m_{\xi_j}(x)$ is close to $0$ or $2\pi$ up to an error bounded by $4e^{-(d-1)}$ using the upper bound in  
\eqref{e:kink-upper-lower}. 
Therefore, 
\[
%m_j  = m_{\xi_j}(x)/2 \le 2e^{-(d-1)} 
%\qquad \mbox{thus} \qquad
\sin m_j \le 2e^{-(d-1)}.
\]
On the other hand, 
 since $x$ is in a neighborhood around $\xi_\ell$, 
$m_{\xi_\ell}(x)/2$ is in a neighborhood of $\pi/2$. More precisely, 
there is a universal constant $c_0>0$ (as long as $d>10$) such that 
\[
\cos(m_\ell+m_j)+1 \ge c_0
\]
since $m_j$ is close to $0$ or $\pi$ up to an exponentially small error as shown above. 

Summarizing these bounds, the ``2-body interaction'' terms
\begin{align*}
4\sum_{i<j}  
\sin m_i \sin m_j (\cos(m_i+m_j)+1)
&\ge 
4 %\sum_{j\neq \ell}  
\sin m_\ell \sin m_{\ell+1} (\cos(m_\ell+m_{\ell+1})+1)
\\
&\ge 
c_2 e^{-|\xi_{\ell+1}-\xi_\ell|} = c_2 e^{-d}
\end{align*}
for  $c_2 = 4\cdot (2e)^{-1} \cdot \frac12 \cdot c_0>0$.  For the ``$n$-body interaction'' terms with $n\ge 3$, using the above bounds on $\sin m_j$,
\[
\Big| \sin m_{i_1} \cdots \sin m_{i_n} \Big|
\le (2e^{-(d-1)})^{n-1},
\]
which, even with $\sum_{n=3}^Q
2^n \sum_{i_1<\cdots<i_n}$, is much smaller than $c_2 e^{-d}$ for $d>0$ sufficiently large.

So the proof is complete once we verify the identity \eqref{e:identityF}.
We will use the following product-to-sum identities: if $n$ is even,
\[
\prod_{k=1}^n \sin \theta_k = \frac{(-1)^{n/2}}{2^{n}} \sum_{e} \cos(e_1 \theta_1+\cdots+e_n\theta_n) \prod_{j=1}^n e_j
\]
and if $n$ is odd,
\[
\prod_{k=1}^n \sin \theta_k
 = \frac{(-1)^{(n-1)/2}}{2^{n}}
 \sum_{e}
  \sin(e_1 \theta_1+\cdots+e_n\theta_n) \prod_{j=1}^n e_j
\]
where $e$ sums over $e=(e_1,\cdots,e_n) \in \{1,-1\}^n$.

%\hao{See Wikipedia: List of trigonometric identities. It says `citation needed' there, but I believe it's correct. In any case, the precise form of the ``$\ge 3$-body interaction'' terms is not extremely crucial here}

On the RHS of  \eqref{e:identityF}, by the above identity,
for $n$ odd,
\begin{align}
(- & 1)^{\frac{n-1}2}   2^n 
\sin m_{i_1} \cdots \sin m_{i_n} 
 \sin( m_{i_1} + \cdots + m_{i_n} )			\notag
\\
&=
\sum_{e}
  \sin(e_1 m_{i_1} +\cdots+e_n m_{i_n} ) 
 \sin( m_{i_1} + \cdots + m_{i_n} )
\prod_{j=1}^n e_j						\notag
\\
&=
\frac12 
\sum_{e}
\Big(\cos((1-e_1)m_{i_1}+\cdots+ (1-e_n) m_{i_n})   \notag
\\
&\qquad\qquad\qquad
-
\cos((1+e_1)m_{i_1}+\cdots+ (1+e_n) m_{i_n})\Big)
\prod_{j=1}^n e_j						\notag
\\
&=
-
\sum_{e}
\cos((1+e_1)m_{i_1}+\cdots+ (1+e_n) m_{i_n})\Big)
\prod_{j=1}^n e_j						\label{e:RHS-n-odd}
\end{align}
where the last step is by a change of variables $e_j \to -e_j$.
Similarly for $n$ even,
\begin{align*}
(- & 1)^{\frac{n-2}2}   2^n 
\sin m_{i_1} \cdots \sin m_{i_n} 
\cos( m_{i_1} + \cdots + m_{i_n} )		
\\
&=
-\frac12 
\sum_{e}
\Big(\cos((1-e_1)m_{i_1}+\cdots+ (1-e_n) m_{i_n})
\\
&\qquad\qquad\qquad
+
\cos((1+e_1)m_{i_1}+\cdots+ (1+e_n) m_{i_n})\Big)
\prod_{j=1}^n e_j									
 \end{align*}
which actually gives the same result \eqref{e:RHS-n-odd} by a change of variables.

For $n=Q$, %in either the odd or the even case, 
the  $e=(1,\cdots,1)$
term precisely gives us the term 
$-\cos(2 \sum_{j=1}^Q m_j )$ on the LHS of  \eqref{e:identityF}.
Consider without loss of generality 
$\cos(2m_1+\cdots+2m_k)$ for $2\le k<Q$. 
This term shows up once in the case $n=Q$ and $e=(1,\cdots,1,-1,\cdots,-1)$
with $\prod_j e_j =(-1)^{Q-k}$,
and shows up ${Q-k \choose 1}$ times in the case $n=Q-1$
 and $e=(1,\cdots,1,-1,\cdots,-1)$
with $\prod_j e_j =(-1)^{Q-k-1}$, etc.,
and finally it shows up once in the case $n=k$  and $e=(1,\cdots,1)$.
So by binomial identity $(1-1)^{Q-k}=0$ which means there is no such term 
$\cos(2m_1+\cdots+2m_k)$ appearing on the LHS.
Regarding the case $k=1$, the same argument applies except that 
we do not have the ``final'' case $n=k=1$ because we only sum over $n\ge 2$:
this precisely gives us the term $\cos(2m_1)$ on the LHS.
Then for each $n$ we also have a term with $e_j=-1$ for all $j$, and they precisely sum up to the constant $1-Q$ on the LHS of \eqref{e:identityF}. 
\end{proof}

\section{Geometry of the multi-soliton manifold}

\subsection{Approximate multi-soliton manifold}

In the previous section, we studied the multi-solitons defined on $\R$ and their properties. When restricted to the interval $[-L_\eps, L_\eps]$, where  $L_\eps=\eps^{-\frac 12+\eta} \to \infty$ as $\eps \to 0$, we define an approximating soliton 
$m^\eps_{\xi}$ and multi-solitons $m^\eps_{\xi_1,\dots,\xi_k}$,  which becomes an increasingly accurate approximation of the soliton $m_\xi$ and the multi-soliton $m_{\xi_1,\dots,\xi_k}$ on $\R$ as $\eps \to 0$.

\noi
Define $m^\eps$ to be a smooth, monotone function that coincides with $m$
on $[-\eps^{-\frac 12+2\eta}, \eps^{-\frac 12 +2\eta} ]$, where $\eps^{-\frac 12+2\eta} \ll L_\eps=\eps^{-\frac 12+\eta}$, and is extended to the constants $0$ and $2\pi$ outside a slightly larger interval:
\begin{align*}
m^\eps(x)=
\begin{cases}
m(x), \;  &x\in [-\eps^{-\frac 12+2\eta}, \eps^{-\frac 12+2\eta } ], \\
2\pi, \;  &x \ge \eps^{-\frac 12+2\eta}+1,\\
0,    \;  &x\le -\eps^{-\frac 12+2\eta}-1.
\end{cases}
\end{align*}

\noi
On the transition intervals $[\eps^{-\frac 12+2\eta}, \eps^{-\frac 12+2\eta} +1 ]$ and $[-\eps^{-\frac 12+2\eta}-1, -\eps^{-\frac 12+2\eta} ]$, we require  
\begin{align*}
m(x) \le m^\eps (x) \le 2\pi \quad \textup{and} \quad m(x) \ge m^\eps(x) \ge 0.
\end{align*}

\noi
Define the translated soliton
\begin{align}
m^\eps_\xi(x):=m^\eps(x-\xi),
\label{MGS1}
\end{align}

\noi
for $\xi \in [-\cj L_\eps, \cj L_\eps ]$, where
\begin{align}
\cj L_\eps:=L_\eps-\eps^{-\frac 12+2\eta}-1=L_\eps(1-\eps^\eta -\eps^{\frac 12-\eta}) \sim L_\eps.
\label{cjL}
\end{align}

\noi 
This choice of $\cj L_\eps$ ensures that the entire transition region of the kink $m_\xi^\eps$ (from $0$ to $2\pi$) remains inside the interval $[-L_\eps, L_\eps]$. Note that one can easily check that, for any $1\le p \le \infty$
\begin{align}
\| m_\xi -m_\xi^\eps  \|_{L^p(\R)} \les e^{-c \eps^{-\frac 12+2\eta} } \label{APM0}\\
\| \dx m_\xi - \dx m_\xi^\eps  \|_{L^p(\R)} \les e^{-c \eps^{-\frac 12+2\eta} }.
\label{APM1}
\end{align}

\noi
Therefore, as $\eps \to 0$, $m_\xi^\eps$ becomes a more and more precise approximation of the topological soliton $m_\xi$ on $\R$. We now define the multi-soliton, given by the superposition of single solitons
\begin{align}
m^\eps_{\xi_1,\dots, \xi_k}=\sum_{j=1}^k m^\eps_{\xi_j},
\label{MS0}
\end{align}

\noi
where $\xi_j \in [-\cj L_\eps, \cj L_\eps]$. When $k=Q$ with $\xi_j \in [-\cj L_\eps, \cj L_\eps]$, this definition is enough for the multi-soliton $m^\eps_{\xi_1,\dots,\xi_k}$ to do the transition from $0$ to $2\pi Q$ on the interval $[-L_\eps, L_\eps]$.   From the single-soliton bounds \eqref{APM0} and \eqref{APM1}, we immediately get the same type of exponential approximation for the multi-soliton
\begin{align}
\| m_{\xi_1,\dots,\xi_k} -m_{\xi_1,\dots,\xi_k}^\eps  \|_{L^p(\R)} \les k e^{-c \eps^{-\frac 12+2\eta} } \label{APM2}\\
\|  \dx m_{\xi_1,\dots,\xi_k} - \dx m_{\xi_1,\dots,\xi_k}^\eps  \|_{L^p(\R)} \les k e^{-c \eps^{-\frac 12+2\eta} },
\label{APM3}
\end{align}

\noi
where $1 \le p \le \infty$ and $ \xi_j \in [-\cj L_\eps, \cj L_\eps]$. 
Thanks to \eqref{APM2} and \eqref{APM3}, we can transfer all lemmas in Section~\ref{SEC:TOPSOL} from the multi-soliton $m_{\xi_1,\dots,\xi_k}$  to the approximating multi-soliton $m^\eps_{\xi_1,\dots,\xi_k}$ as $\eps \to 0$.

With the definition of the multi-soliton in \eqref{MS0}, we define the (approximate) multi-soliton manifold as follows
\begin{align}
\M_k^\eps:=\big\{m^\eps_{\xi_1,\dots,\xi_k}: -\cj L_\eps \le \xi_1\le \dots \le \xi_k \le \cj L_\eps  \big\},
\label{MAN00}
\end{align}

\noi
where $\cj L_\eps$ is defined as in \eqref{cjL}. Then $\M^{\eps}_k$ is a manifold of dimension $k$. Inside the multi-soliton manifold $\M^{\eps}_k$, we further decompose into the collision region 
\begin{align}
\M^{\eps, <d}_{k}:=\Big\{ m^\eps_{\xi_1,\dots,\xi_k}  : -\cj L_\eps \le \xi_1\le \dots  \xi_k \le \cj L_\eps  \; \; \textup{and} \; \; \min_{i \neq j} |\xi_i-\xi_j|<d \Big\},
\label{MAN0}
\end{align}

\noi
where the solitons interact with each other at distances less than $d$, and the non-collision region 
\begin{align}
\M^{\eps, \ge d}_k:&=\{ m^\eps_{\xi_1,\dots,\xi_k}: 
-\cj L_\eps \le \xi_1\le \dots  \xi_k \le \cj L_\eps \quad \textup{and} \quad  \min_{i \neq j}|\xi_i-\xi_j| \ge  d        \},
\label{MAN1} 
\end{align}

\noi
where the interaction between solitons is negligible as $d\to \infty$.
Note that $\M^{\eps, \ge d}_k$ is a closed subset of the compact set $\M^{\eps}_k$, and therefore $\M^{\eps, \ge d}_k$ is also compact.

\subsection{Tangent and normal spaces  and disintegration formula}

In this subsection, we study the geometry of the multi-soliton manifold $\M^{\eps, \ge d}_k$, its tangent and normal spaces, and a related disintegration formula.

When $k>1$, at the collisions,   $\M^{\eps}_k$ will have singularities and thus lose the smooth manifold structure.
We postpone this discussion to the end of this section (see Remark~\ref{REM:strata} and the discussion above). Therefore, in order to carry out Riemannian geometric considerations such as defining tangent and normal vectors, we restrict our analysis to the smooth non-collision manifold $\M_k^{\eps, \ge d}$ by removing the collision region. % $\min_{i \neq j}|\xi_i-\xi_j|<d$.

At each point $m^\eps_{\xi_1,\dots,\xi_k}$ on the multi-soliton manifold $\M^{\eps, \ge d}_k$, we define the normal space to the manifold as follows
\begin{align}
V_{\xi_1,\dots,\xi_k}=\big\{ v \in L^2: \;  \jb{v,  \dd_{\xi_j}m^\eps_{\xi_1,\dots,\xi_k}  }=0 
\quad \mbox{for all } 1\le j \le k
%\notag \\
%&\vdots \notag \\  
%\jb{v, & \dd_{\xi_k}m^\eps_{\xi_1,\dots,\xi_k}   }=0 
\big\}.
\end{align}

\noi 
Here the normal space $V_{\xi_1,\dots,\xi_k}$ at $m^\eps_{\xi_1,\dots, \xi_k}$ is a subspace of codimension $k$ in $L^2$, orthogonal to the tangent vectors 
\begin{align*}
\partial_{\xi_1} m^\eps_{\xi_1,\dots,\xi_k},\dots, \partial_{\xi_k} m^\eps_{\xi_1,\dots,\xi_k}
\end{align*}

\noi
of the multi-soliton manifold $\M^{\eps}_k$. 
Thanks to Lemma \ref{LEM:sol}-(iii), each tangent vector $\partial_{\xi_j} m^\eps_{\xi_1,\dots,\xi_k}$ is highly localized around its center $\xi_j$ with an exponentially decaying tail. Therefore, 
\begin{align}
|\jb{\partial_{\xi_j} m^\eps_{\xi_1,\dots,\xi_k}, \partial_{\xi_i} m^\eps_{\xi_1,\dots,\xi_k}  }| \les e^{-c |\xi_i-\xi_j| }
\label{ALO}
\end{align}

\noi
for $i\neq j$. This implies that the tangent vectors are almost orthogonal when they do not collide, that is, when $\min_{i\neq j}|\xi_i-\xi_j| \to \infty$.

We are now ready to define the projection map $\pi^\eps$ onto the multi-soliton manifold, introduced in Theorem \ref{THM:2}. Let $\M$ be a compact manifold in a Hilbert space $\mathcal{H}$. If $\dl>0$ is sufficiently small, we can assign to any $\phi \in \mathcal{H}$ with $\text{dist}(\phi, \M) < \dl$ a unique closest point $\pi(\phi)$ in the manifold $\M$. This follows from the $\eps$-neighborhood theorem \cite[p.69]{GP}. If  $\text{dist}(\phi,\M) \ge \dl$, then we set $\pi(\phi)=0$.  Recall that 
% \begin{align*}
% \M^{\eps, \ge d}_k:&=\{ m^\eps_{\xi_1,\dots,\xi_k}: 
% -\cj L_\eps \le \xi_1\le \dots  \xi_k \le \cj L_\eps \quad \textup{and} \quad  \min_{i \neq j}|\xi_i-\xi_j| \ge  d        \}
% \end{align*}
$\M^{\eps, \ge d}_k$ defined in \eqref{MAN1} 
is a $k$-dimensional compact manifold. Therefore, if a field $\phi$ satisfies 
\begin{align}
\text{dist}(\phi, \M^{\eps, \ge d}_k)=\inf_{\substack{ -\cj L_\eps \le \xi_1\le \dots  \xi_k \le \cj L_\eps  \\ \min_{i \neq j } |\xi_i-\xi_j| \ge d   } } \|\phi-m^\eps_{\xi_1,\dots,\xi_k}\|_{L^2}<\dl
\end{align}

\noi
for sufficiently small $\dl>0$,  we can assign a unique pair $(\xi_1,\dots,\xi_k)\in [-\cj L_\eps, \cj L_\eps]^k$  such that 
\begin{align}
\phi=m^\eps_{\xi_1,\dots,\xi_k}+v,
\label{APM4} 
\end{align}

\noi 
where the normal coordinate $v\in V_{\xi_1,\dots,\xi_k}$ satisfies $\|v \|_{L^2}<\dl$. Therefore, according to \eqref{APM4}, we are now able to define the projection $\pi^\eps$ onto the multi-soliton manifold $\M^{\eps, \ge d}_k$ as the closest point 
\begin{align}
\pi^\eps(\phi)=
\begin{cases}
m^\eps_{\xi_1,\dots,\xi_k}, \; 
&\text{dist}(\phi, \M^{\eps, \ge d}_k) <\dl \\
0, \; & \text{dist}(\phi, \M^{\eps, \ge d}_k) \ge \dl.
\end{cases}
\label{APM5}
\end{align}

\begin{remark}\rm 
For sufficiently large $d>0$, we can define a projection map onto the (non-approximating) manifold 
\begin{align*}
\M^{\ge d}_k:&=\{ m_{\xi_1,\dots,\xi_k}: 
-\infty< \xi_1\le \dots  \xi_k <\infty \quad \textup{and} \quad  \min_{i \neq j}|\xi_i-\xi_j| \ge  d        \}
\end{align*}

\noi 
with a uniform neighborhood of size $\dl>0$ since the Jacobian matrix has the uniform lower bound
\begin{align*}
\textup{det}(\jb{\partial_{\xi_i}m_{\xi_1,\dots,\xi_k}, \partial_{\xi_j}m_{\xi_1,\dots,\xi_k}   }_{1\le i,j \le k} ) \ges \| \dx m\|_{L^2(\R)}^k(1+O(e^{-cd}))
\end{align*}

\noi
as $d \to \infty$. Note that this allows us to apply the implicit function theorem with uniform control.

\end{remark}

%In order to study the marginal distribution of the multi-soliton centers
%$(\xi_1,\dots,\xi_k)$, we also define the marginal tangential projection. 
%Note that since solitons $m^\eps_{\xi_1},\dots, m^\eps_{\xi_k}$ are indistinguishable, the correct (marginal) tangential projection is the ordered one as follows
%\begin{align}
%\pi^\eps_{(j)}(\phi)=
%\begin{cases}
%\xi_{(j)}, \; &\text{dist}(\phi, \M^{\eps, \ge d}_k) <\dl \\
%0, \; & \text{dist}(\phi, \M^{\eps, \ge d}_k) \ge \dl,
%\end{cases}
%\label{APM6}
%\end{align}

%\noi
%where $\xi_{j}$ denotes the $j$-th ordered center in the increasing rearrangement $\xi_{(1)}<\cdot<\xi_{(k)}$.

%$\text{dist}(\phi, \M^{\eps, \ge d}_k) <\dl $

We now introduce a disintegration formula from \cite[Lemma 3]{ER2}, which expresses Gaussian functional integrals on a small neighborhood of the compact manifold $\M^{\eps, \ge d}_k$, defined in \eqref{MAN1}, in terms of tangential $\xi_1,\dots,\xi_k$ and normal coordinates $v$.

\begin{lemma}\label{LEM:chan}
Let $F$ be a bounded, continuous function on $L^2$. Then, we have
\begin{align}
&\int_{ \{ \textup{dist}(\phi, \mathcal{M}_k^{\eps, \ge d} ) < \dl    \} }F(\phi) \mu_{\eps}^k(d\phi) \notag \\
&=\int \dots \int \limits_{ U_\eps } F( m_{\xi_1,\dots,\xi_k}^\eps+v   )e^{-\frac 12 \| \dx m^\eps_{\xi_1,\dots,\xi_k}  \|_{L^2 }^2  -\jb{ (-\dx^2) m^\eps_{\xi_1,\dots, \xi_k}, v }_{L^2} } \notag  \\
&\hphantom{XXXXXX} \cdot 
\textup{Det}_{\xi_1,\dots,\xi_k}(v) \, \mu^{\perp}_{\eps, \xi_1,\dots,\xi_k}(dv) \, d\s(\xi_1,\dots,\xi_k),
\label{cha0}
\end{align}

\noi
where $\M^{ \eps, \ge d}_k$ is as defined in \eqref{MAN1}, $\Dl_k=\{ -\cj L_\eps \le \xi_1 \le \cdots \le \xi_k \le \cj L_\eps  \}$,
\begin{align}
&U_\eps=\big\{ ( \xi_1,\dots, \xi_k, v )\in  \Dl_k \times V_{\xi_1,\dots,\xi_k}: \|v \|_{L^2 }< \dl \; \textup{and} \; \min_{i \neq j}|\xi_i-\xi_j| \ge d   \big\} \notag \\
&\textup{Det}_{\xi_1,\dots,\xi_k}(v):=\det\big( \Id-W_{\xi_1,\dots \xi_k, v }\big),
\label{wein1}
\end{align}

\noi
and $\mu^{\perp}_{\eps, \xi_1,\dots,\xi_k }$ is the Gaussian measure on the normal space $V_{\xi_1,\dots,\xi_k}$ with covariance $\eps(-\dx^2)^{-1}$, subject to Dirichlet boundary conditions on $[-L_\eps, L_\eps]$
\begin{align*}
\mu^{\perp}_{\eps, \xi_1,\dots,\xi_k }(dv)=Z_{\eps,\xi_1,\dots,\xi_k}^{-1} e^{-\frac 1{2\eps} \int_{-L_\eps}^{L_\eps} |\dx v|^2 dx  }\prod_{x \in [-L_\eps, L_\eps ]} dv(x). 
\end{align*}

\noi
In \eqref{wein1} $W_{\xi_1, \dots, \xi_k, v}$ denotes the Weingarten map, defined in \eqref{Wein}. In addition,  $d\s$ is the surface measure on the manifold, parametrized by $(\xi_1,\dots, \xi_k)\in  [- \cj L_\eps, \cj L_\eps]^k$, as defined in \eqref{surfm}.

%with $ V_{\xi_1,\dots,\xi_k} \cap   H^1([-L_\eps, L_\eps]) \subset  H^1( [-L_\eps, L_\eps]  )$ as its Cameron-Martin space.  
\end{lemma}

\begin{proof}
The formula \eqref{cha0} follows from \cite[Lemma 3]{ER2}.
\end{proof}

\begin{remark}\rm \label{REM:disinteg}
The disintegration formula has been widely employed in recent works \cite{OST1, SS24, GT24} in settings where the energy functional admits explicit minimizers and the analysis is carried out around a single-soliton manifold. However, to the best of our knowledge,  this is the first time the disintegration formula is applied in a regime where no minimizers exist and the analysis is performed around a multi-soliton manifold. See also  Proposition \ref{PROP:cha}.
\end{remark}

Note that the geometry of the smooth soliton manifold $\M^{\eps, \ge d}_k$ is reflected in the surface measure $d\s$ and the Weingarten map $W_{\xi_1, \dots, \xi_k, v}$.  The orthonormal vectors $t_j = t_j(\xi_1,\dots,\xi_k)$, for $j = 1, \dots, k$, are obtained by applying the Gram-Schmidt orthonormalization procedure in $L^2$ to the tangent vectors
\begin{align*}
\big\{ \partial_{\xi_1}m^\eps_{\xi_1,\dots,\xi_k}, \dots, \partial_{\xi_k} m^\eps_{\xi_1,\dots, \xi_k} \big\}
\end{align*} 

as follows 
\begin{align}
t_j:=\frac{w_j}{\|w\|_{L^2} } \qquad \textup{and} \qquad w_j:=\partial_{\xi_j} m^\eps_{\xi_1,\dots, \xi_k}-\sum_{\ell<j} \jb{\partial_{\xi_j} m^\eps_{\xi_1,\dots, \xi_k}, t_\ell}_{L^2} \;  t_\ell.
\label{GRAM}
\end{align}

\noi 
Then, the surface measure $d\s(\xi_1,\dots,\xi_k)$ is given by 
\begin{align}
d \s (\xi_1,\dots,\xi_k)=|\g(\xi_1,\dots,\xi_k)| d \xi_1 \dots   d \xi_k,
\label{surfm}
\end{align}

\noi
where
\begin{align*}
\g(\xi_1, \dots, \xi_k)=\det 
\begin{pmatrix}
\langle \partial_{\xi_1 }m^\eps_{\xi_1,\dots, \xi_k}, t_1\rangle_{L^2}  & \cdots & \langle \partial_{\xi_k } m^\eps_{\xi_1,\dots, \xi_k} ,t_1\rangle_{L^2} \\
\vdots & \ddots & \vdots \\
\langle \partial_{\xi_1}m^\eps_{\xi_1,\dots,\xi_k},t_k\rangle_{L^2} & \cdots & \langle \partial_{\xi_k} m^\eps_{\xi_1,\dots,\xi_k}, t_k \rangle_{L^2}
\end{pmatrix}.
\end{align*}

\noi 
Indeed, thanks to \eqref{ALO}, the tangent vectors $\{ \partial_{\xi_j} m^\eps_{\xi_1,\dots,\xi_k}  \}_{j=1}^k$ are almost orthogonal, with an error of order $e^{-d}$, where $ \min_{i \neq j}|\xi_i-\xi_j| \ge d$. Hence, the Jacobian matrix is nearly diagonal, with diagonal entries given by $\| \dx m \|_{L^2(\R)}^2$, 
\begin{align}\label{e:Jac}
|\g(\xi_1,\dots,\xi_k)|=k \| \dx   m \|_{L^2(\R)}^2(1+O(e^{-cd}) ),
\end{align}

\noi
uniformly in $\xi_1,\dots \xi_k$ satisfying $\min_{i \neq j }|\xi_i-\xi_j| \ge d \to \infty $.

In \eqref{wein1} the Weingarten map $W_{\xi_1,\dots,\xi_k,v}$ encodes the curvature of the surface $\mathcal{M}^{\eps, \ge d}_k$ by capturing how the normal vector $v$ changes direction as we move along different tangent directions on the surface. 
More precisely, the Weingarten map $W_{\xi_1,\dots,\xi_k,v}=-dN_{\xi_1,\dots,\xi_k}(v)$ at a point $m^\eps_{\xi_1,\dots,\xi_k} \in \M^{k,\eps}$, defined via the differential of the Gauss map $N_{\xi_1,\dots,\xi_k }$ at $m^\eps_{\xi_1,\dots,\xi_k} \in \M^{\eps, \ge d}_k$, is the linear map 
\begin{align}
W_{\xi_1,\dots,\xi_k,v}: T_{\xi_1, \dots, \xi_k} \M^{\eps, \ge d}_k \to T_{\xi_1, \dots, \xi_k} \M^{\eps, \ge d}_k, 
\label{Wein} 
\end{align}

%\partial_{\xi_1} m^\eps_{\xi_1,\dots,\xi_k}, \cdots, \partial_{\xi_k} m^\eps_{\xi_1,\dots,\xi_k}

%% and $t_1,\dots,t_k$ are obtained from the Gram–Schmidt orthonormalization described above (see \eqref{GRAM}). 

\noi 
where $T_{\xi_1, \dots,\xi_k}  \M^{\eps, \ge d}_k:=\text{span}\big\{ t_1,\dots,t_k \big\}$ is the tangent space of $ \M^{\eps, \ge d}_k$ at $m^\eps_{\xi_1,\dots,\xi_k}$. 
Specifically, the Weingarten map $W_{\xi_1, \dots,\xi_k, v}$ in the basis $\{ t_1, \dots, t_k  \}$ is given by
\begin{align*}
\begin{pmatrix}
\langle -\partial_{\xi_1} N_{\xi_1,\dots,\xi_k}(v),  t_1 \rangle_{L^2}  & \cdots & \langle -\partial_{\xi_1} N_{\xi_1,\dots,\xi_k}(v),  t_k  \rangle_{L^2} \\
\vdots & \ddots & \vdots \\
\langle- \partial_{\xi_k} N_{\xi_1,\dots \xi_k }(v),  t_1    \rangle_{L^2} & \cdots & \langle - \partial_{\xi_k} N_{\xi_1,\dots \xi_k }(v)  , t_k   \rangle_{L^2}
\end{pmatrix}.
\end{align*}

\noi 
In particular, the $k \times k$ determinant 
\begin{align}
\det\big( \Id-W_{\xi_1,\dots,\xi_k,  v}\big)=1+O\big(\| v\|_{L^\infty  }^k \big)
\end{align}

\noi 
is a $k$-th order function of  $v$. 
% \hao{What does the next sentence mean? As I understand, in general a Gauss map $N$ is from a submanifold $\M$ to the Grassmannian defined by $N(p)=T_p \M$ for $p\in \M$. With this, I think in the previous paragraph I'm able to figure out why given a normal vector $v$ the map $dN_{\xi_1, \dots,\xi_k}(v)$ is from the tangent to the tangent. However, when we write $N_{\xi_1,\dots,\xi_k}(v)$ below, where I suppose $v$ is a normal vector, what do we mean?} 
% \begin{equation*}
% L^2 \ni  v \mapsto N_{\xi_1,\dots,\xi_k}(v)
% \end{equation*}

% \noi
% is a parametrization of the normal space to $ \M^{\eps, \ge d}_k$ at a point $m^\eps_{\xi_1,\dots,\xi_k}$.  

Before concluding this subsection, we present the following lemma, which will be used later.
\begin{lemma}\label{LEM:TGRAM}
Let $t_1,\dots, t_k$ be obtained by applying the Gram–Schmidt orthonormalization in $L^2$ to $\partial_{\xi_1}m^\eps_{\xi_1,\dots,\xi_k}, \dots, \partial_{\xi_k} m^\eps_{\xi_1,\dots, \xi_k}$, as described in \eqref{GRAM}. Then, for each $j$, we have 
\begin{align*} 
|t_j(x)| \les e^{-c|x-\xi_j|}
\end{align*}

\noi
provided that $\min_{i \neq j}|\xi_i-\xi_j| \ge d$ for $d$ sufficiently large.
\end{lemma}

\begin{proof}
Note that $t_1$ is localized at $\xi_1$ with an exponentially decaying tail. Inductively assume $t_\ell$ are localized at $\xi_\ell$ for $\ell<j$. Then $|\jb{\partial_{\xi_j} m^\eps_{\xi_1,\dots, \xi_k}, t_\ell   }| \les e^{-c|\xi_j-\xi_\ell|}$ and so 
\begin{align*}
|\jb{\partial_{\xi_j} m^\eps_{\xi_1,\dots, \xi_k}, t_\ell   } t_\ell(x)|  \les e^{-c|\xi_j-\xi_\ell|} e^{-c_1|x-\xi_\ell| } \les e^{-c\min_{j \neq \ell} |\xi_j-\xi_\ell| }. 
%\les e^{-c^*|x-\xi_j|},
\end{align*}

%where $c^*=\min\{c,c_1\}$ and we used $|x-\xi_\ell| \ge |x-\xi_j|-|\xi_j-\xi_\ell|$. 
\noi 
This, together with \eqref{GRAM}, implies that
\begin{align*}
|w_j(x)|&\le |\partial_{\xi_j} m^\eps_{\xi_1,\dots, \xi_k}(x)|+\sum_{\ell<j} |\jb{\partial_{\xi_j} m^\eps_{\xi_1,\dots, \xi_k}, t_\ell}_{L^2}| |  t_\ell(x)| \les e^{-c|x-\xi_j|}
%\\&\les e^{-c|x-\xi_j|}+\sum_{\ell<j} e^{-c^* |x-\xi_j|} \cdot e^{-c|x-\xi_\ell|}
\end{align*}

\noi
as $\min_{j \neq \ell}|\xi_j-\xi_\ell| \ge d \to \infty$.

\end{proof}

%\begin{align*}
%\jb{\partial_{\xi_i} m_{\xi_1,\dots,\xi_k},  \partial_{\xi_j} m_{\xi_1,\dots,\xi_k}  }=4 \int_{\R} \text{sech}(x-\xi_i)\text{sech}(x-\xi_j) dx \les e^{-\al| \xi_i-\xi_j |}.
%\end{align*}

%\noi
%Hence, the tangent vectors are pairwise almost orthogonal when the minimal separation $\min_{i \neq j}|\xi_i-\xi_j|\to \infty$.

We conclude this section with some discussion on the singular geometry in the collision regime,
since understanding the structure of the collision manifold is of independent interest from a geometric perspective.

When $k=2$, as $|\xi_1-\xi_2|$ becomes small, we lose the almost orthogonality \eqref{ALO} of the tangent vectors 
$\partial_{\xi_1} m^\eps_{\xi_1,\xi_2}$ and $\partial_{\xi_2} m^\eps_{\xi_1,\xi_2}$.
The Jacobian decreases rank by $1$ when $\xi_1=\xi_2$.
Writing in the coordinates $\xi_1=\xi-s$ and $\xi_2=\xi+s$  where $\xi$ is the ``center of mass'', the differentiation of the map $(\xi_1,\xi_2)\mapsto  m^\eps_{\xi_1,\xi_2} \in L^2$ 
in the $s$-direction vanishes at the boundary.

\begin{figure}[h]
\centering
\adjustbox{valign=t}{\includegraphics[scale=0.3]{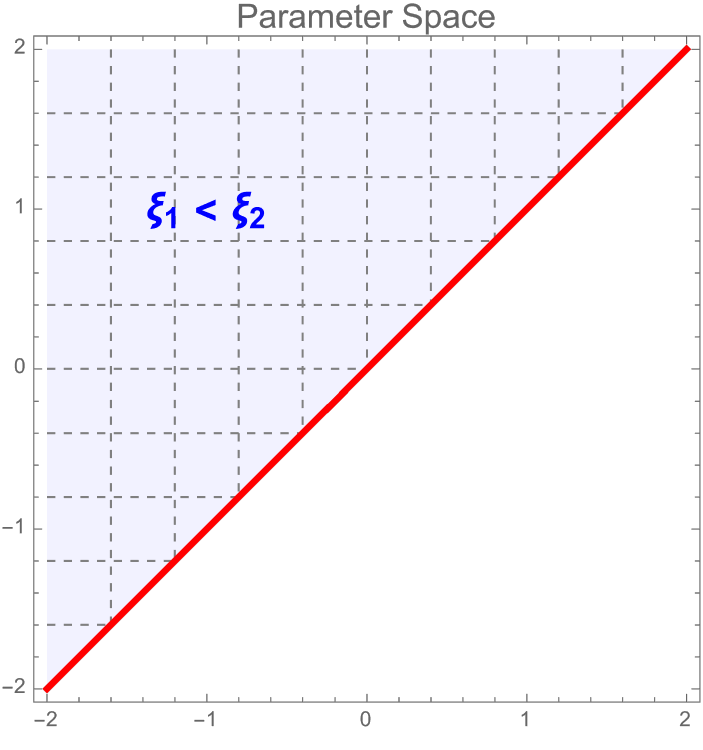}}
\hspace{0.8em}
\raisebox{-10ex}{$\rightarrow$}
\hspace{0.8em}
\adjustbox{valign=t}{\includegraphics[scale=0.4]{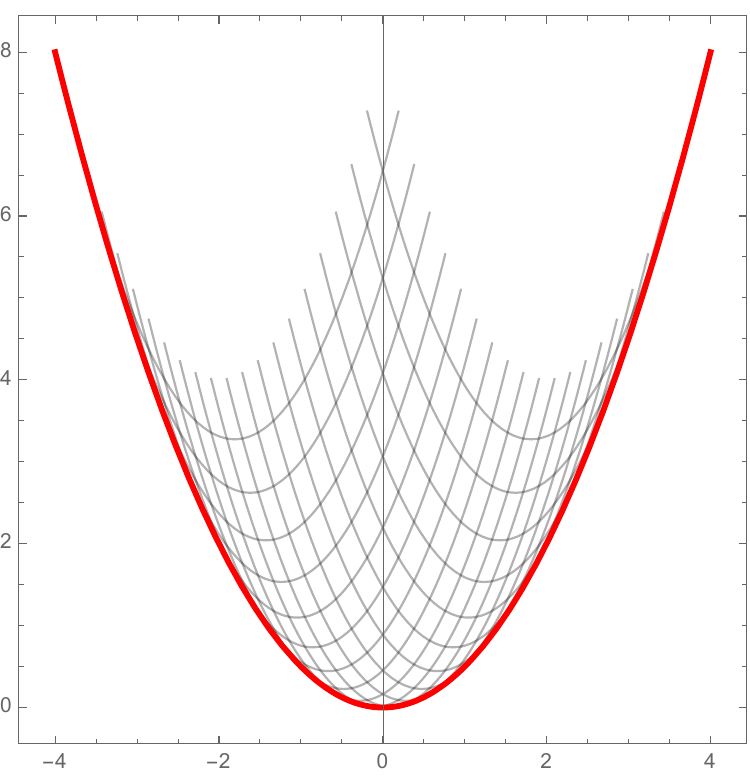}}
\caption{A map $\{(\xi_1,\xi_2):\xi_1<\xi_2\}\to \R^2$ where tangent vectors are almost orthogonal in the bulk, but asymptotically collinear near the boundary. The two families of curves in the bulk are lines of constant $\xi_1$ and $\xi_2$. Picture drawn with a toy model $(\xi_1,\xi_2) \mapsto (\xi_1+\xi_2,\xi_1^2+\xi_2^2)$ to illustrate $(\xi_1,\xi_2) \mapsto m_{\xi_1,\xi_2}$ which is not drawable in infinite dimensions.}
\end{figure}

\begin{remark}\rm\label{REM:strata}

Although for $k=2$ one may endow another differentiable structure to turn $\M_2^\eps$ and $\M_2^{\eps, <d}$ into smooth manifolds with boundaries, in general $\M_k^\eps$ and $\M_k^{\eps, <d}$ are just topological manifolds and fail to be differentiable manifolds for $k\ge 3$ due to the presence of collision strata. 

For example, when $k=3$, the parameter space $\{\xi_1\le \xi_2 \le \xi_3\}$ 
%has the geometry of a “piece of an orange.” It 
has two codimension-one faces, given by  $\xi_1 =\xi_2$ and $\xi_2=\xi_3$. Each point on these faces  admits a neighborhood that is locally diffeomorphic to $\R_{+} \times \R^2$. However, there is also an edge corresponding to $\xi_1=\xi_2=\xi_3$ and each point on this edge has a neighborhood that is locally diffeomorphic to  $\R^2_{+} \times \R$. Since $\R^2_{+}$ is the quadrant with a corner, which is homeomorphic but not diffeomorphic to the half-plane,  $\M_k$ and $\M_k^{\eps,<d}$ fail to be a differentiable manifold (or a differentiable manifold with boundary), 
although they are topological manifolds with boundaries. For $k>3$, the geometry is more complicated, as $\M_k^\eps$ and $\M_k^{\eps, <d}$ decompose into strata of different codimensions. 
\end{remark}

\section{Gaussian measures associated with Schrödinger operators}

%\nb^2 H(m_{\xi_1,\dots,\xi_k})=-\dx^2+1-2\sum_{j=1}^k \text{sech}^2(\cdot-\xi_j)+O(e^{-\min_{ i\neq j}| \xi_j-\xi_i|}),

In this section, we study the linearized operator around the multi-soliton configuration. For well-separated centers $\xi_1,\dots,\xi_k$ satisfying $\min_{i \neq j}|\xi_i-\xi_j| \ge d_\eps \to \infty$, the Hessian of the energy at $m_{\xi_1,\dots,\xi_k}=\sum_{j=1}^k m_{\xi_j} $ is
\begin{align}
\nb^2 E(m_{\xi_1,\dots,\xi_k})
&=-\dx^2+\cos(m_{\xi_1,\dots,\xi_k}) \notag 
\\
&=-\dx^2+1-2 \sum_{j=1}^k \textup{sech}^2(\cdot-\xi_j)+O(e^{-c\min_{i\neq j}|\xi_i-\xi_j| }),
\label{2ndvar}
\end{align}

\noi
where the second equality follows from the well-separated condition $\min_{i \neq j}|\xi_i-\xi_j| \ge d_\eps \to \infty$ and the structure of the multi-soliton. Note that the error term is exponentially small in the minimal separation $d_\eps$ and uniform in all $\xi_1,\dots,\xi_k$ with  $\min_{i \neq j}|\xi_i-\xi_j|  \ge d_\eps $. In Section \ref{SEC:coord}, $\nb^2 E(m_{\xi_1,\dots,\xi_k})$ plays the role of the covariance operator for a new Gaussian measure.

We begin by recalling the following well-known spectral properties of the linearized operator $\nb^2 E(m_\xi)=\mathcal{L}_\xi$ around the single kink
\begin{align}
\mathcal{L}_\xi=-\dx^2+\cos(m_\xi)=-\dx^2+1-2\text{sech}^2(\cdot-\xi),
\label{lin1}
\end{align}

\noi
where $\xi \in \R$. The potential is of the reflectionless Pöschl–Teller type.

\begin{lemma}\label{LEM:SW}
Let $\xi \in \R$.

\begin{itemize}
\item[(1)]  The linearized operator $\mathcal{L}_\xi$, defined in \eqref{lin1},  is self-adjoint. Its spectrum is given by
\begin{align*}
\s(\mathcal{L}_\xi)=\s_d \cup \s_c=\{0\}\cup [1,\infty).
\end{align*}

\smallskip

\item[(2)] The eigenfunction corresponding to the zero eigenvalue arises from the tangent vector $\partial_{\xi} m_{\xi}$ to the soliton manifold $\{ m(\cdot-\xi) \}_{\xi\in \R}$, which is associated with translation invariance
\begin{align}
\mathcal{L}_\xi (\partial_{\xi} m_{\xi})=0.
\label{L0}
\end{align}

\smallskip

\item[(3)] The linearized operator $\mathcal{L}_\xi$  satisfies the following coercivity: there exists $\ld_0 \in [1,\infty ) $ such that  
\begin{align}
\jb{\mathcal{L}_\xi v,v   } \ge \ld_0 \|v \|_{H^1}^2.
\label{AO1}
\end{align}

\noi
for every $v \in H^1(\R)$ with $\jb{v, \partial_{\xi} m_{\xi}}_{L^2(\R)}=0$.

\end{itemize}

\end{lemma}

\begin{proof}
For the proof, see \cite[Lemma 3.1]{CL}.
\end{proof}

We denote 
\begin{align}\label{linQ}
\mathcal{L}_{\xi_1,\dots,\xi_k}:=-\dx^2+1-2\sum_{j=1}^k \text{sech}^2(\cdot-\xi_j).
\end{align}

%Note that the tangent vectors $\partial_{\xi_j} m_{\xi_1,\dots,\xi_k}$ in each direction is given by

\noi
Note that the tangent vectors with respect to the center coordinates are
\begin{align*}
\partial_{\xi_j} m_{\xi_1,\dots,\xi_k}=\partial_{\xi_{j}} m_{\xi_j}=-2 \text{sech}(\cdot-\xi_j), \quad j=1,\dots,k.
\end{align*}

%\mathcal{L}_{\xi_j} (\partial_{\xi_{j}} m_{\xi_j})

\noi
This implies that 
\begin{align*}
\mathcal{L}_{\xi_1,\dots,\xi_k} (\partial_{\xi_j} m_{\xi_1,\dots,\xi_k})&=(-\dx^2+1-2\text{sech}^2(\cdot-\xi_j))\partial_{\xi_{j}} m_{\xi_j} +4 \sum_{i \neq j} \text{sech}^2(\cdot-\xi_i) \text{sech}(\cdot-\xi_j)\\
&=4 \sum_{i \neq j} \text{sech}^2(\cdot-\xi_i) \text{sech}(\cdot-\xi_j),
\end{align*}

\noi
where we used  $\mathcal{L}_{\xi_j} (\partial_{\xi_{j}} m_{\xi_j})=0$  from \eqref{L0}. Hence, the linearized operator $\mathcal{L}_{\xi_1,\dots,\xi_k}$ in \eqref{linQ} satisfies
\begin{align*}
\jb{\mathcal{L}_{\xi_1,\dots,\xi_k} (\partial_{\xi_j} m_{\xi_1,\dots,\xi_k}), \partial_{\xi_j} m_{\xi_1,\dots,\xi_k} }&= -8 \sum_{i \neq j} \int_{\R} \text{sech}^2(x-\xi_i) \text{sech}^2(x-\xi_j) dx\\
&=-O(e^{- \min_{i \neq j}|\xi_i-\xi_j|  } ), 
\end{align*}

\noi 
In the second line, we used the fact that  $\text{sech}(\cdot-\xi_i)$ is localized around $\xi_i$ with exponentially decaying tails.
Consequently, $\mathcal{L}_{\xi_1,\dots,\xi_k}$ has $k$ near-zero (negative) eigenvalues of size $O(e^{-c d_\eps})$, where $d_\eps=\min_{i \neq j} |\xi_i-\xi_j|$.

The next lemma shows that if the centers of the multi-soliton $m_{\xi_1, \dots,\xi_k}$ are sufficiently separated, then the linearized operator $\mathcal{L}_{\xi_1,\dots,\xi_k}$ is uniformly coercive on the normal space, that is, after projecting off the tangent vectors.

\begin{lemma}\label{LEM:coer2}
There exists $\ld_0, d_0>0$ such that, under the separation condition $\min_{i \neq j}|\xi_i-\xi_j| \ge d_0$, the coercivity 
\begin{align*}
\jb{\mathcal{L}_{\xi_1,\dots, \xi_k}v,v }_{L^2(\R)}  \ge \ld_0  \| v\|_{H^1(\R)}^2  
\end{align*}

\noi
holds for every $v \in H^1(\R)$  with $\jb{v,  \partial_{\xi_j} m_{\xi_1,\dots,\xi_k} }_{L^2(\R)}=0$, $j=1,\dots, k$.
\end{lemma}

\begin{proof}

Choose a smooth partition of unity such that $\sum_{j=0}^{k+1} \chi_j^2 =1$ and $\sup_j \| \partial_x \chi_j \|_{L^\infty} \les d_\eps^{-1}$, where $d_\eps=\min_{i \neq j}|\xi_i-\xi_j| \to \infty$ as $\eps \to 0$.  For $j=1,\dots,k$, $\chi_j$ is supported in a fixed neighborhood of $\xi_j$, and $\chi_0$, $\chi_{k+1}$ cover the left/right far field (vacua $0$ and $2\pi Q$). Then, we have 
\begin{align}
\jb{ \mathcal{L}_{\xi_1,\dots,\xi_k}v,v }_{L^2(\R)}&=\sum_{j} \jb{ \mathcal{L}_{\xi_1,\dots,\xi_k} (\chi_j v), \chi_j v }_{L^2(\R)} -\sum_{j} \| (\dx \chi_j) v \|_{L^2}^2 \notag \\
& \ge \sum_{j} \jb{ \mathcal{L}_{\xi_1,\dots,\xi_k} (\chi_j v), \chi_j v }_{L^2(\R)} - C d_\eps^2  \|v \|_{L^2}^2,
\label{AO00}
\end{align}

\noi
where the term $\sum_{j=1}^k \| (\dx \chi_j) v \|_{L^2}^2$ comes from expanding the kinetic term after applying a partition of unity. We analyze each localized piece to obtain a uniform coercivity estimate. Note that
\begin{align*}
\mathcal{L}_{\xi_1,\dots,\xi_k} (\chi_j v)&=\mathcal{L}_{\xi_j}(\chi_j v)-2\sum_{i:i \neq j} \text{sech}^2(\cdot-\xi_i) \chi_j v\\
&=\mathcal{L}_{\xi_j}(\chi_j v)+O(e^{-c d_\eps}) \cdot \chi_j v,
\end{align*}

%=-\dx^2+1-2\text{sech}^2(\cdot-\xi_j)$. 

\noi
where  $\mathcal{L}_{\xi_j}$ is defined as in \eqref{lin1}. This implies that 
\begin{align}
\jb{ \mathcal{L}_{\xi_1,\dots,\xi_k} (\chi_j v), \chi_j v }_{L^2(\R)}\ge \jb{ \mathcal{L}_{\xi_j}(\chi_j v), \chi_j v }_{L^2(\R)}-Ce^{-c d_\eps } \| \chi_j v \|_{L^2}^2.
\label{AO0}
\end{align}

\noi
We now show that $\chi_j v$ is almost orthogonal to $\partial_{\xi_j} m_{\xi_j}$ if $\jb{v, \partial_{\xi_j} m_{\xi_1,\dots,\xi_k}  }=0$. Indeed,
\begin{align*}
\jb{\chi_j v, \partial_{\xi_j} m_{\xi_j} }_{L^2(\R)}&=\jb{v, \partial_{\xi_j} m_{\xi_j} }_{L^2(\R)}- \sum_{i: i \neq j}\jb{v, \chi_i (\partial_{\xi_j} m_{\xi_j} )}_{L^2(\R)}\\
&=- \sum_{i: i \neq j}\jb{v, \chi_i (\partial_{\xi_j} m_{\xi_j} )}_{L^2(\R)}, 
\end{align*}

\noi
where we used  $\jb{v, \partial_{\xi_j}m}=\jb{v, \partial_{\xi_j} m_{\xi_1,\dots,\xi_k}  }=0$. Since $\partial_{\xi_j} m_{\xi_j}$ is localized around $\xi_j$ and has exponentially decaying tail, we obtain almost orthogonality
\begin{align}
|\jb{\chi_j v, \partial_{\xi_j} m_{\xi_j} }_{L^2(\R)} | \les \sum_{i: i \neq j} e^{-c d_\eps } \| \chi_i v \|_{L^2}^2.
\label{A02}
\end{align}

\noi
Combining the almost-orthogonality \eqref{A02} with the coercivity estimate \eqref{AO1} for the linearized operator $\mathcal{L}_{\xi_j}$ around the single kink $m_{\xi_j}$, we obtain
\begin{align}
\jb{ \mathcal{L}_{\xi_j}(\chi_j v), \chi_j v }_{L^2(\R)} \ge  \frac {\ld_0}2 \| \chi_j v \|_{H^1}^2 -Ce^{-c d_\eps} \sum_{i: i \neq j} \|  \chi_i v \|_{L^2}^2
\label{A03}
\end{align}

\noi
for some $\ld_0>0$. Putting the pieces together with \eqref{AO00}, \eqref{AO0}, \eqref{A03}, and taking $d_\eps \ge d_0$ large enough, we obtain 
\begin{align*}
\jb{ \mathcal{L}_{\xi_1,\dots,\xi_k}v,v }_{L^2(\R)}& \ge  \sum_{j} \jb{ \mathcal{L}_{\xi_1,\dots,\xi_k} (\chi_j v), \chi_j v }_{L^2(\R)} - C d_\eps^2  \|v \|_{L^2}^2\\
&\ge \frac {\ld_0}2 \sum_{j=1}^k \|  \chi_j v \|_{H^1}^2 -C( e^{-c d_\eps } +d_\eps^2  )\|   v \|_{L^2}^2\\
& \ge \frac {\ld_0}4 \| v\|_{H^1}^2.
\end{align*}

\noi
This completes the proof of Lemma \ref{LEM:coer2}.

\end{proof}

%\begin{align*}
%\mathcal{L}_{\xi_1,\dots,\xi_k}=-\dx^2+1-2\sum_{j=1}^k \text{sech}^2(\cdot-\xi_j)
%\end{align*}

From now on, we study the operator $\mathcal{L}_{\xi_1,\dots,\xi_k}$, defined in \eqref{linQ}, on the finite interval $[-L_\eps, L_\eps]$ with Dirichlet boundary conditions.

%Define an operator on $L^2([-L_\eps,L_\eps])$
%\begin{align*}
%B_{\xi_1,\dots,\xi_k}:=
%\end{align*}

\begin{lemma}\label{LEM:coer3}
Let $d_0>0$ be as in Lemma \ref{LEM:coer2}. Then there exits $L_0, \zeta>0$ such that if $L_\eps \ge L_0$ and $\min_{i \neq j} |\xi_i-\xi_j| \ge d_0 $, 
\begin{align*}
\jb{\mathcal{L}_{\xi_1,\dots,\xi_k } v,v    }_{L^2([-L_\eps,L_\eps])}  \ge \zeta \|v \|^2_{L^2([-L_\eps,L_\eps])}
\end{align*}

\noi
holds for every $v\in H^1_0([-L_\eps,L_\eps])$  with $\jb{v,  \partial_{\xi_j} m_{\xi_1,\dots,\xi_k} }_{L^2([-L_\eps,L_\eps])}=0$, $j=1,\dots, k$. Here $\zeta$  is independent of $L_\eps, \xi_1,\dots, \xi_k$. 

\end{lemma}

\begin{proof}
Assume not. Then there exist sequences $L_n \to \infty$, configurations $\xi_1^n,\dots, \xi_k^n$ with $\min_{i\neq j} |\xi_i^n-\xi_j^n | \ge d_0$, and $v_n \in H^1_0([-L_n,L_n])$ with $\jb{v_n,  \partial_{\xi_j^n} m_{\xi_1^n,\dots,\xi_k^n} }_{L^2([-L_n,L_n])}=0$   such that
\begin{align*}
\jb{\mathcal{L}_{\xi_1^n,\dots, \xi_k^n}v_n,v_n}_{L^2([-L_n,L_n])} \le \frac 1n \| v_n\|^2_{L^2([-L_n,L_n])}.
\end{align*}

\noi
We extend $v_n$ by zero outside $[-L_\eps, L_\eps]$ and denote the resulting function by $\wt v_n$. Then $\wt v_n \in H^1(\R)$ and 
\begin{align}
\jb{\mathcal{L}_{\xi_1^n,\dots, \xi_k^n} \wt v_n, \wt v_n}_{L^2(\R)} =
\jb{\mathcal{L}_{\xi_1^n,\dots, \xi_k^n}v_n,v_n}_{L^2([-L_n,L_n])} \le \frac 1n \| v_n\|^2_{L^2([-L_n,L_n])}= \frac 1n \|\wt v_n\|^2_{L^2(\R)},
\label{L1}   
\end{align}

\noi
where we used the Dirichlet boundary condition. Since 
\begin{align*}
\jb{\wt v_n, \partial_{\xi_j^n} m_{\xi_1^n,\dots,\xi_k^n} }_{L^2(\R)} =\jb{v_n, \partial_{\xi_j^n} m_{\xi_1^n,\dots,\xi_k^n}  }_{L^2([-L_n, L_n])}=0
\end{align*}

\noi
for $j=1,\dots,k$, Lemma \ref{LEM:coer2} implies that  
\begin{align}
\jb{\mathcal{L}_{\xi_1^n,\dots, \xi_k^n} \wt v_n, \wt v_n}_{L^2(\R)} \ge \ld_0 \| \wt v_n\|_{L^2(\R)}^2.
\label{L2}
\end{align}

\noi 
Combining \eqref{L1} and \eqref{L2} gives a contradiction for large $n$. 
Hence, a uniform $\zeta>0$ exists.

\end{proof}

According to Lemma \ref{LEM:coer3}, to avoid zero eigenvalues, we need to project off the tangential directions. Once the zero modes are removed, we can invert the operator on the subspace, namely, the normal space $V_{\xi_1,\dots,\xi_k}$.  We define the projected operator 
\begin{align}
C_{\xi_1,\dots,\xi_k}=\P_{V_{\xi_1,\dots, \xi_k} } \big(-\dx^2+1-2\sum_{j=1}^k \text{sech}^2(\cdot-\xi_j) \big)^{-1} \P_{V_{\xi_1,\dots, \xi_k} },
\label{cov}
\end{align}

\noi
viewed as an operator on the finite interval $[-L_\eps, L_\eps]$, with Dirichlet boundary conditions. Here the projection  $\P_{V_{\xi_1,\dots,\xi_k}}$ is given by
\begin{align}
\P_{V_{\xi_1,\dots,\xi_k}}
=\Id-  \sum_{j=1}^k  \jb{  \cdot,  t_j  }   t_j
=\Id-\sum_{j=1}^k \P_j,
\label{proj1}
\end{align}

\noi
where $t_1,\dots,t_k$ are obtained by applying the Gram–Schmidt orthonormalization to the tangent vectors $\partial_{\xi_j}  m^\eps_{\xi_1,\dots,\xi_k}$, $j=1,\dots,k$, as described above \eqref{GRAM}, and
\begin{align}
\P_j= \jb{\cdot, t_j  } t_j.
\label{dec00}
\end{align}

%\begin{align*}
%\partial_{\xi_j}  \cj m^\eps_{\xi_1,\dots,\xi_k}= \frac{\partial_{\xi_j} m^\eps_{\xi_1,\dots,\xi_k} }{  \| \partial_{\xi_j} m^\eps_{\xi_1,\dots,\xi_k}  \|_{L^2 }  }.
%\end{align*}

\noi
As mentioned in Section \ref{SEC:Notation}, the inner product $\jb{\cdot,\cdot}$ is understood as $\jb{\cdot, \cdot}_{L^2([-L_\eps, L_\eps])}$ unless otherwise specified.

We are now ready to define the Gaussian measure with the covariance operator $C_{\xi_1,\dots,\xi_k}$. 
\begin{lemma}\label{LEM:Gauss}
There exist sufficiently large $L_0$, $d_0>0$ such that if $L_\eps \ge L_0$ and $\min_{i\neq j} |\xi_i-\xi_j| \ge d_0$, we can define the Gaussian measure  
\begin{align*}
d\nu^\perp_{ \xi_1,\dots,\xi_k}=Z_{\xi_1,\dots,\xi_k}^{-1} e^{-\frac 12 \jb{C_{\xi_1,\dots, \xi_k}^{-1} v,v  } }\prod_{x \in [-L_\eps, L_\eps] } dv(x).
\end{align*}

\noi
for any fixed $\xi_1,\dots,\xi_k \in [-L_\eps, L_\eps]$.

\end{lemma}

\begin{proof}
The operator $\mathcal{L}_{\xi_1,\dots,\xi_k}=-\dx^2+1-2\sum_{j=1}^k \text{sech}^2(\cdot-\xi_j)$, considered on the finite interval $[-L_\eps, L_\eps]$, is a self-adjoint Sturm–Liouville operator since $2\sum_{j=1}^k \text{sech}^2(\cdot-\xi_j)$ is a Schwartz function. Hence, the operator has purely discrete spectrum $\{\ld_n\}_n$ with $\ld_n \sim |n|^2$ as $n \to \infty$. In particular, if the operator is strictly positive, its inverse belongs to the trace class. 
By Lemma \ref{LEM:coer3}, the projected operator is strictly positive, hence its inverse on that subspace is trace class. Therefore, the corresponding Gaussian measure is well defined.

\end{proof}

Before concluding this subsection, we present the structure of the partition function of the Gaussian measure in Lemma \ref{LEM:Gauss}. As the interaction between solitons becomes negligible, the spectrum of the multi-well operator $\mathcal{L}_{\xi_1,\dots,\xi_k}$ becomes exactly $k$ copies of the single-well spectrum.

\begin{lemma}\label{LEM:part} 
Let $Q\in \Z$ with $|Q|=k$. Then we have 
\begin{align*}
Z_{\xi_1,\dots,\xi_k}=(Z_0)^k(1+O(e^{-cd_\eps})),
\end{align*}

\noi
uniformly in $\xi_1,\dots,\xi_k$ satisfying $\min_{i \neq j} |\xi_i-\xi_j| \ge d_\eps \to \infty $ as $\eps \to 0$, where $Z_0$ is the partition function of the Gaussian measure associated with the single-well operator $\mathcal{L}_0=-\dx^2+1-\textup{sech}^2(x)$.
\end{lemma}

\begin{proof}

Since $\mathcal{L}_0=-\dx^2+1-\textup{sech}^2(x)$ acts on the normal space and is a Sturm–Liouville operator on the compact interval  $[-L_\eps, L_\eps]$, it has a discrete spectrum  
\begin{align}
1\le \ld_1<\ld_2<\dots<\ld_\ell<\dots \to \infty.
\label{sp}
\end{align}

\noi 
Let $\{ \phi \}_{\ell \ge 1}$ be an orthonormal eigenbasis of $\mathcal{L}_0$ on the normal space with eigenvalues $\ld_\ell \ge 1$, thanks to Lemma \ref{LEM:SW}. Since $\mathcal{L}_0$ is a bounded perturbation of $-\dx^2+1$ by
$\textup{sech}^2$,  we can show that  
\begin{align}
|\phi_\ell (x)| +|\dx \phi_{\ell}(x)| \les e^{-c|x|}.
\label{spt}
\end{align}

\noi
For each $\xi_j$, $j=1,\dots,k$, we define $\phi_{\ell, j}(x):=\phi_{\ell}(x-\xi_j)$.  We can easily verify that $\phi_{\ell,j}$ is an eigenfunction of $\mathcal{L}_{\xi_j}$, defined in \eqref{lin1}, with the same eigenvalue $\ld_\ell$.   Furthermore the translated  eigenfunctions $\phi_{\ell, j}$ act as approximate eigenfunctions for $\mathcal{L}_{\xi_1,\dots,\xi_k}$, defined in \eqref{linQ}, as follows
\begin{align*}
\jb{\phi_{\ell,j }, \mathcal{L}_{\xi_1,\dots,\xi_k} \phi_{\ell,j}  }&=\jb{\phi_{\ell,j}, (-\dx^2+1-2\textup{sech}^2(\cdot-\xi_j) ) \phi_{\ell,j}  }-2\sum_{i \neq j} \int \textup{sech}^2(x-\xi_i) |\phi_{\ell,j}(x)|^2 dx\\
&=(\ld_{\ell}+O(e^{-cd_\eps}) ) \| \phi_{\ell, j} \|_{L^2}^2,
\end{align*}

\noi
where we used \eqref{spt}.
Hence,  every single-well level $\ld_\ell$ generates a cluster of $k$ nearby eigenvalues $\ld_{\ell,1}, \dots, \ld_{\ell,k}$ of the multiwell operator $\mathcal{L}_{\xi_1 \dots,\xi_k}$
\begin{align}
\ld_{\ell,j}=\ld_{\ell,j}(\xi_1,\dots,\xi_k)=\ld_{\ell}+O(e^{-cd_\eps}), \quad j=1, \dots, k.
\label{sp3}
\end{align}

\noi
By using the fact that $\{ \phi_{\ell,j }\}_{j=1, \dots, k}$ is an almost-orthonormal family when $d_\eps$ is sufficiently large, and there exists a single-well spectral gap $\ld_{_\ell+1}-\ld_{\ell}>0$ in \eqref{sp}, we can easily check that each cluster has exactly $k$ members and the clusters are disjoint.

To describe all discrete eigenvalues for the multi-well operator $\mathcal{L}_{\xi_1, \dots, \xi_k }$, we take the product over all $\ell \ge 1$ and use \eqref{sp3}.  More precisely, for any finite $N$, 
\begin{align*}
\textup{det}_{V_{\xi_1,\dots,\xi_k}}(\P_N\mathcal{L}_{\xi_1,\dots,\xi_k}\P_N)&= \prod_{\ell=1}^N  \bigg( \prod_{j=1}^k \ld_{\ell, j} \bigg)= \bigg( \prod_{\ell=1}^N \ld_\ell \bigg)^k(1+O(e^{-cd_\eps}))\\
&=\textup{det}_{V_0}(\P_N \mathcal{L}_0 \P_N )^k(1+O(e^{-cd_\eps})),
\end{align*}

\noi 
uniformly in $\xi_1,\dots,\xi_k$ satisfying $\min_{i \neq j} |\xi_i-\xi_j| \ge d_\eps $, where  $\P_N$ is the finite-dimensional projection onto the modes with $|\ell| \le N$. 

\end{proof}

%Here, 
%\[\mathbf{P}_{V^L_{x_0,\theta}}:L^2\rightarrow \dot{H}^1(\T_L)\]
%\noi 
%denotes the extension of the projection onto $V^L_{x_0,\theta} \cap \dot{H}^1(\T_L)$. 

%\section{Concentration around the multi-soliton manifold without collisions}
\section{Proof of Theorem \ref{THM:1} (concentration around non-collision manifold)}
\label{SEC:LDP}
In this section, we prove Theorem \ref{THM:1}.
Note that 
\[
\rho_\eps^Q  \big(\{   \textup{dist}(\phi, \M_Q) \ge \dl   \} \big)=
\frac{Z_\eps^Q \Big[ \ind_{ \{ \textup{dist}(\phi, \M_Q ) \ge \dl \}    } \Big] }{ Z_\eps^Q[1]}
\]
where
\[
Z_\eps^Q[g] := \E_{\mu^Q_\eps }\bigg[\exp\Big\{-\frac 1\eps \int_{-L_\eps}^{L_\eps} (1-\cos \phi) dx \Big \} g \bigg] .
\]
This implies 
\begin{align}
\eps \log \rho_\eps^Q  \big(\{   \textup{dist}(\phi, \M_Q) \ge \dl   \} \big) 
= \eps \log  Z_\eps^Q \Big[\ind_{ \{ \textup{dist}(\phi, \M_Q ) \ge \dl \}  }\Big]
 -\eps \log Z_\eps^Q[1].
\label{LDP1}
\end{align}
Similarly,
\begin{align}
\eps \log \rho_\eps^Q  \big(\{   \textup{dist}(\phi, \M_Q^{<d} ) < \dl  \big)
= \eps \log  Z_\eps^Q \Big[\ind_{ \{ \textup{dist}(\phi, \M_Q^{<d} ) < \dl \}  }\Big]
 - \eps \log Z_\eps^Q[1].
\label{LDP1s}
\end{align}

%we will use  the Boué--Dupuis formula, Lemma \ref{LEM:BD},  to

Below we prove the behavior of the free energy $\eps \log Z_\eps^Q[1]$ as $\eps \to 0$, and also obtain upper and lower bounds on $\eps \log Z_\eps^Q[g]$, with $g$ being the respective indicator functions in \eqref{LDP1} and \eqref{LDP1s}.

%\begin{comment}
%In this section, we analyze how the Gibbs measure $\rho^Q_\eps$ associated with the homotopy class $\mathcal{C}_Q$ concentrates around the multi-soliton manifold $\M_Q$, defined in \eqref{man0}, as follows 
%\begin{align*}
%\limsup_{\eps \to 0} \eps \log \rho_\eps^Q  \big(\{   \textup{dist}(\phi, \M_Q ) \ge \dl   \} \big) \le -c \dl^2,
%\end{align*}
%
%\noi
%where $\dl>0$. Here $\textup{dist}$ denotes the $L^2(\R)$-distance since any field $\phi$ distributed according to $\rho^Q_\eps$ is interpreted as a function on $\R$, extended trivially by $0$ and $2\pi |Q|$ outside $[-L_\eps, L_\eps]$. Furthermore, even if a field $\phi$ is close to the multi-soliton manifold, when the solitons collide, that is, when they are not well separated, the collision region becomes a large-deviation event:
%\begin{align*}
%\limsup_{\eps \to 0} \eps \log \rho_\eps^Q  \big(\{   \textup{dist}(\phi, \M_Q^{<d} ) < \dl  \big) \le - ce^{- d},
%\end{align*}
%
%\noi
%where $\M^{<d}_Q$ is defined in \eqref{man1}. As mentioned in Section \ref{SEC:Notation}, in the following we assume $Q>0$.
%For $Q<0$, we define the corresponding multi-soliton manifold with anti-kinks in an analogous manner.
%\end{comment}

For the free energy,
by  Lemma \ref{LEM:BD}, we write  
\begin{align}
\eps\log Z^Q_\eps[1]
&=\sup_{ \dr  \in  \Ha } 
\E\bigg[-  \int_{-L_\eps}^{L_\eps} \big( 1-\cos( \ell^Q+ \sqrt{\eps}Y+  \sqrt{\eps}\Dr )  \big)dx 
-\frac{\eps}2 \int_0^1 \| \dr(t) \|_{L^2_x}^2 dt  \bigg] \notag
\\
&= \sup_{ \dr \in \Ha} \E\bigg[-\int_{-L_\eps}^{L_\eps} ( 1-\cos(\sqrt{\eps}Y +\ell^Q +\Dr )  )dx - \frac 12  \int_0^1 \| \dr(t) \|_{L^2_x}^2 dt   \bigg] \label{e:applyBD},
\end{align}
where we used the notation $Y:=Y(1)$ and $\Dr:=\Dr(1)$ for convenience,
and in the second line we  perform the change of variables $\sqrt{\eps} \dr \to \dr$ which does not affect the variational problem.

Below we prove an upper bound on $\eps \log Z_\eps^Q[1]$  in terms of the minimal energy (see \eqref{UPP1}). To this end, we suitably expand the potential in the fluctuation $Y$ as in \eqref{F1}. However, in a naive expansion of the potential,
one would obtain, at the first order 
\begin{align}
\E \int_{-L_\eps}^{L_\eps} |\sqrt{\eps}Y| dx \sim \eps^{\frac12} L_\eps^{\frac32}
\label{FL1}
\end{align}
in view of \eqref{1M} in Lemma \ref{LEM:Moment}.
So for the fluctuation \eqref{FL1} to vanish in order to get a bound merely in terms of the minimal energy,  one would eventually have to restrict the interval size to $L_\eps=\eps^{-\frac 13+\eta}$,
which is much unnatural and non-optimal.
For the upper bounds on $\eps \log Z_\eps^Q[g]$ with $g$ being the respective indicator functions,
we would run into the same problem.

However, below we show that by exploiting a simple but crucial specific structure of the sine–Gordon model, we are actually able to extend the interval up to $L_\eps=\eps^{-\frac 12+\eta}$. See \eqref{F1}, \eqref{F77}, and \eqref{LDP4}.

\begin{proposition}\label{PROP:free1}
Let $Q \in \Z$ and $L_\eps=\eps^{-\frac 12+\eta}$. Then, we have 
\begin{align*}
\lim_{\eps \to 0}\eps \log Z_\eps^Q[1]=-\inf_{\phi \in \mathcal{C}_Q} E(\phi),
\end{align*}

\noi
where $\mathcal{C}_Q$  is the topological sector in \eqref{HT1} 
%and $Z^Q_\eps$ is the partition function of the Gibbs measure $\rho^Q_\eps$ defined in \eqref{Gibbs1}.
%\kihoon{I tried to check. $Z_\eps^Q$ only appears in Section 6. In later parts, we mainly use the partition function $Z_\eps$ in Proposition \ref{PROP:cha}. When it is used, I tried to mention Proposition \ref{PROP:cha}}

\end{proposition}

\begin{proof}
In this proof we write $Z_\eps^Q=Z_\eps^Q[1]$.
We first derive an upper bound.  By \eqref{e:applyBD} and \eqref{DRE}
\begin{align}
\eps \log Z^Q_\eps
&\le \sup_{ \Dr \in \dot{\mathbb{H}}^1  }  \E\bigg[-\int_{-L_\eps}^{L_\eps} ( 1-\cos(\sqrt{\eps}Y +\ell^Q +\Dr )  )dx - \frac 12  \| \dx \Dr \|_{L^2}^2    \bigg],
\label{F0}
\end{align}

\noi
where in the last line we used \eqref{CY1} and $\dot{\mathbb{H}}^1$ denotes the space of $\dot{H}^1$-valued random variables. We now take the change of variable 
\begin{align}
W=\ell^Q+\Dr.
\label{F00}
\end{align}

\noi
Combining \eqref{F0} and \eqref{F00} yields  
\begin{align}
\eps \log Z_\eps^Q&\le \sup_{ W \in  \ell^Q+\dot{\mathbb{H}}^1  } \E\bigg[-\int_{-L_\eps}^{L_\eps} ( 1-\cos(\sqrt{\eps}Y+W )  )dx - \frac 12  \| \dx W \|_{L^2}^2  \notag \\
&\hphantom{XXXXXXXXXXXX}-\frac 12 \| \dx \ell^Q \|_{L^2}^2 +\int_{-L_\eps}^{L_\eps}  \dx W \cdot \dx \ell^Q dx   \bigg].
\label{F5}
\end{align}

\noi
By using Young's inequality, for any small $\dl>0$ we have 
\begin{align}
\bigg| \int_{-L_\eps}^{L_\eps}  \dx W \cdot \dx \ell^Q dx \bigg| \le \frac \dl2 \|\dx W \|_{L^2}^2+C_\dl \| \dx \ell^Q \|_{L^2}^2,
\label{F6}
\end{align}

\noi
where $C_\dl$ is a large constant. By plugging \eqref{F6} into \eqref{F5}, we have
\begin{align}
\eps \log Z_\eps^Q&\le \sup_{ W \in  \ell^Q+\dot{\mathbb{H}}^1  } \E\bigg[-\int_{-L_\eps}^{L_\eps} ( 1-\cos(\sqrt{\eps}Y+W )  )dx - \frac {1-\dl}2  \| \dx W \|_{L^2}^2    \bigg]+C_\dl  \| \dx \ell^Q \|_{L^2}^2.
\label{F7}
\end{align}

\noi
By taking the Taylor expansion around $W$, we have 
\begin{align}
\cos(\sqrt{\eps}Y+W)=\cos(W)-\sin(W) \cdot (\sqrt{\eps}Y)+O\big( (\sqrt{\eps} Y)^2 \big).
\label{F1}
\end{align}

\noi 
Note that 
\[
\sin^2(W)=(1-\cos(W))(1+\cos(W))\le 2(1-\cos(W)) .
\]
Hence, 
\begin{align}
|\sin(W) \cdot (\sqrt{\eps}Y)|
\le 
\kappa (1-\cos(W)) + C_\kappa (\sqrt{\eps}Y)^2
\label{F77}
\end{align}

\noi
for an arbitrarily small $\kappa>0$ and a large constant $C_\kappa$.
Note that the first term can be absorbed into the potential.
By plugging \eqref{F1} into \eqref{F7} and using \eqref{F77}, we obtain
\begin{align*}
\eps \log Z^Q_\eps  &\le \sup_{ W \in \ell^Q+\dot{\mathbb{H}}^1  } \E \bigg [-\int_{-L_\eps}^{L_\eps} (1-\kappa) (1-\cos W) dx -\frac {1-\dl}2 \int_{-L_\eps}^{L_\eps} |\dx W|^2 dx   \bigg]\\
&\hphantom{X}+ \E\bigg[  \int_{-L_\eps}^{L_\eps} C_\kappa | \sqrt{\eps} Y|^2 dx    \bigg]+C_\dl \| \dx \ell^Q \|_{L^2}^2.
\end{align*}

\noi
Using Lemma \ref{LEM:Moment} and $\dx \ell^Q= \frac{\pi Q}{L_\eps}$,
\begin{align}
& \eps \log Z_\eps^Q  \notag \\
&\le \sup_{ W \in \ell^Q+\dot{\mathbb{H}}^1  } \E \bigg [-\int_{-L_\eps}^{L_\eps} (1-\kappa)(1-\cos W) dx -\frac {1-\dl}2 \int_{-L_\eps}^{L_\eps} |\dx W|^2 dx   \bigg] + C_\kappa \eps L_\eps^2+C_\dl \frac{(\pi Q)^2}{L_\eps} .
\label{LDP4}
\end{align}

\noi
Under the condition $L_\eps=\eps^{-\frac 12+}$, by taking the limits $\eps \to 0$, $\dl \to 0$, and $\kappa \to 0$ in order, we obtain
\begin{align}
\limsup_{\eps \to 0} \eps \log Z_\eps^Q \le -\inf_{\phi \in \mathcal{C}_Q} E(\phi).
\label{UPP1}
\end{align}

From now on, we study the lower bound on the free energy. From \eqref{e:applyBD}, recall that 
\begin{align}
\eps \log Z_\eps^Q
=\sup_{ \dr \in \Ha} \E\bigg[-\int_{-L_\eps}^{L_\eps} ( 1-\cos(\sqrt{\eps}Y +\ell^Q +\Dr )  )dx - \frac 12  \int_0^1 \| \dr(t) \|_{L^2_x}^2 dt   \bigg].
\label{F10}
\end{align}

We choose a drift $\dr=\dr^0$ by 
\begin{align}
\dr^0(t)=\zeta^{-1} \cdot \ind_{ \{   t>1-\zeta  \}  } (t) (-\dx^2)^{\frac 12}(-\sqrt{\eps} Y_N(1-\zeta)-\ell^Q+m^\eps_{\xi_1,\dots,\xi_Q}),
\label{F8} 
\end{align}

\noi 
where $\zeta>0$, $\xi_1,\dots,\xi_k \in  [-L_\eps, L_\eps]$, and 
\begin{align*}
Y_N(1-\zeta)=\sum_{1 \le n \le N} \frac{B_n(1-\zeta)}{\sqrt{\ld_n}}e_n(x),
\end{align*}

\noi 
approximating the Gaussian field $Y=Y(1)=\sum_{n \ge 1} \frac{B_n(1)}{\sqrt{\ld_n} }e_n(x)$. Then the definition of $\Dr$ in \eqref{CY1} implies that 
\begin{align}
\Dr^0:=\int_0^1 (-\dx^2)^{-\frac 12} \dr^0(t)dt=-\sqrt{\eps}Y_N(1-\zeta)-\ell^Q+m^\eps_{\xi_1,\dots,\xi_Q}.
\label{F9}
\end{align}

\noi 
In \eqref{F8} the choice of the drift $\dr^0(t)$  is admissible $\dr^0 \in \Ha$ since it is an adapted process due to the cutoff $\ind_{ \{t>1-\zeta\} }$ and it satisfies the required regularity condition thanks to the truncation to the modes $1 \le n \le N$. By plugging \eqref{F8} and \eqref{F9} into \eqref{F10}, we obtain 
\begin{align}
\eps \log Z_\eps^Q 
\ge 
\E\bigg[-\int_{-L_\eps}^{L_\eps} ( 1-\cos(\sqrt{\eps}(Y-Y_N(1-\zeta) ) +m^\eps_{\xi_1,\dots, \xi_Q} )  )dx - \frac 12  \int_0^1 \| \dr^0(t) \|_{L^2_x}^2 dt   \bigg].
\label{F12}
\end{align}

\noi
From the definition of $\dr^0$ in \eqref{F8}, we write    
\begin{align*}
\frac 12  \int_0^1 \| \dr^0(t) \|_{L^2_x}^2 dt=\frac 12\|(-\dx^2)^{\frac 12}(-\sqrt{\eps} Y_N(1-\zeta)-\ell^Q+m^\eps_{\xi_1,\dots,\xi_Q}) \|_{L^2_x}^2.
\end{align*}

\noi
Using the inequality $(a+b+c)^2 \le (1+\dl)|a|^2 + C_\dl(|b|^2+|c|^2)$  for any real numbers $a,b,c$, and $\dl>0$, which follows from Young’s inequality, we have 
\begin{align}
\frac 12  \int_0^1 \| \dr^0(t) \|_{L^2_x}^2 dt  \le \frac{1+\dl}{2} \| \dx m_{\xi_1,\dots,\xi_Q}^\eps \|_{L^2}^2 +\frac{C_\dl}{2}  \big(\| \dx (\sqrt{\eps} Y_N(1-\zeta) )  \|_{L^2}^2+    \| \dx \ell^Q \|_{L^2}^2 \big)
\label{F11} 
\end{align}

\noi
for any $\dl>0 $ and the corresponding constant $C_\dl>0$. Regarding the potential energy term,  by taking the Taylor expansion around $m^\eps_{\xi_1,\dots,\xi_Q}$, 
\begin{align}
&\cos(\sqrt{\eps}(Y-Y_N(1-\zeta) ) +m^\eps_{\xi_1,\dots, \xi_Q} )  ) \notag \\
&=\cos(m_{\xi_1,\dots,\xi_Q}^\eps )-\sin(m_{\xi_1,\dots,\xi_Q}^\eps) \cdot \sqrt{\eps}(Y-Y_N(1-\zeta) )+O\big( (\sqrt{\eps}(Y-Y_N(1-\zeta) ) )^2 \big).
\label{F2}
\end{align}

\noi
By plugging \eqref{F2} and \eqref{F11} into \eqref{F12}, we obtain
\begin{align}
\eps \log Z^Q_\eps &\ge  \E\bigg[ -\int_{-L_\eps}^{L_\eps} (1-\cos m_{\xi_1,\dots,\xi_Q}^\eps ) dx -\frac{1+\dl}{2} \int_{-L_\eps}^{L_\eps} |\dx m_{\xi_1,\dots,\xi_Q}^\eps |^2 dx   \bigg] \notag \\
&\hphantom{X}- \E\bigg[  \int_{-L_\eps}^{L_\eps} | \sqrt{\eps} (Y-Y_N(1-\zeta) ) |^2 dx    \bigg] \notag \\
&\hphantom{X}-C_\dl \| \dx \ell^Q \|_{L^2}^2-C_\dl  \big(\| \dx (\sqrt{\eps} Y_N(1-\zeta) ) \|_{L^2}^2,
\end{align}

\noi
where we used $\E\big[\sqrt{\eps}(Y-Y_N(1-\zeta) ) \big]=0$. Using Lemma \ref{LEM:Moment} and $\dx \ell^Q= \frac{\pi Q}{L_\eps}$, we obtain 
\begin{align*}
\eps \log Z^Q_\eps &\ge  \E\bigg[ -\int_{-L_\eps}^{L_\eps} (1-\cos m^\eps_{\xi_1,\dots,\xi_Q} ) dx -\frac{1+\dl}{2} \int_{-L_\eps}^{L_\eps} |\dx m^\eps_{\xi_1,\dots,\xi_Q}|^2 dx   \bigg]\\
&\hphantom{X}  -c\eps L_\eps^2-C_\dl \frac{(\pi Q)^2}{L_\eps} -C_\dl  \eps \| \dx  Y_N(1-\zeta)  \|_{L^2}^2.
\end{align*}

\noi
Under the condition $L_\eps=\eps^{-\frac 12+}$, by taking the limits $\eps \to 0$ and $\dl \to 0$ in order, we obtain
\begin{align*}
\liminf_{\eps \to 0}\eps \log Z_\eps^Q \ge  - E(m_{\xi_1,\dots,\xi_Q}).
\end{align*}

\noi
By letting $\min_{i \neq j}|\xi_i-\xi_j| \to \infty $, and using Remark \ref{REM:AM}, we obtain
\begin{align*}
\liminf_{\eps \to 0}\eps \log Z_\eps^Q \ge -\inf_{\phi \in \mathcal{C}_Q} E(\phi).
\end{align*}
%\noi
%This completes the proof of Proposition \ref{PROP:free1}.

\end{proof}

%In order to prove the main result of this section, that is, the large deviation estimates, we need the following lemma.

\begin{lemma}\label{LEM:LDP}
Let $Q \in \Z$ with $Q>0$ and $L_\eps=\eps^{-\frac 12+}$. Then for $d,\dl>0$,
\begin{align*}
\limsup_{\eps \to 0} \eps \log Z_\eps^Q \Big[ \ind_{ \{ \textup{dist}(\phi, \M_Q ) \ge \dl \}    }   \Big] 
&\le \;\;  - \!\!\!\!\! \inf_{\substack{\phi \in \mathcal{C}_Q \\ \textup{dist}(\phi, \M_Q) \ge  2\dl }  } \!\!\! E(\phi),
\\
\limsup_{\eps \to 0} \eps \log Z_\eps^Q \Big[ \ind_{ \{ \textup{dist}(\phi, \M_Q^{<d} ) < \dl  }   \Big] &\le
\;\; - \!\!\!\!\! \inf_{\substack{\phi \in \mathcal{C}_Q \\ \textup{dist}(\phi, \M_Q^{<d} ) <  \frac \dl2  } } \!\!\! E(\phi).
\end{align*}

\end{lemma}

\begin{proof}
%As in \eqref{F0}, using the Boué–Dupuis formula (Lemma \ref{LEM:BD}) and the change of variables $\sqrt{\eps}\dr \to \dr$, we write 
Note that 
\[
\eps \log Z_\eps^Q \Big[ \ind_{ \{ \textup{dist}(\phi, \M_Q ) \ge \dl \}    }   \Big]
\le \eps \log \E_{\mu^Q_\eps }\bigg[\exp\Big\{-\frac 1\eps \int_{-L_\eps}^{L_\eps} (1-\cos \phi) dx \cdot \ind_{ \{ \textup{dist}(\phi, \M^Q ) \ge \dl \}    } \Big \}    \bigg]
\]
since the indicator may only take values in $\{0,1\}$. Proceeding as in the previous proof, 
we can bound the above expression by
\begin{align}
 \sup_{ \Dr \in \dot{\mathbb{H}}^1  }  \E\bigg[-\int_{-L_\eps}^{L_\eps} ( 1-\cos(\sqrt{\eps}Y +\ell^Q +\Dr )  )dx \cdot
\ind_{ \{ \textup{dist}(\sqrt{\eps}Y +\ell^Q +\Dr, \M^Q ) \ge \dl \}    }
- \frac 12  \| \dx \Dr \|_{L^2}^2    \bigg].
\label{LDP3}
\end{align}
Take  the change of variable $W=\ell^Q+\Dr$.
Under the condition $L_\eps=\eps^{-\frac 12+}$, we have $\E\big[ \| \sqrt{\eps}  Y\|_{L^2}^2 \big] \to 0$ by \eqref{1M}. Therefore, with high probability 
\begin{align}
\{ \textup{dist}(W,\M_Q) \ge 2\dl \} \subset \{ \textup{dist}(\sqrt{\eps}Y +W, \M_Q ) \ge \dl \}. 
\label{LDP5}
\end{align}

\noi
By following the steps used to obtain \eqref{LDP4} together with \eqref{LDP5}, we obtain
\begin{align}
&\eqref{LDP3}\le \sup_{ W \in \ell^Q+\dot{\mathbb{H}}^1  } \E \bigg [ \bigg(-\int_{-L_\eps}^{L_\eps} (1-\kappa) (1-\cos W) dx -\frac {1-\eta}2 \int_{-L_\eps}^{L_\eps} |\dx W|^2 dx \bigg) \ind_{ 
\{ \textup{dist}(W,\M_Q ) \ge 2\dl \}    }   \bigg] \notag \\
&\hphantom{XXXXXXXXX} +C_\kappa \eps L_\eps^2+C_\eta \frac{(\pi Q)^2}{L_\eps} 
\label{LDP6} 
\end{align}

\noi
for any $\eta>0$ and $\kappa>0$. Here we have multiplied $\int_{-L_\eps}^{L_\eps} |\dx W|^2 dx$ by the indicator as it is positive.
This implies that  under the condition $L_\eps=\eps^{-\frac 12+}$, by taking the limits $\eps \to 0$, $\eta \to 0$, and $\kappa \to 0$ in order, we obtain
 the first part of Lemma \ref{LEM:LDP}.

For the second part of Lemma \ref{LEM:LDP}, the only modification is that now
%we only need to modify the following part. Under the condition $L_\eps=\eps^{-\frac 12+}$, we have $\E\big[ \| \sqrt{\eps}  Y\|_{L^2}^2 \big] \to 0$ by \eqref{1M}. Therefore, 
with high probability 
\begin{align}
\{ \textup{dist}(W,\M_Q^{<d} ) < \tfrac \dl2 \} &\subset \{ \textup{dist}(\sqrt{\eps}Y +W, \M_Q^{<d} ) < \dl \}. 
\label{LDP7}
\end{align}
So for $\eps \log Z_\eps^Q \Big[ \ind_{ \{ \textup{dist}(\phi, \M_Q^{<d} ) < \dl  }   \Big]$ we get the same bound
as in \eqref{LDP6}, %together with \eqref{LDP7},
with the indicator replaced by $\ind_{ 
\{ \textup{dist}(W,\M_Q^{<d} ) < \frac \dl2 \}    } $.
 So  we again obtain the desired result. 

\end{proof}

%around the multi-soliton manifold $\M_Q$ in the no-collision regime.

%, that is, when $\min_{i \neq j}|\xi_i-\xi_j| \to \infty$.

%\begin{comment}
%\begin{proposition}\label{PROP:LDP}
%
%Let $Q \in \Z$ and $L_\eps =\eps^{-\frac 13+}$. Then there exists $c_0>0$ such that the following holds.
%\begin{itemize}
%\item[(i)] Far away from the multi-soliton manifold: Let $\dl>0$. Then we have 
%\begin{align*}
%\limsup_{\eps \to 0} \eps \log \rho_\eps^Q  \big(\{   \textup{dist}(\phi, \M_Q) \ge \dl   \} \big) \le -c_0 \dl^2,
%\end{align*}
%
%\noi
%where $\M^{Q}$ is defined as in \eqref{MQd0}.
%
%
%
%\medskip 
%
%
%\item[(ii)] Neighborhood of the multi-soliton manifold in the collision regime. Let $d,\dl>0$. Then
%\begin{align*} 
%\limsup_{\eps \to 0} \eps \log \rho_\eps^Q  \big(\{   \textup{dist}(\phi, \M^{Q, <d} ) < \dl   \} \big) \le -c e^{-d},
%\end{align*}
%
%\noi
%where for $Q>0$,
%\begin{align}
%\M^{Q,<d}=\{m_{\xi_1,\dots,\xi_k}: \xi_1,\dots,\xi_{Q} \in [-L_\eps, L_\eps] \quad \textup{and} \quad \min_{i \neq j} |\xi_i-\xi_j|<d \}.
%\label{MQd}
%\end{align}
%
%
%
%\noi
%For $Q<0$, we define the corresponding multi-soliton manifold with anti-kinks. 
%
%
%
%
%\end{itemize}
%
%\end{proposition}
%\end{comment}

%By first assuming Lemma \ref{LEM:LDP}, we prove the main result of this section, namely the large deviation estimates in Proposition \ref{PROP:LDP}.

\begin{proof}[Proof of Theorem \ref{THM:1}]

\noi
Combining \eqref{LDP1}, Proposition \ref{PROP:free1}, and Lemma \ref{LEM:LDP} yields 
\begin{align}
\limsup_{\eps \to 0} \eps \log \rho_\eps^Q  \big(\{   \textup{dist}(\phi, \M_Q) \ge \dl   \} \big) \le -\inf_{\substack{\phi \in \mathcal{C}_Q \\ 
\textup{dist}(\phi, \M_Q ) \ge  2 \dl }} E(\phi) +\inf_{  \phi \in \mathcal{C}_Q  } E(\phi).
\label{LDP2}
\end{align}

\noi
This, along with the energy gap estimate in Lemma \ref{LEM:GAP1}, implies
\begin{align}
\limsup_{\eps \to 0} \eps \log \rho_\eps^Q  \big(\{   \textup{dist}(\phi, \M_Q) \ge \dl   \} \big) \le -c \dl^2.
\label{lgdp} 
\end{align}

\noi
This proves \eqref{ldp1}.
Moreover,
by \eqref{LDP1s}, Proposition \ref{PROP:free1} and Lemma \ref{LEM:LDP}
\begin{align}
\limsup_{\eps \to 0} \eps \log \rho_\eps^Q  \big(\{   \textup{dist}(\phi, \M_Q^{<d} ) < \dl  \big)
\le -\inf_{\substack{ \textup{dist}(\phi, \M_Q^{<d} ) <  \tfrac {\dl}{2}    } } E(\phi)+\inf_{  \phi \in \mathcal{C}_Q  } E(\phi).
\label{lldp4}
\end{align}

\noi
Under the condition $c e^{-d} \ge \dl^2$, we can apply 
 Lemma \ref{LEM:GAP2} to obtain 
\begin{align}
\limsup_{\eps \to 0} \eps \log \rho_\eps^Q  \big(\{   \textup{dist}(\phi, \M_Q^{<d} ) < \dl  \big)
 \le - c e^{-d}.
\label{ldp4}
\end{align}

\noi
This proves \eqref{ldp2}, thus completing the proof of Theorem \ref{THM:1}.

\end{proof}

\begin{corollary}\label{COR:LDPA}
Under $\dl_\eps=\eta \sqrt{\eps \log \frac 1\eps}$  and $d_\eps= \big|\log (\eps \log \frac 1\eps) \big| $, we have 
\begin{align}
&\rho_\eps^Q  \big(\{   \textup{dist}(\phi, \M_Q^\eps ) \ge \dl_\eps   \} \big) \les e^{-c \log \frac 1\eps} \label{LDPIN2}\\
&\rho_\eps^Q  \big(\{   \textup{dist}(\phi, \M_Q^{\eps, <d_\eps} ) < \dl_\eps    \} \big)  \les e^{-c \log \frac 1 \eps}
\label{LDPIN3}
\end{align}

\noi
where $\M_Q^\eps$, $\M^{\eps, <d}_Q$ are approximating soliton manifolds, defined as in \eqref{MAN00} and \eqref{MAN0}.

%for $0<\eps< \eps_0$, where $\eps_0$ is a small number.

\end{corollary}

\begin{proof}
By following the proof of Theorem \ref{THM:1}, we can also establish the following large deviation estimates for the approximating soliton manifolds: 
\begin{align}
\limsup_{\eps \to 0} \eps \log \rho_\eps^Q  \big(\{   \textup{dist}(\phi, \M_Q^\eps ) \ge \dl   \} \big) &\le -c \dl^2
\label{LDPIN}\\
\limsup_{\eps \to 0} \eps \log  \rho_\eps^Q  \big(\{   \textup{dist}(\phi, \M_Q^{\eps, <d} ) < \dl    \} \big)  &\le - c e^{-d},
\label{LDPIN1}
\end{align}

\noi
where the condition $c \cdot e^{-d} \ge \dl^2$  follows from Lemma~\ref{LEM:GAP2}. Therefore, from \eqref{LDPIN}, \eqref{LDPIN1} and $c \cdot e^{-d} \ge \dl^2 $, by choosing $\dl_\eps=\eta \sqrt{\eps \log \frac 1\eps}$  and $d_\eps= \big|\log (\eps \log \frac 1\eps) \big| $,  we obtain  \eqref{LDPIN2} and \eqref{LDPIN3}.

\end{proof}

%\begin{remark}\rm \label{REM:LDP}

%%In the same way that $\phi$ is extended by $0$ and $2\pi Q$ outside $[-L_\eps, L_\eps]$, each element of the manifolds is understood with its corresponding extension. 

%\noi

%Therefore,  most of the probability mass is concentrated on the set 
%$\{ \textup{dist}(\phi, \M_Q^{\eps, \, \ge  d_\eps} ) <\dl_\eps  \}  $.    

%For the use of this large deviation estimate, see Lemma \ref{LEM:partition}.

%\end{remark}

%\begin{align*}
%\{ \textup{dist}(\phi, \M_Q^{ \ge d })<\dl \} \cup \{ \textup{dist}(\phi, \M_Q^{ \ge d }) \ge \dl \}
%\end{align*}

%\begin{align*}
%\{ \textup{dist}(\phi, \M_Q^{, \ge d }) \ge \dl \} \subset \{ \textup{dist}(\phi, \M_Q) \ge \dl \} \cup   \{ \textup{dist}(\phi, \M^{ <d  }) < \dl \}
%\end{align*}

%\begin{align*}
%\rho_Q^\eps\big(\textup{dist}(\phi, \M_Q) \ge \dl \} \big) \les e^{-\frac 1\eps \dl^2}
%\end{align*}

%\begin{align*}
%\rho_Q^\eps\big(\{ \textup{dist}(\phi, \M^{ <d  }) < \dl \}  \big) \les e^{-\frac 1\eps e^{-d}}
%\end{align*}

%if we have $\frac 12 e^{-d} \ge \dl$.

%This forces to choose
%\begin{align*}
%\dl_\eps=\eta \sqrt{\eps \log \frac 1\eps}, \qquad \eta \ll 1
%\end{align*}

%\begin{align*}
%d_\eps=\big|\log \sqrt{\eps \log \tfrac 1\eps}\big|
%\end{align*}

%\section{Preliminaries for the central limit theorem}
\section{Change of coordinates}
\label{SEC:coord}

%\subsection{Partition into a typical event and large deviation events}

The main result of this section is Proposition~\ref{PROP:cha}
which allows us to turn the integration over $\phi$ to an integration
over $(\xi_1,\cdots,\xi_Q)$ and $v\in V_{\xi_1,\dots,\xi_Q}$, up to small errors.
To this end we first
use the large-deviation estimates established above to decompose the entire space into a typical event and large-deviation events as follows. Recall the definition of $\M^{\eps,  \ge d}_Q$ in \eqref{MAN1}.

\begin{lemma}\label{LEM:partition}
Let $ Q \in \Z$ with $Q>0$, $\dl_\eps=\eta\sqrt{\eps \log \frac 1\eps}$, $d_\eps=\big|\log (\eps \log \frac 1\eps)   \big|$, as in \eqref{LDPIN2} and \eqref{LDPIN3}, and let $F$ be a bounded function. Then  
\begin{align*}
\int F(\phi) \rho_\eps^Q(d\phi)
= \int_{   \{  \textup{dist}(\phi, \M_Q^{\eps, \, \ge d_\eps} ) <\dl_\eps        \} } F(\phi) \rho_\eps^Q(d\phi)+\textup{O}(e^{-c  \log \frac 1\eps }).
\end{align*}

\end{lemma}

\begin{proof}
Write 
$A:=\{ \textup{dist}(\phi, \M_Q^{\eps}) <\delta_\eps \}$ 
and 
$B:=\{ \textup{dist}(\phi, \M_Q^{\eps,\, \ge d_\eps}) <\delta_\eps \}$.

Since $\M_Q^\eps=\M_Q^{\eps, \ge d_\eps} \cup \M_Q^{\eps, < d_\eps} $,
\[
A\cap B = B,
\qquad 
A\cap B^c \subset 
\{   \textup{dist}(\phi, \M_Q^{\eps, < d_\eps} ) < \dl_\eps  \}.
\]
Decomposing the full set as 
$(A\cap B)\cup (A\cap B^c) \cup A^c$,
then applying \eqref{LDPIN2} to $A^c$,
and  \eqref{LDPIN3} to $A\cap B^c$,
we obtain the lemma.

\end{proof}

%\subsection{Coordinate changes}

%As shown in the previous subsection, we prove that  
The above lemma shows that
as $\eps \to 0$, most of the probability mass concentrates in the region,
\begin{align*}
 \{  \textup{dist}(\phi, \M_Q^{\eps, \, \ge d_\eps} ) <\dl_\eps        \},
\end{align*}

\noi
where the field $\phi$ is close to the manifold and the solitons are well separated, that is, there is essentially no collision as $d_\eps \to \infty$. 
%See the definition of $\M^{\eps,  \ge d}_Q$ in \eqref{MAN1}. 

To prove the central limit theorem (Theorem \ref{THM:2}), we first perform the following change of coordinates using the disintegration formula in Lemma~\ref{LEM:chan}. As discussed in Remark \ref{REM:disinteg}, this is the first time the disintegration formula is applied in a regime with no minimizer, and in a setting around a multi-soliton manifold.
Recall that 
$\pi^\eps$ denotes the projection onto  $\M_Q^{\eps, \,  \ge d_\eps }$  defined in \eqref{APM5},
and we write $\Dl_Q=\{ -\cj L_\eps \le \xi_1 \le \cdots \le \xi_Q \le \cj L_\eps  \}$.

%\exp\Big\{-\frac 1\eps \int_{-L_\eps}^{L_\eps}  (1-\cos \phi) dx  \Big\}

% \begin{remark}\rm 
% As discussed in Remark \ref{REM:disinteg}, the disintegration formula in tangential and normal coordinates has been widely used in settings where the energy functional admits explicit minimizers. However, we emphasize again that, to the best of our knowledge, this is the first time the disintegration formula is applied in a regime where no minimizers exist, and the framework is set up around a multi-soliton manifold.
% \end{remark}

\begin{proposition}\label{PROP:cha}

Let $ Q \in \Z$ with $Q>0$, $\dl_\eps=\eta\sqrt{\eps \log \frac 1\eps}$, $d_\eps=\big|\log (\eps \log \frac 1\eps)   \big|$, as in \eqref{LDPIN2} and \eqref{LDPIN3}, and let $F$ be a bounded, continuous function. Then  
\begin{align}
&\int_{ \{  \textup{dist}(\phi, \M_Q^{\eps, \, \ge d_\eps} ) <\dl_\eps        \}  } F( \sqrt{\eps}^{-1} (\phi- \pi^\eps(\phi) )  )  \rho_{\eps}^Q(d\phi) \notag \\
&=\cj Z_{\eps}^{-1}  \int \cdots \int_{U_\eps }  F(v)  e^{\frac 1\eps \mathcal{E}_{\xi_1\,\dots,\xi_Q}(\sqrt{\eps}v ) }  
\textup{Det}_{\xi_1,\dots,\xi_Q}(\sqrt{\eps}v) \,  \nu^\perp_{\xi_1, \dots, \xi_Q}(dv)  \,  d\xi_1 \dots, d\xi_Q,  
\label{CC0}
\end{align}

%|\g( 0, \dots, 0)|

%
\noi
where  $\nu^\perp_{ \xi_1, \dots, \xi_Q }$ is the Gaussian measure in Lemma \ref{LEM:Gauss}, and 
\begin{align*}
&U_\eps:=\big\{ ( \xi_1,\dots, \xi_Q, v )\in  \Dl_Q \times V_{\xi_1,\dots,\xi_Q}: \| \sqrt{\eps} v \|_{L^2 }< \dl_\eps \; \textup{and} \; \min_{i \neq j}|\xi_i-\xi_j | \ge d_\eps    \big\}\\
&\mathcal{E}_{\xi_1,\dots,\xi_Q}(\sqrt{\eps}v )= \int_{-L_\eps}^{L_\eps}  \sin\big(\dr_x m_{\xi_1,\dots,\xi_Q}(x) +(1-\dr_x) \sqrt{\eps}v(x)  \big) \cdot (\sqrt{\eps}v  )^3  dx, \quad |\dr_x| \le 1\\
&\textup{Det}_{\xi_1,\dots,\xi_Q}(\sqrt{\eps}v):=\det\big( \Id-W_{\xi_1,\dots \xi_Q, \sqrt{\eps }v }\big).
\end{align*}

%of order $O((\sqrt{\eps}v )^3)$,
 In \eqref{CC0}, the partition function $\cj Z_\eps$ represents the integral  over $U_\eps$ with $F=1$, and $W_{\xi_1, \dots, \xi_Q, \sqrt{\eps}v}$ is the Weingarten map, defined in \eqref{Wein}.

%Here, $d_\eps=|\log \eps|^{1+\eta}$ and the higher-order term $\mathcal{E}_{\xi_1,\dots,\xi_k}(\sqrt{\eps}v )$ is given by
\end{proposition}

\begin{proof}

Using the disintegration formula in Lemma \ref{LEM:chan} with the coordinate 
\[
\phi=\pi^\eps(\phi)+\sqrt{\eps}v,
\qquad
\pi^\eps(\phi) = m^\eps_* := m^\eps_{\xi_1,\dots,\xi_Q}
\]
we decompose the integral into its tangential and normal components as follows 
\begin{align}
Z[F]:
&=\int_{  \{  \textup{dist}(\phi, \M_Q^{\eps, \, \ge  d_\eps} ) <\dl_\eps        \} 
} F( \sqrt{\eps}^{-1} (\phi- \pi^\eps(\phi) )  )  \exp\Big\{-\frac 1\eps \int_{-L_\eps}^{L_\eps} (1-\cos \phi) dx  \Big\} \mu_\eps^Q(d\phi)  \notag \\
&=  \int \cdots \int_{U_\eps} F(v) 
\exp\Big\{-\frac 1\eps \int_{-L_\eps}^{L_\eps} 1-\cos(m^\eps_* +\sqrt{\eps}v)  dx  \Big\} \notag 
\\
&\hspace{27mm} \cdot  \exp\Big\{-\frac 1{2\eps} \| \dx m^\eps_*  \|_{L^2 }^2  -\frac 1\eps \jb{ (-\dx^2) m^\eps_*, \sqrt{\eps}v }_{L^2}  \Big\} \notag \\
&\hspace{27mm} \cdot  \textup{Det}_{\xi_1,\dots,\xi_k}(\sqrt{\eps}v) \,\mu^\perp_{\xi_1,\dots,\xi_Q}(dv) \, d\s(\xi_1,\dots,\xi_Q),
\label{ch0}
\end{align}
where
$\mu^\perp_{\xi_1,\dots,\xi_k}$ is the Brownian bridge with covariance $(-\dx^2)^{-1}$ from $0$ to $0$, on the normal space, and $d\s$ is the surface measure defined in \eqref{surfm}.

Note that the terms in the exponential together with the Cameron-Martin term in $\mu^\perp_{\xi_1,\dots,\xi_k} (dv)$ is just $-\frac{1}{\eps} E_{L_\eps}(\phi)$ where 
$E_{L_\eps}(\phi)=\frac 12 \int_{-L_\eps}^{L_\eps} |\dx \phi|^2 dx +\int_{-L_\eps}^{L_\eps} (1-\cos \phi) \, dx $. Taylor expansion gives
\begin{align}
-\frac 1\eps E_{L_\eps}(m^\eps_* +\sqrt{\eps} v )
&= -\frac 1\eps E_{L_\eps}(m^\eps_*)   -\frac 1\eps \jb{  \nb E_{L_\eps}(m^\eps_* ), \sqrt{\eps} v  } \notag 
\\
&\hphantom{X}-\frac 1{2\eps} \jb{ \sqrt{\eps}v, (-\dx^2+\cos m^\eps_*) \sqrt{\eps}v   }  
+\frac 1\eps \mathcal{E}_{\xi_1,\dots,\xi_Q}(\sqrt{\eps}v),
\label{ch4}
\end{align}
where $\mathcal{E}_{\xi_1,\dots,\xi_Q}(\sqrt{\eps}v)=\int_{-L_\eps}^{L_\eps} O\big((\sqrt{\eps}v  )^3 \big) dx$. 
Therefore
\begin{align}
Z[F]&=\int \cdots \int_{U_\eps} F(v) 
\exp\Big\{ -\frac 1\eps E_{L_\eps}(m^\eps_*) -\frac 1\eps \jb{  \nb E_{L_\eps}(m^\eps_* ), \sqrt{\eps} v  }   \Big\} \notag \\
&\hphantom{XXXXX} \cdot \exp\Big\{ \frac 1\eps \mathcal{E}_{\xi_1,\dots,\xi_Q}(\sqrt{\eps}v)   \Big\}  \exp\Big\{  -\frac 12 \jb{v, \cos( m^\eps_*)v } \Big\}  \notag \\
&\hphantom{XXXXX} \cdot \textup{Det}_{\xi_1,\dots,\xi_Q}(\sqrt{\eps}v) \,\mu^\perp_{\xi_1,\dots,\xi_Q}(dv) \, d\s(\xi_1,\dots,\xi_Q).
\label{ch00}
\end{align}
We estimate the terms in \eqref{ch00} one by one,
and all the estimates will be uniform in $\xi_1,\dots \xi_Q$ satisfying $\min_{i \neq j }|\xi_i-\xi_j| \ge  d_\eps=\big|\log (\eps \log \frac 1\eps) \big| $.

By Lemma \ref{LEM:MUS} and \eqref{APM2},  
\begin{align}
\exp\Big\{ -\frac 1\eps E_{L_\eps}(m^\eps_* )    \Big\}&=\exp\Big\{ -\frac 1\eps |Q|E_{\textup{kink}} -\frac 1\eps O(e^{-cd_\eps} )    \Big\} \notag \\
&=\exp\Big\{ -\frac 1\eps |Q|E_{\textup{kink}}(1+O(\eps^{0+}) )     \Big\},
\label{ch2}
\end{align}

%Note that the error term $\frac 1\eps e^{-d_\eps}$ satisfies
%\begin{align*}
%\frac 1\eps e^{-d_\eps}=\frac{\eps \log \frac 1\eps}{\eps}=\log \frac 1\eps
%\end{align*}

%\noi
%uniformly in $\xi_1,\dots \xi_Q$ satisfying $\min_{i \neq j }|\xi_i-\xi_j| \ge d_\eps $. Then we have 
%\begin{align}
%\exp\Big\{-\frac 1\eps e^{-d_\eps} \Big\}=1+O(\eps^M),
%\label{ch5}
%\end{align}

%\noi 
%which implies  
%\begin{align}
%\exp\Big\{ -\frac 1\eps E_{L_\eps}(m^\eps_{\xi_1,\dots,\xi_Q} )    \Big\}=\exp\Big\{ -\frac 1\eps |Q|E_{\textup{kink}}     \Big\}(1+O(\eps^M)),
%\label{ch6}
%\end{align}

%\noi 
%uniformly in $\xi_1,\dots \xi_Q$ satisfying $\min_{i \neq j }|\xi_i-\xi_j| \ge d_\eps $.

 By Lemma \ref{LEM:gra} and the condition $\| \sqrt{\eps} v\|_{L^2} \le \dl_\eps=\eta \sqrt{\eps \log \tfrac 1\eps}$, 
\begin{align}
\bigg| \frac 1\eps \int_{-L_\eps}^{L_\eps} \nb E_{L_\eps}(m^\eps_*) \sqrt{\eps}v  dx \bigg| 
\le 
\frac 1 \eps \| \nb E_{L_\eps}(m^\eps_*)  \|_{L^2} \| \sqrt{\eps}v \|_{L^2} \les \frac 1\eps e^{-(1-\eta)d_\eps} \dl_\eps=O(\eps^{\frac 12-}),
\label{ch3}
\end{align}
which implies that 
\begin{align}
\exp\Big\{  -\frac 1\eps \jb{  \nb E_{L_\eps}(m^\eps_* ), \sqrt{\eps} v  }   \Big\}    =1+O(\eps^{\frac 12-}).
\label{ch7}
\end{align}

\begin{comment}

\kihoon{ $d_\eps=\frac 1c_0 |\log (\eps \log \frac 1\eps )|$
\begin{align*}
\frac 1\eps e^{-c d_\eps}=\frac 1\eps e^{ \frac{c}{c_0} \log (\eps \log \frac 1\eps)   }  \sqrt{\eps \log \frac 1\eps }=\eps^{-\frac 12} \sqrt{\log \frac 1\eps}   (\eps \log \frac 1\eps)^{\frac c{c_0}}
\end{align*}

If we define $d_\eps=\eta\frac 1c_0 |\log (\eps \log \frac 1\eps )|$, $\eta\gg 1$, we can work out the issue but the condition $e^{-c_0d_\eps} \gg \dl_\eps^2 $ fails
 }

\end{comment}

We now study the second-order term $-\frac 1{2} \jb{ v, (-\dx^2+\cos m^\eps_*) v   }$ in \eqref{ch4}, which defines a new base Gaussian measure on the normal space. Note that as observed in \eqref{2ndvar}
\begin{align*}
-\dx^2 +\cos m^\eps_*
=-\dx^2+1-2\sum_{j=1}^Q \text{sech}^2(\cdot-\xi_j)+O(e^{-cd_\eps})
=\mathcal{L}_{\xi_1,\dots,\xi_Q}+O(e^{-c d_\eps}),
\end{align*}
where $\mathcal{L}_{\xi_1,\dots,\xi_Q}$ is defined in 
\eqref{linQ}.
Since $v \in V_{\xi_1,\dots,\xi_Q}$ satisfies the orthogonality conditions, from Lemma \ref{LEM:Gauss}, we can define the Gaussian measure 
\begin{align}
d\nu^\perp_{\xi_1,\dots,\xi_Q}(v)=Z_{\xi_1,\dots,\xi_Q}^{-1}\exp\Big\{ -\frac 12 \jb{v, \mathcal{L}_{\xi_1,\dots,\xi_Q} v }  \Big\} \prod_{x \in [-L_\eps, L_\eps]} dv(x),
\label{PT1}
\end{align}

\noi
where $Z_{\xi_1,\dots,\xi_Q}=(Z_0)^Q(1+O(e^{-cd_\eps}))$ from Lemma \ref{LEM:part}.

\noi 
\begin{comment}
When the solitons are well separated, the operator $\mathcal{L}_{\xi_1,\dots,\xi_k}$ factorizes into these $k$ independent blocks up to exponentially small interactions, giving 
\begin{align*}
\textup{det}_{V_{\xi_1,\dots,\xi_k}} \mathcal{L}_{\xi_1,\dots,\xi_k}=(\textup{det}_{V_0}  \mathcal{L}_0)^k(1+O(e^{-d_\eps})),
\end{align*}

\noi 
uniformly in $\xi_1,\dots \xi_k$ satisfying $\min_{i \neq j }|\xi_i-\xi_j| \ge d_\eps $, where $\mathcal{L}_0=-\dx^2+1-2\textup{sech}^2(x)$.
Here, we note that both operators $\mathcal{L}_{\xi_1,\dots,\xi_k}$ and $\mathcal{L}_0$ act on the normal spaces. 
\end{comment}

Finally, for the surface measure $d\s(\xi_1,\dots,\xi_Q)$ on the multi-soliton manifold, since the tangent vectors $\{ \partial_{\xi_j} m^\eps_{\xi_1,\dots,\xi_Q}  \}_{j=1}^Q$ are almost orthogonal, as in \eqref{e:Jac} the determinant of Jacobian equals
\begin{align}
Q \| \dx   m \|_{L^2(\R)}^2(1+O(e^{-cd_\eps}) ).
\label{ch8} 
\end{align}

\noi 
By using \eqref{ch0}, \eqref{ch00}, \eqref{ch2}, \eqref{ch7}, \eqref{PT1}, and \eqref{ch8},
we can  take out terms independent of tangential modes $\xi_1,\dots, \xi_Q$ satisfying $\min_{i \neq j }|\xi_i-\xi_j| \ge  d_\eps $ as follows: 
\begin{align*}
Z [F]&=J_\eps  \int \cdots \int_{U_\eps} F(v) e^{ \frac 1\eps  \mathcal{E}_{\xi_1,\dots,\xi_Q}( \sqrt{\eps}v )   } \textup{Det}_{\xi_1,\dots,\xi_Q}(\sqrt{\eps}v)  \, d\nu^\perp_{m^\eps_{\xi_1,\dots,\xi_Q}}(v) \, d\xi_Q\dots d\xi_1,
\end{align*}

\noi
where 
\begin{align*}
J_\eps=\exp\Big\{ -\frac 1\eps |Q| E_{\textup{kink}} \Big\} (Z_0)^Q \cdot  Q \| \dx m \|_{L^2(\R)}^2(1+O(e^{-cd_\eps}) ) (1+O(\eps^{\frac 12-})).
\end{align*}

\noi
Since the partition function $Z_\eps=Z_{\eps}[1]$ contains the same factor
$J_\eps$, we can cancel the common term $J_\eps$ and thus obtain the desired result \eqref{CC0}.

\end{proof}

%\begin{align*}
%&\int_{ \{   \textup{dist}_{L^2}(\phi, M_Q) < \dl,  \; \min_{i \neq j} |\xi_j-\xi_k| <\ell  \}   } F( \sqrt{\eps}^{-1}(\phi- \pi_\eps(\phi) ) ) \rho_\eps^Q (d\phi)\\
%&=\int\limits_{\xi_1 \in [-L_\eps, L_\eps] } \dots  \int\limits_{\xi_k \in [-L_\eps, L_\eps] }   \int_{  \| \sqrt{\eps} v \|_{L^2([-L_\eps, L_\eps])} <\dl  } F(v) 
%\end{align*}

\section{Analysis of the Gaussian measure with Schr\"odinger operator}

In this section, our goal is to establish the correlation decay and 
the extreme value  for the Gaussian measure $\nu^\perp_{\xi_1,\dots,\xi_k }$ on the normal space, with the covariance operator
\begin{align}
C_{\xi_1,\dots,\xi_k}=\P_{V_{\xi_1,\dots, \xi_k} } \big(-\dx^2+1-2\sum_{j=1}^k \text{sech}^2(\cdot-\xi_j) \big)^{-1} \P_{V_{\xi_1,\dots, \xi_k} },
\label{Cdec}
\end{align}

\noi
subject to Dirichlet boundary conditions, as defined in \eqref{cov}.

%In order to study the extreme value (Lemma \ref{LEM:T1}) and the law of large numbers (Lemma \ref{LEM:T2}) for the Gaussian measure, we first analyze the covariance structure in Subsections \ref{SUBSEC:cd} and \ref{SUBSEC:rg}.

\begin{remark}\rm 
In this section, we study a joint limit in which $\eps \to 0$ and $L_\eps \to \infty$ simultaneously, leading to a competition between energetic and entropic effects. Moreover, in our setting the relevant covariance operator \eqref{Cdec} is defined around a multi-soliton configuration, which is not a minimizer of the energy. This is in sharp contrast to \cite[Theorem 4]{ER2}, where the analysis is based on the covariance operator around an explicit energy minimizer and only the energetic limit $\eps \to 0$ is considered to establish a central limit theorem.
\end{remark}

\subsection{Correlation decay }\label{SUBSEC:cd}

In this subsection, we study the decay of correlations for the Gaussian measure with covariance operator  $C_{\xi_1,\dots,\xi_k}$ in \eqref{Cdec}.
Notice that the operator $\mathcal{L}_{\xi_1,\dots,\xi_k}=-\dx^2+1-2\sum_{j=1}^k \text{sech}^2(\cdot-\xi_j) $, defined in \eqref{linQ}, is invertible on the normal space $V_{\xi_1,\dots,\xi_k}$. Therefore, we can take its inverse to study the Green function
\begin{align}
G^D_{\xi_1,\dots, \xi_k}(x,y):=\jb{\dl_x, C_{\xi_1,\dots, \xi_k}  \dl_y }=\E_{ \nu^\perp_{ \xi_1, \dots, \xi_k } } \big[ v(x) v(y)  \big]
\label{Cdec0}
\end{align}

\noi 
with Dirichlet boundary conditions on $[-L_\eps, L_\eps]$. In order to study the decay of correlations in \eqref{Cdec0}, we first analyze the projected Ornstein–Uhlenbeck operator on the normal space $V_{\xi_1,\dots, \xi_k}$
\begin{align}
G_{\text{OU}, \xi_1,\dots, \xi_k }^D(x,y)= \big( \P_{V_{\xi_1,\dots, \xi_k} } (-\dx^2+1)^{-1} \P_{V_{\xi_1,\dots, \xi_k} } \big)(x,y)
\label{GDOU}
\end{align}

\noi 
with Dirichlet boundary conditions on $[-L_\eps, L_\eps]$.  Note that Dirichlet boundary conditions suppress the variance $(-\dx^2+1)^{-1}(x,x) \approx 0$ near the edges, eliminating randomness there. In particular, the influence of the boundary diminishes exponentially fast as one moves into the interior, and the field in the bulk behaves almost like the infinite-volume Ornstein–Uhlenbeck field. The following result shows that the Green function $G_{\text{OU}, \xi_1,\dots, \xi_k }^D$ exhibits correlation decay away from the boundary points $-L_\eps$, $L_\eps$.

\begin{lemma}\label{LEM:OUdec}
Let $L_\eps=\eps^{-\frac 12+\eta}$ and $ \min_{ \ell \neq i}|\xi_\ell-\xi_i | \ge d_\eps= \big|\log  (\eps \log \frac 1 \eps) \big|$. 

\begin{itemize}

\item[(i)] Bulk regime: Let $|x|, |y| \le L_\eps -|\log \eps|$ and $|\xi_j| \le \cj L_\eps$, $j=1,\dots,k$, where $\cj L_\eps$ is defined in
\eqref{cjL}.  Then 
\begin{align*}
|G_{\textup{OU}, \xi_1,\dots, \xi_k }^D(x,y)-c e^{-|x-y|}| \les e^{-2(L_\eps -\max \{ |x|, |y| \}  )}
\end{align*}

\noi
as $\eps \to 0$, where the exponentially small error term is independent of the tangential modes $\xi_1,\dots,\xi_k$ satisfying $|\xi_j| \le \cj L_\eps$.

\medskip

\item[(ii)] Diagonal regime: for any $x\in [-L_\eps, L_\eps]$
\begin{align*}
|G_{\textup{OU}, \xi_1,\dots, \xi_k }^D(x,x)| \les 1,
\end{align*}

\noi
uniformly in $\xi_1,\dots, \xi_k \in [-\cj L_\eps, \cj L_\eps]$.

\end{itemize}

\end{lemma}

%%\noi
%Here $t_1,\dots,t_k$ are obtained by applying the Gram–Schmidt orthonormalization to the tangent vectors $\partial_{\xi_j}  m^\eps_{\xi_1,\dots,\xi_k}$, $j=1,\dots,k$.

%the projection operator $\P_{V_{\xi_1,\dots,\xi_k}}$, 

\begin{proof}

We expand the covariance operator $ \P_{V_{\xi_1,\dots, \xi_k} } (-\dx^2+1)^{-1} \P_{V_{\xi_1,\dots, \xi_k} }$. Recall the projection operator  $\P_{V_{\xi_1,\dots,\xi_k}}=\Id-  \sum_{j=1}^k \P_j$ \eqref{proj1} onto the normal space $V_{\xi_1,\dots, \xi_k}$. By expanding 
\begin{align}
\P_{V_{\xi_1,\dots, \xi_k} } \big(-\dx^2+1)^{-1} \P_{V_{\xi_1,\dots, \xi_k} }&=(-\dx^2+1)^{-1}+\sum_{i,j=1}^k \P_i (-\dx^2+1)^{-1} \P_j \notag \\
&\hphantom{X}-\sum_{j=1}^k  (-\dx^2+1)^{-1} \P_j \notag \\
&\hphantom{X}-\sum_{i=1}^k \P_i (-\dx^2+1)^{-1}, 
\label{dec0}
\end{align}

\noi
we study each term separately.  We denote by $G^D_{\text{OU} }$ the Green function of $(-\dx^2+1)^{-1}$ on $[-L_\eps, L_\eps]$ with Dirichlet boundary conditions.  Then, the explicit formula and the bulk behavior are well known as follows:
\begin{align}
G_{\text{OU}}^D(x,y)=  \frac{ \text{sinh}(L_\eps-   \max\{x,y\} ) \cdot \text{sinh}(L_\eps+   \min\{x,y\} )   }{  \text{sinh}(2L_\eps )  } 
\label{OUker}
\end{align}

\noi
and if $|x|, |y| \le L_\eps- |\log \eps| $, 
\begin{align}
\big| G_{\text{OU}}^D(x,y) -\frac 12 e^{-|x-y|} \big| \les e^{-2(L_\eps -\max\{|x|, |y|\})} \les e^{-2|\log \eps|}.
\label{dec2}
\end{align}

\noi 
Hence, in the bulk region $|x|, |y| \le L_\eps -|\log \eps|$, the Green function coincides with the whole-line kernel up to an exponentially small error in the distance to the boundary.

We now consider the first projected term $\P_i (-\dx^2 + 1)^{-1} \P_j$. By direct computation using the definition of $\P_j$ in \eqref{dec00}, 
\begin{align}
\big\langle \dl_x, \P_i (-\dx^2+1)^{-1} \P_j \dl_y \big\rangle=  t_i(x) t_j(y)  \big \langle t_i, (-\dx^2+1)^{-1} t_j  \big\rangle.
\label{dec22}
\end{align}

%the fact that each tangent vector $t_j$ is localized around $\xi_j$ with an exponentially decaying tail,

\noi
Using Lemma \ref{LEM:TGRAM},  that is, $|t_j(x)| \les e^{-|x-\xi_j|}$, we have
\begin{align}
|\eqref{dec22}|
&\les e^{-c|x-\xi_i |} e^{-c|y-\xi_j |} e^{-c |\xi_i-\xi_j| }  \le e^{-c |x-y|}.
\label{dec1}
\end{align}

\noi
We next consider the second projection term $  (-\dx^2+1)^{-1} \P_j $ in \eqref{dec0}.  By direct computation using the definition of $\P_j$ in \eqref{dec00}, 
\begin{align}
\big\langle \dl_x,  (-\dx^2+1)^{-1} \P_j \dl_y \big\rangle &=t_j(y)  \big\langle \dl_x, (-\dx^2+1)^{-1}   t_j   \big\rangle=t_j(y)  \int_{-L_\eps}^{L_\eps} G^D_{\text{OU} }(x,z)t_j(z) dz, 
\label{dec33}
\end{align}

\noi
where $G^D_{\text{OU} }$ denotes the Green function of $(-\dx^2+1)^{-1}$ on $[-L_\eps, L_\eps]$ with Dirichlet boundary conditions. Note that 
\begin{align}
\int_{-L_\eps}^{L_\eps} G^D_{\text{OU} }(x,z) t_j(z) dz 
&=\int_{-L_\eps}^{L_\eps} e^{-|x-z|} t_j(z) dz+ \int_{-L_\eps }^{L_\eps} (G^D_{\text{OU} }(x,z)-e^{-|x-z|}) t_j(z) dz  \notag \\
&=\I_1+\I_2.
\label{sta1} 
\end{align}

\noi 
Since $|t_j(x)| \les e^{-|x-\xi_j|}$, we have 
\begin{align}
|\I_1| \les \int_{-L_\eps}^{L_\eps} e^{-|x-z|}e^{-|z-\xi_j|} dz \les e^{-c|x-\xi_j|}.
\label{sta2}
\end{align}

\noi
Regarding the error term $\I_2$, under $|z-\xi_j| \le \frac 12 |\log \eps|$,
we use $\max\{|x|, |z| \} \le \max\{ |x|, |\xi_j| \}+\frac 12 |\log \eps|$  to obtain 
\begin{align}
|\I_2|&\les \int_{ |z-\xi_j| \le \frac 12 |\log \eps|  } e^{-2(L_\eps-\max\{|x|,|z| \} )}  t_j(z) dz +\int_{|z-\xi_j| \ge \frac 12 |\log \eps| } e^{-c|z-\xi_j|} dz \notag \\
& \les e^{|\log \eps|}  e^{-2(L_\eps- \max\{|x|,|\xi_j|  \} ) }+ e^{-c|\log \eps| } \les e^{-|\log \eps|},
\label{sta3}
\end{align}

\noi
where in the last line we used the bulk conditions $|x| \le L_\eps-|\log \eps|$ and $|\xi_j|\le \cj L_\eps=L_\eps -\eps^{-\frac 12+2\eta}-1$. Hence, by combining \eqref{sta1}, \eqref{sta2}, and \eqref{sta3}, we obtain 
\begin{align}
\bigg|  \int_{-L_\eps}^{L_\eps} G^D_{\text{OU} }(x,z)t_j(z) dz  \bigg| \les e^{-c|x-\xi_j|}
\label{dec333}
\end{align}

\noi
as $\eps \to 0$.  Then, it follows from \eqref{dec33} and \eqref{dec333} that 
\begin{align}
|\eqref{dec33}|  \les e^{-c|y-\xi_j|} e^{-c |x-\xi_j|} \les e^{-c  |x-y|}.
\label{dec3}
\end{align}

\noi
As for the projected term $\P_i (-\dx^2+1)^{-1} $ in \eqref{dec0}, following the above calculations, we have
\begin{align}
\big| \big\langle \dl_x,  \P_i (-\dx^2+1)^{-1}  \dl_y \big\rangle  \big|  \les e^{-c|x-y|}.
\label{dec4}
\end{align}

\noi
It follows from \eqref{dec0}, \eqref{dec2}, \eqref{dec1}, \eqref{dec3}, and \eqref{dec4} that 
\begin{align*}
G_{\text{OU}, \xi_1,\dots,\xi_k }^D(x,y)=c e^{-|x-y|}+ O\big(e^{-2(L_\eps -\max \{ |x|, |y| \}  )} \big) \sim e^{-c|x-y|}
\end{align*}

\noi
when $x$ and $y$ are far from the boundaries.

We now prove part (ii) of Lemma \ref{LEM:OUdec} (diagonal regime).
Using the closed form \eqref{OUker}, we have  $0\le G^D_{\text{OU} }(x,x) \le \frac 12$
for all $x\in [-L_\eps, L_\eps]$. Since $\P_{V_{\xi_1,\dots,\xi_k}}$ is an orthogonal projection, so $0 \le \P_{V_{\xi_1,\dots,\xi_k}} \le \Id$. Hence, adding the projection cannot increase the diagonal  $0\le G_{\textup{OU}, \xi_1,\dots, \xi_k }^D(x,x) \le G^D_{\text{OU}}(x,x) \le \frac 12$ for all $x\in [-L_\eps, L_\eps]$, uniformly in $\xi_1,\dots,\xi_k \in [-\cj L_\eps, \cj L_\eps]$. This completes the proof.

\end{proof}

We are now ready to prove the main part of this subsection. Recall the Green function $G^D_{\xi_1,\dots,\xi_k}$ \eqref{Cdec0} for the covariance operator $\mathcal{C}_{\xi_1,\dots,\xi_k}$ in \eqref{Cdec}. The following proposition shows that the Green function  exhibits correlation decay away from the boundary points $-L_\eps$, $L_\eps$.

\begin{proposition}\label{PROP:dec}
Let $L_\eps=\eps^{-\frac 12+\eta}$ and $ \min_{ \ell \neq i}|\xi_\ell-\xi_i | \ge d_\eps= \big|\log  (\eps \log \frac 1 \eps) \big|$.

\begin{itemize}

\item[(i)] Bulk regime: Let $|x|, |y| \le L_\eps -|\log \eps|$ and $|\xi_j| \le \cj L_\eps$, $j=1,\dots,k$, where $\cj L_\eps$ is defined in
\eqref{cjL}.  Then 
\begin{align*}
|G^D_{\xi_1,\dots,\xi_k}(x,y)-c e^{-|x-y|}| \les  e^{-2(L_\eps -\max \{ |x|, |y| \}  )},
\end{align*}

\noi
as $\eps \to 0$, where the exponentially small error term is independent of the tangential modes $\xi_1,\dots,\xi_k$ satisfying $|\xi_j| \le \cj L_\eps $.

\medskip 

\item[(ii)]  Diagonal regime: for any $x\in [-L_\eps, L_\eps]$
\begin{align}
|G_{\xi_1,\dots, \xi_k }^D(x,x)| \les 1,
\label{diag}
\end{align}

\noi
uniformly in $\xi_1,\dots, \xi_k \in [-\cj L_\eps, \cj L_\eps]$.

\end{itemize}

\end{proposition}

%\begin{align*}
%C_{\xi_1,\dots,\xi_k}=\P_{V_{\xi_1,\dots, \xi_k} } \big(-\dx^2+1-2\sum_{j=1}^k \text{sech}^2(\cdot-\xi_j) \big)^{-1} \P_{V_{\xi_1,\dots, \xi_k} }
%\end{align*}

\begin{proof}

Using the resolvent identity, we view the Schr\"odinger operator $C_{\xi_1,\dots,\xi_k}$, defined in \eqref{Cdec},  as a perturbation of the Ornstein–Uhlenbeck operator $\P_{V_{\xi_1,\dots, \xi_k} } \big(-\dx^2+1)^{-1} \P_{V_{\xi_1,\dots, \xi_k} }$ as follows 
\begin{align}
G^D_{\xi_1,\dots,\xi_k}(x,y)&= G^D_{ \text{OU}, \xi_1,\dots,\xi_k}(x,y)+(G^D_{ \text{OU}, \xi_1,\dots,\xi_k} W G^D_{\xi_1,\dots,\xi_k} )(x,y) \notag \\
&=G^D_{ \text{OU}, \xi_1,\dots,\xi_k}(x,y) +( G^D_{\xi_1,\dots,\xi_k}  W  G^D_{ \text{OU}, \xi_1,\dots,\xi_k})(x,y),
\label{k1}
\end{align}

\noi
where $W(z)=-2\sum_{j=1}^k \text{sech}^2(\cdot-\xi_j)$. From \eqref{k1}, we have 
\begin{align}
\sum_{j=1}^k  (G^D_{ \text{OU}, \xi_1,\dots,\xi_k} W_j G^D_{\xi_1,\dots,\xi_k} )&=\sum_{j=1}^k ( G^D_{\xi_1,\dots,\xi_k}  W_j  G^D_{ \text{OU}, \xi_1,\dots,\xi_k}),
\label{k10} 
\end{align}

\noi
where $W_j(z)=-2\text{sech}^2(z-\xi_j)$. Note that the equality \eqref{k10} holds after summing over $j$, although the individual components do not commute. Indeed, by Lemma \ref{LEM:CT} below, we have
\begin{align}
\big\| \big(G^D_{ \text{OU}, \xi_1,\dots,\xi_k} W_j G^D_{\xi_1,\dots,\xi_k} \big) - \big( G^D_{\xi_1,\dots,\xi_k}  W_j  G^D_{ \text{OU}, \xi_1,\dots,\xi_k} \big)   \big\|_{L^2 \to L^2} \les e^{-c \min_{\ell \neq i }|\xi_\ell-\xi_i| }, 
\label{com1}
\end{align}

\noi
which will be used below.  Using Lemma \ref{LEM:OUdec},  we write 
\begin{align}
G^D_{ \text{OU}, \xi_1,\dots,\xi_k}(x_1,x_2)&=e^{- |x_1-x_2|  }+O(e^{-2(L_\eps -\max\{ |x_1|,|x_2| \}  ) }).
\label{k9}
\end{align}

%Under the bulk condition, 

\noi
Following the error estimates in \eqref{sta1} and \eqref{sta3}, and using \eqref{k9}, $|W_j(z)|\les e^{-|z-\xi_j|}$, we obtain
\begin{align}
\bigg| \int_{-L_\eps}^{L_\eps}
G^D_{ \text{OU}, \xi_1,\dots,\xi_k}(x,z)W_j(z) G^D_{\xi_1,\dots,\xi_k} (z,y) dz  \bigg| 
&\les  \int_{-L_\eps }^{ L_\eps }  e^{-\ld |x-z| } e^{-\ld |z-\xi_j| } G^D_{\xi_1,\dots,\xi_k} (z,y) dz+ e^{-c|\log \eps|}
\label{k7} 
\end{align}

\noi
for some $\ld>0$, and
\begin{align}
&\bigg| \int_{-L_\eps}^{L_\eps}
G^D_{\xi_1,\dots,\xi_k} (x,z) W_j(z) G^D_{ \text{OU}, \xi_1,\dots,\xi_k}(z,y) dz  \bigg| \les \int_{-L_\eps}^{L_\eps}
G^D_{\xi_1,\dots,\xi_k} (x,z) e^{-\ld |z-\xi_j|}  e^{-\ld |y-z|}  dz +e^{-c|\log \eps|}.
\label{k8}
\end{align}

\noi
We proceed by considering two cases: $|y-\xi_j|  \le \frac 12 |x-y| $ and  $|y-\xi_j| \ge \frac 12 |x-y|$.  In the former case, we use \eqref{k7}, while in the latter we use \eqref{k8}.

\textbf{Case 1}: if $|y-\xi_j|  \le \frac 12 |x-y| $, then
\begin{align}
\int_{-L_\eps }^{ L_\eps }  e^{-\ld |x-z| } e^{-\ld |z-\xi_j| } G^D_{\xi_1,\dots,\xi_k} (z,y) dz  
&\les \int_{-L_\eps }^{L_\eps}  e^{ -\frac \ld2   (|x-y| -|y-\xi_j| -|z-\xi_j| )         }   e^{-\ld |z-\xi_j |  }  G^D_{\xi_1,\dots,\xi_k} (z,y)  dz \notag \\
& \les e^{-\frac \ld4 |x-y| } \int_{-L_\eps}^{L_\eps}   e^{-\frac \ld2 |z-\xi_j| } G^D_{\xi_1,\dots,\xi_k} (z,y) dz.
\label{k3} 
\end{align}

Denoting by $u_f=\mathcal{L}^{-1}_{\xi_1,\dots,\xi_k} (\P_{V_{\xi_1,\dots,\xi_k}} f)$ the solution to the elliptic equation associated with $\mathcal{L}_{\xi_1,\dots,\xi_k}=-\dx^2+1-2\sum_{j=1}^k \text{sech}^2(\cdot-\xi_j)$, the elliptic regularity theorem yields
\begin{align}
\| \dx^2 u_f \|_{L^2}^2+ \| u_f \|_{L^2}^2 \le C \| \P_{V_{\xi_1,\dots,\xi_k}} f\|_{L^2}^2 
\label{k2}
\end{align}

\noi
for any $f \in L^2$. By using the Sobolev embedding and \eqref{k2}, we have that for $f=e^{-\frac \ld2 |z-\xi_j| }$,
\begin{align}
\sup_{y\in [-L_\eps, L_\eps] }\bigg| \int_{-L_\eps}^{L_\eps}   e^{-\frac \ld2 |z-\xi_j| } G^D_{\xi_1,\dots,\xi_k} (z,y)   dz \bigg|  &\le C \|  u_f \|_{H^{\frac 12+}}  \le C \| u_f \|_{H^2} 
 \le C \| \P_{V_{\xi_1,\dots,\xi_k}}   f\|_{L^2}.
\label{k4}
\end{align}

\noi
Combining \eqref{k7}, \eqref{k3} and \eqref{k4} yields 
\begin{align}
\bigg| \int_{-L_\eps}^{L_\eps}
G^D_{ \text{OU}, \xi_1,\dots,\xi_k}(x,z)W_j(z) G^D_{\xi_1,\dots,\xi_k} (z,y) dz  \bigg| \les e^{-\frac \ld4 |x-y|}.
\label{k11} 
\end{align}

\textbf{Case 2}: if $|y-\xi_j|  \ge \frac 12 |x-y|$, then 
\begin{align}
\int_{-L_\eps}^{L_\eps}
G^D_{\xi_1,\dots,\xi_k} (x,z) e^{-\ld |z-\xi_j|}  e^{-\ld |y-z|}  dz  &\les \int_{-L_\eps}^{L_\eps}
G^D_{\xi_1,\dots,\xi_k} (x,z) e^{-\ld |z-\xi_j|}   e^{- \frac \ld2 (|y-\xi_j |- |z-\xi_j| )      }   dz \notag \\
&  \les     e^{- \frac \ld4 |x-y|   }     \int_{-L_\eps }^{L_\eps } G^D_{\xi_1,\dots,\xi_k} (x,z)    e^{ -\frac \ld2 |z-\xi_j| } dz.
\label{k5}
\end{align}

\noi
Using the elliptic regularity theorem \eqref{k2}, and then following \eqref{k4}, we obtain
\begin{align}
\sup_{x \in [-L_\eps, L_\eps] } \bigg| \int_{-L_\eps}^{L_\eps}   G^D_{\xi_1,\dots,\xi_k} (x,z) e^{-\frac \ld2 |z-\xi_j| } dz \bigg| \les 1.
\label{k6}
\end{align}

\noi
Combining \eqref{k8}, \eqref{k5} and \eqref{k6} yields 
\begin{align}
\bigg| \int_{-L_\eps}^{L_\eps}
G^D_{\xi_1,\dots,\xi_k} (x,z) W_j(z) G^D_{ \text{OU}, \xi_1,\dots,\xi_k}(z,y) dz  \bigg| \les e^{-\frac \ld4 |x-y|}.
\label{k12}
\end{align}

\noi
We conclude the case study.

%\begin{align}
%G^D_{\xi_1,\dots,\xi_k}(x,y)&= G^D_{ \text{OU}, \xi_1,\dots,\xi_k}(x,y)-2\sum_{j=1}^k (G^D_{ \text{OU}, \xi_1,\dots,\xi_k} W_j G^D_{\xi_1,\dots,\xi_k} )(x,y).
%\label{k15}
%\end{align}

We now go back to the resolvent identities~\eqref{k1} and~\eqref{k10}. If $|y-\xi_j|  \le \frac 12 |x-y| $, then we use \eqref{k11} to get
\begin{align}
|G^D_{ \text{OU}, \xi_1,\dots,\xi_k} W_j G^D_{\xi_1,\dots,\xi_k}(x,y)| \les e^{-c |x-y|}.
\label{k13}
\end{align}

\noi
If $|y-\xi_j|  \ge \frac 12 |x-y|$, then using the commutator estimate \eqref{com1}, we switch the order
\begin{align*}
G^D_{ \text{OU}, \xi_1,\dots,\xi_k} W_j G^D_{\xi_1,\dots,\xi_k}(x,y)=   G^D_{\xi_1,\dots,\xi_k} W_j  G^D_{ \text{OU}, \xi_1,\dots,\xi_k  }(x,y)+R_j(x,y), 
\end{align*}

\noi
where $|R_j(x,y)| \les e^{-c d_\eps}$. This, along with \eqref{k12}, implies that
\begin{align}
|G^D_{ \text{OU}, \xi_1,\dots,\xi_k} W_j G^D_{\xi_1,\dots,\xi_k}(x,y)| \les e^{-c|x-y|}+e^{-c|\log (\eps \log \frac 1\eps) |}.
\label{k14}
\end{align}

\noi
Hence, combining \eqref{k1}, \eqref{k13}, \eqref{k14}, and Lemma \ref{LEM:OUdec} yields 
\begin{align*}
G^D_{\xi_1,\dots,\xi_k}(x,y)=c e^{-|x-y|}+ O\big(e^{-2(L_\eps -\max \{ |x|, |y| \}  )} \big).
\end{align*}

\noi
We now prove the diagonal regime \eqref{diag}. From the resolvent identity \eqref{k1} and Lemma \ref{LEM:OUdec} (ii), we have that for any $x\in [-L_\eps, L_\eps]$, $0 \le G^D_{\xi_1,\dots,\xi_k}(x,x) \le  G^D_{ \text{OU}, \xi_1,\dots,\xi_k}(x,x)  \les 1$,  uniformly in $\xi_1,\dots,\xi_k$ since $G^D_{ \text{OU}, \xi_1,\dots,\xi_k}$ and  $G^D_{\xi_1,\dots,\xi_k}  $ are positive operators and $W_j \ge 0$.  This completes the proof of Proposition \ref{PROP:dec}.

%\begin{align*}
%G^D_{\xi_1,\dots,\xi_k}(x,y)&= G^D_{ \text{OU}, \xi_1,\dots,\xi_k}(x,y)+( G^D_{\xi_1,\dots,\xi_k} W G^D_{ \text{OU}, \xi_1,\dots,\xi_k}  )(x,y)\\
%&=G^D_{ \text{OU}, \xi_1,\dots,\xi_k}(x,y)+   \int_{-L_\eps}^{L_\eps}
%G^D_{\xi_1,\dots,\xi_k} (x,z) W(z) G^D_{ \text{OU}, \xi_1,\dots,\xi_k} (z,y) dz,
%\end{align*}

\end{proof}

Before concluding this subsection, we present the proof of the following lemma under the separation condition.

\begin{lemma}\label{LEM:CT}
Under $\min_{\ell \neq i} |\xi_\ell -\xi_i| \ge d_\eps=|\log(\eps \log \frac 1\eps)|$, we have 
\begin{align*}
\big\| \big(G^D_{ \text{OU}, \xi_1,\dots,\xi_k} W_j G^D_{\xi_1,\dots,\xi_k} \big) - \big( G^D_{\xi_1,\dots,\xi_k}  W_j  G^D_{ \text{OU}, \xi_1,\dots,\xi_k} \big)   \big\|_{L^2 \to L^2} \les e^{-c \min_{\ell \neq i} |\xi_\ell -\xi_i|  } \to 0, 
\end{align*}

\noi
as $\eps \to 0$, where $W_j=\textup{sech}^2(\cdot-\xi_j)$.

\end{lemma}

%=\P_{V_{\xi_1,\dots, \xi_k} } \big(-\dx^2+1-2\sum_{j=1}^k \text{sech}^2(\cdot-\xi_j) \big)^{-1} \P_{V_{\xi_1,\dots, \xi_k} }(x,y)$

\begin{proof}

Recall $G^D_{\xi_1,\dots,\xi_k}(x,y)$, defined in \eqref{Cdec} and \eqref{Cdec0}. Using the resolvent identity, we write 
\begin{align}
G^D_{\xi_1,\dots,\xi_k}(x,y)&= G^D_{ \xi_j}(x,y)- G^D_{ \xi_j}\Big(\sum_{\ell \neq j} W_\ell \Big) G^D_{\xi_1,\dots,\xi_k}(x,y), 
\label{CT0} 
\end{align}

\noi
where $G^D_{ \xi_j}(x,y)=\P_{V_{\xi_1,\dots, \xi_k} } \big(-\dx^2+1-2\text{sech}^2(\cdot-\xi_j) \big)^{-1} \P_{V_{\xi_1,\dots, \xi_k} }(x,y)$. Again, using the single–bump resolvent identities
\begin{align*}
G^D_{\xi_j}&=G^D_{ \text{OU}, \xi_1,\dots,\xi_k}+G^D_{ \text{OU}, \xi_1,\dots,\xi_k} W_j G^D_{\xi_j}\\
G^D_{\xi_j}&=G^D_{ \text{OU}, \xi_1,\dots,\xi_k}+  G^D_{\xi_j} W_j  G^D_{ \text{OU}, \xi_1,\dots,\xi_k}, 
\end{align*}

\noi
we have 
\begin{align}
G^D_{ \text{OU}, \xi_1,\dots,\xi_k} W_j G^D_{\xi_j}=G^D_{\xi_j} W_j G^D_{ \text{OU}, \xi_1,\dots,\xi_k}.
\label{CT1}
\end{align}

\noi 
Inserting the resolvent identity \eqref{CT0} into both $G^D_{ \text{OU}, \xi_1,\dots,\xi_k} W_j G^D_{\xi_1,\dots,\xi_k}$ and $G^D_{\xi_1,\dots,\xi_k}  W_j  G^D_{ \text{OU}, \xi_1,\dots,\xi_k}$ and using \eqref{CT1}, we obtain
\begin{align}
&G^D_{ \text{OU}, \xi_1,\dots,\xi_k} W_j G^D_{\xi_1,\dots,\xi_k}  -  G^D_{\xi_1,\dots,\xi_k}  W_j  G^D_{ \text{OU}, \xi_1,\dots,\xi_k} \notag \\
&=\sum_{\ell \neq j} 
\Big(G^D_{ \text{OU}, \xi_1,\dots,\xi_k} W_j G^D_{ \xi_j} W_\ell  G^D_{\xi_1,\dots,\xi_k} -G^D_{\xi_1,\dots,\xi_k} W_\ell  G^D_{ \xi_j} W_j G^D_{ \text{OU}, \xi_1,\dots,\xi_k} \Big) \notag \\
&=\sum_{\ell \neq j} (T_{j \ell }+ T_{\ell j} ).
\label{CT2}
\end{align}

\noi 
By Combes–Thomas bounds for 1D Schrödinger resolvents with a positive mass term, for any bounded $V$ such that the inverse $(-\dx^2+1+V)^{-1}$ exists, we have 
\begin{align*}
\| \ind_A   (-\dx^2+1+V)^{-1}   \ind_B \|_{L^2\to L^2} \les e^{-c \cdot \text{dist}(A,B) }.
\end{align*}

\noi
This implies that 
\begin{align}
\|  W_j G^D_{ \xi_j} W_\ell   \|_{L^2 \to L^2} \les e^{-c |\xi_j-\xi_\ell|}
\label{CT3}
\end{align}

\noi
since $W_j=\text{sech}^2(\cdot-\xi_j)$ is highly localized around $\xi_j$ with an exponentially decaying tail.  Using \eqref{CT3}, the operator $T_{j\ell}$ in \eqref{CT2} satisfies
\begin{align}
\| T_{j \ell} \|_{L^2 \to L^2} \le \| G^D_{ \text{OU}, \xi_1,\dots,\xi_k} \|_{L^2 \to L^2}   \|  W_j G^D_{ \xi_j} W_\ell   \|_{L^2 \to L^2} \| G^D_{\xi_1,\dots,\xi_k } \|_{L^2 \to L^2} \les e^{-c |\xi_j-\xi_\ell|}
\label{CT4}
\end{align}

Combining \eqref{CT2} and \eqref{CT4} yields 
\begin{align*}
\| G^D_{ \text{OU}, \xi_1,\dots,\xi_k} W_j G^D_{\xi_1,\dots,\xi_k}  -  G^D_{\xi_1,\dots,\xi_k}  W_j  G^D_{ \text{OU}, \xi_1,\dots,\xi_k} \|_{L^2\to L^2}&\les  \sum_{\ell \neq j}  \big(\|T_{j\ell}  \|_{L^2  \to L^2} +\|T_{\ell j}  \|_{L^2 \to L^2} \big)\\
&\les \sum_{\ell \neq j} e^{-c |\xi_j -\xi_\ell |} \les e^{-c d_\eps},
\end{align*}

\noi
where $d_\eps=\min_{\ell \neq j}|\xi_\ell -\xi_j| \ges  |\log (\eps \log \frac 1\eps) | $. This completes the proof of Lemma \ref{LEM:CT}.

\end{proof}

\subsection{Maximum of the Gaussian process}

In this subsection, we study the maximal behavior of the field $v$ under the Gaussian measure
\begin{align*}
d\nu^\perp_{ \xi_1,\dots,\xi_k}=Z_{\xi_1,\dots,\xi_k}^{-1} e^{-\frac 12 \jb{C_{\xi_1,\dots, \xi_k}^{-1} v,v  } }\prod_{x \in [-L_\eps, L_\eps] } dv(x),
\end{align*}

\noi
defined in Lemma \ref{LEM:Gauss}, with covariance operator $C_{\xi_1,\dots, \xi_k}$ \eqref{cov}.  In the following proposition, we show that the typical size of $v$ under the measure $\nu^\perp_{\xi_1,\dots,\xi_k}$ satisfies $ \| v\|_{L^\infty}  \le K \sqrt{|\log L_\eps|}  $  for some sufficiently large $K \ge 1$.

\begin{lemma}\label{LEM:T1}
There exist a constant $c>0$ such that for sufficiently large $K$, we have
\begin{align}
\nu^\perp_{\xi_1,\dots,\xi_k}(A_\eps^c) \les e^{-cK^2 |\log L_\eps | }, 
\label{max1}
\end{align}

\noi
uniformly in $\xi_1,\dots,\xi_k \in [-\cj L_\eps, \cj L_\eps]^k$, where 
\begin{align*}
A_\eps:=\{   \|v \|_{L^\infty (  [-L_\eps, L_\eps] ) } \le K \sqrt{|\log L_\eps|}   \}.
\end{align*}

\end{lemma}

\begin{proof}

We first  show that there exists a constant $C>0$ such that for any finite $p \ge 2$, we have 
\begin{align}
\E_{\nu^\perp_{\xi_1,\dots,\xi_k}} \big[ | v(x)-v(y)|^p \big] \le C p^{\frac p2} |x-y|^{\frac p2},
\label{vmax00}
\end{align}

\noi
uniformly in $(\xi_1,\dots,\xi_k) \in  [-\cj L_\eps, \cj L_\eps]^k $ and $\cj L_\eps$. 
Since $\nu^\perp_{\xi_1,\dots,\xi_k }$ is a Gaussian
\begin{align}
\E_{\nu^\perp_{  \xi_1,\dots,\xi_k}    } \big[  |v(x)-v(y)|^p  \big] \le p^{\frac p2} \Big(  \E_{\nu^\perp_{\xi_1,\dots,\xi_k}   } \big[ | v(x)-v(y)|^2  \big]     \Big)^{\frac p2}.
\label{vmax0}
\end{align}

\noi
By expanding the second moment, we have
\begin{align}
\E_{\nu^\perp_{  \xi_1,\dots,\xi_k }   } \big[ | v(x)- v(y)|^2  \big] 
% &=\big(\E_{ \nu^\perp_{\xi_1,\dots,\xi_k}    } \big[ v(x) v(x) \big]-\E_{\nu^\perp_{\xi_1,\dots,\xi_k} } \big[ v(x)  v(y) \big] \big) \notag \\
% &\hphantom{X}+\big(\E_{ \nu^\perp_{\xi_1,\dots,\xi_k}    } \big[ v(y)  v(y) \big]-\E_{ \nu^\perp_{\xi_1,\dots,\xi_k}   } \big[ v(y) v(x)  \big] \big) \notag \\
% &=\I_1+\I_2,
=G^D_{\xi_1,\dots, \xi_k}(x,x) -G^D_{\xi_1,\dots, \xi_k}(x,y)
+G^D_{\xi_1,\dots,\xi_k}(y,y) -G^D_{\xi_1,\dots,\xi_k }(y,x)
\label{vmax1} 
\end{align}

% \noi
% where
% \begin{align*}
% \I_1&= G^D_{\xi_1,\dots, \xi_k}(x,x) -G^D_{\xi_1,\dots, \xi_k}(x,y)  \\
% \I_2&= G^D_{\xi_1,\dots,\xi_k}(y,y) -G^D_{\xi_1,\dots,\xi_k }(y,x).    
% \end{align*}

\noi
Here $G^D_{\xi_1,\dots, \xi_k}$ is the Green function defined in \eqref{Cdec0}, corresponding to the covariance operator $\mathcal{C}_{\xi_1,\dots,\xi_k}$  given in \eqref{cov}. Using the elliptic regularity theorem, one can show that
\begin{align}
\sup_{x \in [-L_\eps, L_\eps]} | G^D_{\xi_1,\dots,\xi_k}(x,y)-G^D_{\xi_1,\dots,\xi_k}(x,z)  | \le c |y-z| 
% \label{RG0} \\
% \sup_{y \in [-L_\eps, L_\eps]} | G^D_{\xi_1,\dots,\xi_k}(x,y)-G^D_{\xi_1,\dots,\xi_k}(z,y)  | &\le c |x-z|, \label{RG1}
\end{align}

\noi
uniformly in $\xi_1,\dots, \xi_k \in [-\cj L_\eps, \cj L_\eps ]^k$ and even in $L_\eps$, and the same holds for the other variable.
So \eqref{vmax1} is bounded by $ c |x-y|$,
thus \eqref{vmax00} follows.
Once we have \eqref{vmax00}, the maximal behavior of the Gaussian field follows from Talagrand’s chaining argument. See \cite[Proposition 5.1]{SS24} or \cite[Chapter 2]{TAL}. 

\end{proof}

In the previous lemma, we showed that the typical behavior of the field
$v$ under the Gaussian measure $\nu^\perp_{\xi_1,\dots,\xi_k}$  is captured by the events $A_\eps=\{ \|v \|_{L^\infty} \le K \sqrt{|\log L_\eps|} \}$. In the following, we study the error estimate outside the typical behavior $A_\eps$. 

\begin{lemma}\label{LEM:AEC}
Let $F$ be a bounded, continuous function and $\dl_\eps=\eta\sqrt{\eps \log \frac 1\eps}$, defined in \eqref{LDPIN2}. Then
\begin{align*}
\E_{ \nu^\perp_{ \xi_1,\dots,\xi_k } }\Big[F(v) e^{\frac 1\eps \mathcal{E}_{\xi_1,\dots,\xi_k}(v)  } \textup{Det}_{\xi_1,\dots,\xi_k}(\sqrt{\eps}v), \, A_\eps^c, \, \| \sqrt{\eps}v \|_{L^2} < \dl_\eps   \Big] \les e^{-c K^2 |\log L_\eps| },
\end{align*}

\noi
uniformly in $\xi_1, \dots,\xi_k \in [-\cj L_\eps, \cj L_\eps]^k $, where  $\mathcal{E}_{\xi_1,\dots,\xi_k}$ and $\textup{Det}_{\xi_1,\dots,\xi_k}$ are defined in Proposition \ref{PROP:cha}.

\end{lemma}

\begin{proof}

%\begin{align*}
%\E_{\nu^\perp_{m_{\xi_1,\dots,\xi_k}  } } \Big[   \|v \|_{L^\infty}^{k p}, \, A_\eps^c, \, B_\eps  \Big] \le  \E_{\nu^\perp_{m_{\xi_1,\dots,\xi_k}  } } \Big[   \|v \|_{L^\infty}^{2k p}  \Big]^{\frac 12} 
%\end{align*}

%\noi
%where $A_\eps^c=\{ \| v\|_{L^\infty} \ge K \sqrt{|\log L_\eps|} \}$

We decompose the event $A_\eps^c=\{ \| v\|_{L^\infty} \ge K \sqrt{|\log L_\eps|} \}$ into dyadic shells
\begin{align*}
&\E_{\nu^\perp_{m_{\xi_1,\dots,\xi_k}}  } \Big[ e^{\frac 1\eps \int_{-L_\eps}^{L_\eps} (\sqrt{\eps}v)^3 dx } \cdot \textup{Det}_{\xi_1,\dots,\xi_k}(\sqrt{\eps}v) ,  \, A_\eps^c, \,  \| \sqrt{\eps}v \|_{L^2} < \dl_\eps   \Big]\\
&= \sum_{\ell \ge K} \E_{\nu^\perp_{m_{\xi_1,\dots,\xi_k}}  } \Big[ e^{\frac 1\eps \int_{-L_\eps}^{L_\eps} (\sqrt{\eps}v)^3 dx } \cdot \textup{Det}_{\xi_1,\dots,\xi_k}(\sqrt{\eps}v) ,  \, G_\ell, \,  \| \sqrt{\eps}v \|_{L^2} < \dl_\eps  \Big],
\end{align*}

\noi
where $\ell$ ranges over the dyadic numbers $\ell=2^j K$, $ j \ge 0$, so that the sum runs over $\ell=K,2K,4K,\dots$, and  $G_\ell=\big\{ \ell \sqrt{| \log L_\eps|} \le \|v \|_{L^\infty} <2\ell \sqrt{|\log L_\eps|}   \big\}$.  Since $\textup{Det}_{\xi_1,\dots,\xi_k}(\sqrt{\eps}v)=1+O(\| \sqrt{\eps}v \|_{L^\infty}^k)$ and $\| v\|_{L^2}<\eta \sqrt{\log \frac 1\eps}$,  using Hölder’s inequality and Lemma \ref{LEM:T1}, we obtain
\begin{align*}
&\sum_{\ell \ge K}  \bigg(\E_{\nu^\perp_{m_{\xi_1,\dots,\xi_k}}  } \Big[ e^{\frac 1\eps \int_{-L_\eps}^{L_\eps} (\sqrt{\eps}v)^3 dx },  \, G_\ell, \,  \| \sqrt{\eps}v \|_{L^2} < \dl_\eps  \Big] \bigg)^\frac 12 \nu^\perp_{m_{\xi_1,\dots,\xi_k} }(G_\ell)^\frac 12\\
&\les \sum_{\ell \ge K} e^{\sqrt{\eps} \log \frac 1\eps \cdot \ell \sqrt{|\log L_\eps| } } \cdot  (1+\eps^{\frac k2} \ell^k |\log L_\eps|^{\frac k2}  ) \cdot     e^{-c \ell^2 |\log L_\eps| }\\
&\les \sum_{\ell \ge K} e^{-\frac c2 \ell^2 |\log L_\eps| }
\les  e^{-c K^2 |\log L_\eps|  },
\end{align*}

\noi 
uniformly in $\xi_1, \dots,\xi_k \in [-\cj L_\eps, \cj L_\eps]^k $ and $L_\eps \ge 1$.

\end{proof}

\section{Ornstein–Uhlenbeck limit on normal space}

In this section, we present the central limit theorem on the normal space to the Ornstein–Uhlenbeck measure under the conditions: 
\begin{align*}
\textup{dist}(\supp g, \xi_j) \ge \eps^{-\frac 12 +2\eta}
\end{align*}

%\begin{itemize}
%\item[(i)] $

%\medskip 

%\item[(ii)] Tangential–mode localization: $|\xi_j| \le L_\eps -|\log \eps|$

%\end{itemize}

\noi
for every $1 \le j \le k$. This condition ensures that the test function $g$  is at a distance at least  $\eps^{-\frac 12 +2\eta }$ from all solitons $m^\eps_{\xi_1},\dots, m^\eps_{\xi_k}$.

%The second bulk localization condition, which keeps the centers $\xi_j$ away from the boundary, is required to apply the correlation decay in Proposition \ref{PROP:dec}.

%$ \min_{ \ell \neq i}|\xi_\ell-\xi_i | \ge d_\eps\sim |\log  (\eps \log \frac 1 \eps) |$,

\begin{proposition}\label{PROP:OUN}
Let $L_\eps=\eps^{-\frac 12+\eta}$ and let $g$ be a real-valued, smooth, compactly supported function. Under the conditions   
\begin{align}
\textup{dist}(\supp g, \xi_j) \ge \eps^{-\frac 12 +2\eta} 
%\quad \textup{and} \quad |\xi_j| \le L_\eps -|\log \eps|
\label{twocond}
\end{align}

\noi
for every  $1 \le j \le k$,  we have 
\begin{align}
&\E_{\nu^\perp_{\xi_1,\dots,\xi_k } } \Big[ e^{i \jb{v,g} } e^{\frac 1\eps \mathcal{E}_{\xi_1,\dots,\xi_k}( \sqrt{\eps}v ) } \textup{Det}_{\xi_1,\dots,\xi_k}(\sqrt{\eps}v)  , \; A_\eps, \;  \| \sqrt{\eps}v\|_{L^2}<\dl_\eps  \Big] \notag \\ 
&=\exp\Big\{-\frac 12 \jb{g, (-\dx^2+1)^{-1} g }_{L^2(\R)} \Big \} \cdot(1+O(\eps^{0+})), 
\label{Glimit}
\end{align}

\noi
uniformly in $\xi_1,\dots,\xi_k \in [-\cj L_\eps, \cj L_\eps]$ and $\cj L_\eps$. Here  $A_\eps=\big\{ \|v \|_{L^\infty} \le K \sqrt{|\log L_\eps|} \big\}$, defined in Lemmas \ref{LEM:T1}, and $\dl_\eps=\sqrt{\eps \log \frac 1\eps}$ in \eqref{LDPIN2}.
\end{proposition}

\begin{remark}\rm 
If the tangential modes $\xi_j$ do not satisfy the condition  $\textup{dist}(\supp g, \xi_j) \ge \eps^{-\frac 12 +2\eta}$ for every $j=1,\dots,k$ ($g$ is close to one of the solitons), then the central limit theorem to the mean–zero Ornstein–Uhlenbeck measure does not hold. However, in the next section, we show that the contribution of the forbidden region $\textup{dist}(\supp g, \xi_j)< \eps^{-\frac 12+2\eta} $ becomes relatively small compared to the size of the partition function as $L_\eps \to 0$.  Hence, we can still obtain the desired result with a more careful geometric analysis.
\end{remark}

\subsection{Asymptotic behavior of the covariance}

In order to prove Proposition \ref{PROP:OUN},  we first take some preliminary steps. With respect to the measure $\nu^\perp_{\xi_1,\dots,\xi_k } $, we perform the following orthogonal decomposition
\begin{align}
v(x)= \P(x) \jb{v,g}+w(x),
\label{orth1}
\end{align}

\noi
where $\P(x)$ is the projection of $v$ onto the direction of $\jb{v,g}$ 
\begin{align*}
\P(x)=\frac{ \E_{\nu^\perp_{\xi_1,\dots,\xi_k} } \big[ \jb{v,g}v(x)   \big]   }{\E_{\nu^\perp_{\xi_1,\dots,\xi_k} } \big[ |\jb{v,g}|^2  \big]}.
\end{align*}

\noi
The orthogonal decomposition implies that $w$ and $\jb{v,g}$ are independent Gaussian random variables. Hence, the measure $\nu^\perp_{\xi_1,\dots,\xi_k}$ can be decomposed as follows
\begin{align*}
d\nu^\perp_{\xi_1,\dots,\xi_k}(v)= \frac{1}{\sqrt{2\pi}\s_g } e^{-\frac {t^2}{2\s_g^2 } } \mathrm{d}t\, \mathrm{d} \nu^{\perp \perp }_{\xi_1,\dots,\xi_k}(w),
\end{align*}

\noi
where $\s_g^2=\E_{\nu^\perp_{\xi_1,\dots,\xi_k}} \big[  |\jb{v,g}|^2  \big]$.   Before presenting the proof of Proposition \ref{PROP:OUN}, we study the following lemma, which provides control on the variance $\s_g^2$.

%and the projection coefficient $\P(x)$.

\begin{lemma}\label{LEM:cov}
Let $g$ be a smooth, compactly supported function. Under the conditions   
\begin{align*}
\textup{dist}(\supp g, \xi_j) \ge \eps^{-\frac 12 +2\eta} 
%\quad \textup{and} \quad |\xi_j| \le L_\eps -|\log \eps|
\end{align*}

\noi
for every  $1 \le j \le k$,  we have 
\begin{align*}
\s_g^2=\E_{\nu^\perp_{\xi_1,\dots,\xi_k } } \big[ |\jb{v,g}|^2 \big]=\jb{g, (-\dx^2+1)^{-1} g}\cdot \big(1+O(e^{-c \eps^{-\frac 12+2\eta} }) \big),
\end{align*}

\noi
uniformly in $\xi_1,\dots,\xi_k \in [- \cj L_\eps, \cj L_\eps] $ and $\cj L_\eps$.

%uniformly in $(\xi_1,\dots,\xi_k) \in [-L_\eps, L_\eps]^k$ and $L_\eps$. 

%Furthermore, the function $\P(x)$ associated with the projection admits the following estimate
%\begin{align*}
%\|\P \|_{L^p(\R)} \les 1
%\end{align*}

%\noi
%for any $1 \le p< \infty$.

\end{lemma}

\begin{proof}

By the definition of the Gaussian measure $\nu^\perp_{ \xi_1,\dots,\xi_k }$, we have
\begin{align}
\E_{\nu^\perp_{ \xi_1,\dots,\xi_k } } \big[ |\jb{v,g}|^2 \big]=\jb{g,G^D_{\xi_1,\dots,\xi_k} g},
\label{v0}
\end{align}

\noi
where $G^D_{\xi_1,\dots,\xi_k}$ is the Green's function, defined in \eqref{Cdec0}.  
Recall the resolvent identity \eqref{k1} for $G^D_{\xi_1,\dots,\xi_k}$.  We first study the leading order term $G_{\text{OU},\xi_1,\dots,\xi_k}^D$ in \eqref{k1}.  By expanding the projection  $\P_{V_{\xi_1\,\dots,\xi_k}}$ as in \eqref{dec0},  we begin with the first projected term $\P_i(-\dx^2+1)^{-1} \P_j$. Note that 
\begin{align*}
\jb{g, \P_i(-\dx^2+1)^{-1} \P_j g}=\jb{g, t_i  } \jb{g, t_j  } \big \langle t_i, (-\dx^2+1)^{-1} t_j  \big\rangle.
\end{align*}

\noi
Under the condition $\textup{dist}(\supp g, \xi_j) \ge \eps^{-\frac 12 +2\eta}$, together with Lemma \ref{LEM:TGRAM} $|t_j(x)| \les e^{-|x-\xi_j|}$,
\begin{align}
|\jb{g, t_j  }| \les \bigg| \int_{\supp g} g(x)  e^{-|x-\xi_j|}  dx \bigg| \les \|g \|_{L^2} e^{-\textup{dist}(\supp g, \xi_j) } \les e^{- \eps^{-\frac 12+2\eta}}.
\label{v5}
\end{align}

\noi
This implies that 
\begin{align}
|\jb{g, \P_i(-\dx^2+1)^{-1} \P_j g} | \les  e^{- \eps^{-\frac 12+2\eta}}.
\label{v7}
\end{align}

\noi
Furthermore, the separation condition \eqref{v5} implies 
\begin{align}
|\jb{g, \P_i(-\dx^2+1)^{-1} g} | \les  e^{- \eps^{-\frac 12+2\eta}} \quad \textup{and} \quad  |\jb{g, (-\dx^2+1)^{-1} \P_j g} |   \les e^{- \eps^{-\frac 12+2\eta}}.
\label{v6}
\end{align}

\noi 
Combining \eqref{dec0}, \eqref{v7}, and \eqref{v6} yields 
\begin{align}
\jb{g, G^D_{\textup{OU}, \xi_1,\dots, \xi_k }  g }=\jb{g, (-\dx^2+1)^{-1}g}+O(e^{-\eps^{-\frac 12+2\eta}}).
\label{v15}
\end{align}

\noi
This completes the leading order term $G_{\text{OU},\xi_1,\dots,\xi_k}^D$ in \eqref{k1}.

We now study the perturbation term $G_{\text{OU}, \xi_1,\dots,\xi_k}^D W G^D_{\xi_1,\dots,\xi_k}$ in \eqref{k1}.  We expand 
\begin{align}
&|\jb{g, (G_{\text{OU}, \xi_1,\dots,\xi_k }^D W G^D_{\xi_1,\dots,\xi_k} )  g}| \notag \\
& \les \sum_{j=1}^{k } \bigg|  \int_{-L_\eps}^{L_\eps} g(x) \bigg( \int_{-L_\eps}^{L_\eps} G_{\text{OU},\xi_1,\dots,\xi_k}^D(x,z) e^{-|z-\xi_j|} \bigg(\int_{-L_\eps}^{L_\eps}  G^D_{\xi_1,\dots,\xi_k}(z,y) g(y) dy\bigg) dz \bigg) dx  \bigg|,
\label{v11}
\end{align}

\noi
where we used $\textup{sech}^2(z-\xi_j)\les e^{-|z-\xi_j|}$. To use the correlation decay, we separate $G_{\text{OU},\xi_1,\dots,\xi_k}^D(x,z)$ and $G^D_{\xi_1,\dots,\xi_k}(z,y)$ into its bulk and boundary parts,  based on Lemma \ref{LEM:OUdec} and Proposition \ref{PROP:dec}
We first focus on the bulk effect  $e^{-|x-z|}$ and $e^{-|z-y|}$ by plugging them into \eqref{v11}
\begin{align}
&\sum_{j=1}^{k } \bigg|  \int_{-L_\eps}^{L_\eps} g(x) \bigg( \int_{-L_\eps}^{L_\eps} e^{-|x-z|} e^{-|z-\xi_j|} \bigg(\int_{-L_\eps}^{L_\eps}  e^{-|z-y|} g(y) dy \bigg) dz \bigg) dx  \bigg| \notag \\
& \les \sum_{j=1}^{k } L_\eps \bigg|  \int_{\supp g} \int_{\supp g}   e^{-\frac 14|x-\xi_j|} e^{-\frac 14  |y-\xi_j| }   g(y) g(x) dy dx    \bigg| \notag \\
&\les \sum_{j=1}^{k} L_\eps e^{-\frac 12 \text{dist}(\supp g, \xi_j)  }  \|g \|_{L^1} \| g \|_{L^1}  \les  L_\eps  e^{- \frac 12 \eps^{-\frac 12+2\eta}}  \les e^{- \frac 14 \eps^{-\frac 12+2\eta}},
\label{v14}
\end{align}

\noi
where in the last step, we use the condition  $\textup{dist}(\supp g, \xi_j) \ge \eps^{-\frac 12 +2\eta} $ for every $1 \le j \le k$.  For the boundary effects $e^{-2(L_\eps-\max\{ |x|,|z|\}    )}$ and $e^{-2(L_\eps-\max\{ |z|, |y| \} )}$ in Lemma \ref{LEM:OUdec} and Proposition \ref{PROP:dec}, following the arguments in \eqref{sta1} and \eqref{sta3}, we obtain the error term
$e^{-c \eps^{-\frac 12+2\eta} }$. Combining \eqref{v11},  \eqref{v14}, and the boundary error $e^{-c \eps^{-\frac 12+2\eta} }$ yields 
\begin{align}
|\jb{g, (G_{\text{OU}, \xi_1,\dots,\xi_k }^D W G^D_{\xi_1,\dots,\xi_k} )  g}| \les e^{- \frac 14 \eps^{-\frac 12+2\eta}}+ e^{-c \eps^{-\frac 12+2\eta} }
\label{v16}
\end{align}

\noi
as $\eps \to 0$. Finally, using \eqref{v0}, \eqref{k1}, \eqref{v15}, and \eqref{v16}, we obtain
\begin{align*}
\E_{\nu^\perp_{ \xi_1,\dots,\xi_k } } \big[ |\jb{v,g}|^2 \big]&=\jb{g,G^D_{\xi_1,\dots,\xi_k} g}\\
&=\jb{g, G^D_{\textup{OU},\xi_1,\dots,\xi_k }g }+\jb{g, (G_{\text{OU}, \xi_1,\dots,\xi_k }^D W G^D_{\xi_1,\dots,\xi_k} )  g}\\
&=\jb{g, (-\dx^2+1)^{-1}g}(1+O( e^{-c \eps^{-\frac 12+2\eta} }  ) ).
\end{align*}

\noi
This completes the proof of Lemma \ref{LEM:cov}.

\end{proof}

\subsection{Ornstein–Uhlenbeck limit on the normal space}

We are now ready to prove Proposition \ref{PROP:OUN}.

\begin{proof}[Proof of Proposition \ref{PROP:OUN}]

We decompose the expectation as follows
\begin{align}
&\E_{\nu^\perp_{\xi_1,\dots,\xi_k} }\Big[ e^{i \jb{v,g} } e^{\frac 1\eps \mathcal{E}_{\xi_1,\dots,\xi_k}(\sqrt{\eps}v) }  \textup{Det}_{\xi_1,\dots,\xi_k}(\sqrt{\eps}v), \, A_\eps, \, \| \sqrt{\eps}v\|_{L^2}< \dl_\eps   \Big] \notag \\
&=\E_{\nu^\perp_{\xi_1,\dots,\xi_k} } \Big[ e^{i \jb{v,g} } \textup{Det}_{\xi_1,\dots,\xi_k}(\sqrt{\eps}v) , \; A_\eps, \;  \| \sqrt{\eps}v\|_{L^2}< \dl_\eps,  \Big] \notag \\
&\hphantom{X}+\E_{\nu^\perp_{\xi_1,\dots,\xi_k} } \Big[ e^{i \jb{v,g} } (e^{\frac 1\eps \mathcal{E}_{\xi_1,\dots,\xi_k}(\sqrt{\eps}v ) }-1) \textup{Det}_{\xi_1,\dots,\xi_k}(\sqrt{\eps}v), \; A_\eps, \;  \| \sqrt{\eps}v\|_{L^2}< \dl_\eps      \Big] \notag \\
&=\I_1+\I_2,
\label{CLT1}
\end{align}

\noi
where $\I_1 $  is the main term and $\I_2$ is an error term. On the set $A_\eps=\{  \|v\|_{L^\infty} \le K \sqrt{|\log L_\eps|}   \}$, the higher-order error term and the determinant in Proposition \ref{PROP:cha} satisfy
\begin{align*}
\Big| \frac 1\eps \mathcal{E}_{\xi_1,\dots,\xi_k}(\sqrt{\eps}v) \Big| \le \frac 1\eps \int_{-L_\eps}^{L_\eps} |\sqrt{\eps} v|^3 dx  &\le \sqrt{\eps} L_\eps \|v \|_{L^\infty}^3 \les \eps^{0+} |\log \eps|^{\frac 32}=\eps^{0+},
\end{align*}

\noi
where we used $L_\eps=\eps^{-\frac 12+}$, and
\begin{align}
\textup{Det}_{\xi_1,\dots,\xi_k}(\sqrt{\eps}v)=1+O(\| \sqrt{\eps}v \|_{L^\infty}^k)=1+O(\eps^{\frac k2-})
\label{M10}
\end{align}

\noi
uniformly in $\xi_1,\dots,\xi_k \in [-\cj L_\eps, \cj L_\eps]^k$. This implies that 
\begin{align}
\I_2=O(\eps^{0+}),
\label{CLT2}
\end{align}

\noi
uniformly in $\xi_1,\dots,\xi_k \in [-\cj L_\eps, \cj L_\eps]^k$.  Regarding the main term $\I_1$, we use \eqref{M10} and the tail probability estimates in Lemmas \ref{LEM:T1} and \eqref{LDPIN2} to obtain
\begin{align}
\I_1=\E_{\nu^\perp_{\xi_1,\dots,\xi_k } } \Big[ e^{i \jb{v,g} } \Big]+O(\eps^{\frac k2-}) +O(e^{-c K^2 |\log L_\eps| })+O(e^{- \frac{c \dl_\eps^2}{\eps}}),
\label{CLT3}
\end{align}

\noi
where $\dl_\eps=\eta\sqrt{\eps \log \frac 1\eps}$. Using the orthogonal decomposition $v=\P(x)\jb{v,g}+w$ in \eqref{orth1}, where $\jb{v,g}$ and $w$ are independent Gaussian, we write
\begin{align}
\E_{\nu^\perp_{ \xi_1,\dots,\xi_k } } \Big[ e^{i \jb{v,g} } \Big]&= \iint e^{it} e^{-\frac{t^2}{2\s_g^2}} \frac{dt}{\sqrt{2\pi} \s_g }\nu^{\perp \perp}_{\xi_1,\dots,\xi_k}(w)=e^{-\frac {\s_g^2}{2} } \notag \\
&=\exp\Big\{-\frac 12  \jb{g,   (-\dx^2+1)^{-1} g  } \cdot \big(1+O(e^{-c \eps^{-\frac 12+2\eta} }) \big)  \Big\},
\label{CLT4}
\end{align}

%\begin{align*}
%\s_g^2= \E_{\nu^\perp_{m_{\xi_1,\dots,\xi_k}} } \big[ |\jb{v,g}|^2 \big]=\jb{g, (-\dx^2+1)^{-1} g}\cdot \big(1+O(e^{-c \eps^{-\frac 12+2\eta} }) \big).
%\end{align*}

\noi
where we used Lemma \ref{LEM:cov}. By combining \eqref{CLT1}, \eqref{CLT2}, \eqref{CLT3}, and \eqref{CLT4}, we obtain 
\begin{align*}
&\E_{\nu^\perp_{\xi_1,\dots,\xi_k} } \Big[ e^{i \jb{v,g} } e^{\frac 1\eps \mathcal{E}_{\xi_1,\dots,\xi_k}( \sqrt{\eps}v ) }\textup{Det}_{\xi_1,\dots,\xi_k}(\sqrt{\eps}v), \; A_\eps, \;  \| \sqrt{\eps}v \|_{L^2} <\dl_\eps  \Big] \notag \\
&=\exp\Big\{-\frac 12  \jb{g,   (-\dx^2+1)^{-1} g  } \cdot \big(1+O(e^{-c \eps^{-\frac 12+2\eta} }) \big)  \Big\}+O(\eps^{0+}).
\end{align*}

\noi
Therefore, by taking the limit $\eps \to 0$, we have 
\begin{align*}
&\lim_{\eps \to 0} \E_{\nu^\perp_{\xi_1,\dots,\xi_k} } \Big[ e^{i \jb{v,g} } e^{\frac 1\eps \mathcal{E}_{\xi_1,\dots,\xi_k}( \sqrt{\eps}v ) }\textup{Det}_{\xi_1,\dots,\xi_k}(\sqrt{\eps}v), \; A_\eps, \;  \| \sqrt{\eps}v \|_{L^2}<\dl_\eps  \Big]\\
&=\exp\Big\{-\frac 12 \jb{g, (-\dx^2+1)^{-1} g }_{L^2(\R)} \Big \}.
\end{align*}

%\noi
%Regarding the denominator, by following the same computations as above, we have
%\begin{align}
%\E_{\nu^\perp_{m_{\xi_1,\dots,\xi_k}} } \Big[  e^{\frac 1\eps \mathcal{E}(m_{\xi_1,\dots,\xi_k}, \sqrt{\eps}v ) }, \; A_\eps, \; B_{\eps}   \Big]=1+ O(e^{-c K^2 |\log L_\eps| })+O(e^{- DL_{\eps}^2})+O(\eps^{0+}). 
%\label{CLT6}
%\end{align}

%\noi
%It follows from \eqref{CLT5} and \eqref{CLT6} that 
%\begin{align*}
%\lim_{\eps \to 0} \frac{\E_{\nu^\perp_{m_{\xi_1,\dots,\xi_k}} } \Big[ e^{i \jb{v,g} } e^{\frac 1\eps \mathcal{E}(m_{\xi_1,\dots,\xi_k}, \sqrt{\eps}v ) }, \; A_\eps, \;  B_{\eps}  \Big]} {\E_{\nu^\perp_{m_{\xi_1,\dots,\xi_k}} } \Big[  e^{\frac 1\eps \mathcal{E}(m_{\xi_1,\dots,\xi_k}, \sqrt{\eps}v ) }, \; A_\eps, \; B_{\eps}   \Big]  }=\exp\Big\{-\frac 12 \jb{g, (-\dx^2+1)^{-1} g }_{L^2(\R)} \Big \}.
%\end{align*}

\noi
This completes the proof of Proposition \ref{PROP:OUN}.

\end{proof}

\begin{remark}\rm 
Following the proof of Proposition \ref{PROP:OUN}, we also obtain
\begin{align}
\E_{\nu^\perp_{ \xi_1,\dots,\xi_k} } \Big[ e^{\frac 1\eps \mathcal{E}_{\xi_1,\dots,\xi_k}( \sqrt{\eps}v ) }\textup{Det}_{\xi_1,\dots,\xi_k}(\sqrt{\eps}v), \; A_\eps, \;  \| \sqrt{\eps}v \|_{L^2}<\dl_\eps  \Big] = 1+O(\eps^{0+}),
\label{part1}
\end{align}

\noi
uniformly in $\xi_1,\dots,\xi_k \in [-\cj L_\eps, \cj L_\eps]^k$, where $\dl_\eps=\eta\sqrt{\eps \log \frac 1\eps}$. Note that in \eqref{part1} we do not need the conditions $\textup{dist}(\supp g, \xi_j) \ge \eps^{-\frac 12 +2\eta}$  for $j=1,\dots k$. These conditions are only used in \eqref{CLT4}, where we used Lemma \ref{LEM:cov}. Note that if $e^{it}$ is replaced by $1$, then \eqref{CLT4} is immediately equal to $1$.

\end{remark}

\section{Proof of Theorem \ref{THM:2}}

In this section we prove Theorem \ref{THM:2}.  By studying the geometric structure of the forbidden sets $\textup{dist}(\supp g, \xi_j)<\eps^{-\frac 12+2\eta}$, we remove the conditions used in proving the Ornstein–Uhlenbeck limit (Proposition \ref{PROP:OUN}): $\textup{dist}(\supp g, \xi_j) \ge \eps^{-\frac 12 +2\eta}$ for every $1\le j \le k$.

\begin{proof}[Proof of Theorem \ref{THM:2}]

Let $g$ be a smooth, compactly supported function. It suffices to prove that 
\begin{align*}
\lim_{\eps \to 0} \int e^{i\jb{\phi,g }}  (T^\eps)_{\#} \rho_\eps^Q(d\phi)=\exp\Big\{-\frac 12 \jb{g, (-\dx^2+1)^{-1}g } \Big\},
\end{align*}

\noi
where $T^{\eps}(\phi)=\eps^{-\frac 12}(\phi-\pi^\eps(\phi) ) $ and $\pi^\eps$ denotes the projection onto the multi-soliton manifold $\M^{\eps, \ge  d_\eps}_Q$, defined in \eqref{APM5}. Using Lemma \ref{LEM:partition}, we decompose the integral into the large-deviation events and the main term $\J_1$ as follows
\begin{align}
\int e^{i\jb{T^{\eps}(\phi),g }}  \rho_\eps^Q(d\phi)=
\J_1+\textup{O}(e^{-c  \log \frac 1\eps })
\label{M9}
\end{align}

\noi
where
\begin{align*}
\J_1&=\int_{   \{  \textup{dist}(\phi, \M_Q^{\eps, \, \ge  d_\eps} ) <\dl_\eps        \} } e^{\jb{T^\eps(\phi),g } } \rho_\eps^Q(d\phi).
\end{align*}

\noi
Here $\dl_\eps= \eta \sqrt{\eps \log \frac 1\eps}$ and $d_\eps= \big|\log(\eps \log \frac 1\eps) \big| $, as defined in \eqref{LDPIN2} and \eqref{LDPIN3}. 
By using the coordinate expression $\phi=m^\eps_{\xi_1,\dots,\xi_Q}+\sqrt{\eps}v$  and Proposition \ref{PROP:cha}, we write 
\begin{align}
\J_1=\cj Z_\eps^{-1} \int\cdots\int_{\Dl_Q} \mathcal{F}_{\xi_1,\dots,\xi_Q}(e^{i\jb{v,g}}) \ind_{ \{ \min_{i \neq j}|\xi_i-\xi_j|  \ge  d_\eps  \} }
 d\xi_1 \dots d\xi_Q,  
\label{M1} 
\end{align}

\noi
where $\Dl_Q=\{-\cj L_\eps\le \xi_1\le \cdots \le \xi_Q \le \cj L_\eps\}$  is defined as in~\eqref{wein1} and
\begin{align*}
\mathcal{F}_{\xi_1,\dots,\xi_Q}(\psi)=\E_{\nu^\perp_{ \xi_1,\dots,\xi_Q} } \Big[ \psi \cdot e^{\frac 1\eps \mathcal{E}_{\xi_1\,\dots,\xi_Q}(\sqrt{\eps}v ) }  
\textup{Det}_{\xi_1,\dots,\xi_Q}(\sqrt{\eps}v), \, \| \sqrt{\eps}v \|_{L^2} < \dl_\eps  \Big].
\end{align*}

\noi
Here, the partition function $\cj Z_\eps$ in \eqref{M1} is 
\begin{align}
\cj Z_\eps=\int\cdots\int_{\Dl_Q} \mathcal{F}_{\xi_1,\dots,\xi_Q}(1) \ind_{ \{ \min_{i \neq j}|\xi_i-\xi_j|  \ge d_\eps  \} } d\xi_1 \dots, d\xi_Q.
\label{part2} 
\end{align}

\noi
We now split the main term $J_1$ in \eqref{M1} into two parts, $\Xi_\eps$ and $\Xi_\eps^c$ as follows  
\begin{align}
\J_1&=\cj Z_\eps^{-1} \int \cdots \int_{ \Xi_\eps }  \mathcal{F}_{\xi_1,\dots,\xi_Q}(e^{i\jb{v,g}}) \ind_{ \{ \min_{i \neq j}|\xi_i-\xi_j|  \ge  d_\eps  \} }
 d\xi_1 \dots, d\xi_Q   \notag \\
&\hphantom{X}+\cj Z_\eps^{-1} \int \cdots \int_{ \Xi_\eps^c }  \mathcal{F}_{\xi_1,\dots,\xi_Q}(e^{i\jb{v,g}}) \ind_{ \{ \min_{i \neq j}|\xi_i-\xi_j|  \ge   d_\eps  \} }
 d\xi_1 \dots, d\xi_Q \notag \\
&=\J_1^{(1)}+\J_1^{(2)},
\label{M2} 
\end{align}

\noi
where
\begin{align}
\Xi_\eps=\big\{(\xi_1,\dots,\xi_Q)\in \Dl_Q: \textup{dist}(\supp g, \xi_j) \ge \eps^{-\frac 12 +2\eta} \; \; \textup{for every $j$} \big\}.
\label{M5}
\end{align}

%On the set $\Xi_\eps$, the test function $g$ is at least a distance $\eps^{-\frac 12+2\eta}$ from all solitons $m_{\xi_1},\dots, m_{\xi_k}$, and the soliton centers $\xi_j$ are far from the boundaries $\pm L_\eps$. Hence, on $\Xi_\eps$ we can use the central limit theorem (Proposition \ref{PROP:OUN}). 

\noi
In order to use Proposition \ref{PROP:OUN}, we decompose 
\begin{align*}
\mathcal{F}_{\xi_1,\dots,\xi_Q}(e^{i\jb{v,g}})&=\mathcal{F}_{\xi_1,\dots,\xi_Q}(e^{i\jb{v,g}} \cdot \ind_{ A_\eps   } )+\mathcal{F}_{\xi_1,\dots,\xi_Q}(e^{i\jb{v,g}} \cdot \ind_{ A_\eps^c   } ),
%+\mathcal{F}_{\xi_1,\dots,\xi_k}(e^{i\jb{v,g}} \cdot \ind_{  B_\eps^c  } )
\end{align*}

\noi
where $A_\eps=\{ \|v\|_{L^\infty} \le K \sqrt{|\log L_\eps|} \}$. Using the tail estimate Lemma \ref{LEM:AEC} with Proposition \ref{PROP:OUN}, we obtain that on the set $\Xi_\eps$
\begin{align}
\mathcal{F}_{\xi_1,\dots,\xi_Q}(e^{i\jb{v,g}})=\exp\Big\{-\frac 12 \jb{g, (-\dx^2+1)^{-1} g }_{L^2(\R)} \Big \}\cdot (1+O(\eps^{0+})),
\label{M3}
\end{align}

\noi
uniformly in $(\xi_1,\dots,\xi_Q)\in  \Xi_\eps $.  For the term $\mathcal{F}_{\xi_1,\dots,\xi_Q}(1)$ appearing in the partition function  $\cj Z_\eps$ in \eqref{part2}, we obtain, using \eqref{part1}, that
\begin{align}
\mathcal{F}_{\xi_1,\dots,\xi_Q}(1)=1+O(\eps^{0+}),
\label{M4} 
\end{align}

\noi
uniformly in $(\xi_1,\dots,\xi_Q)\in \Dl_Q$.  Combining \eqref{M2}, \eqref{M3}, \eqref{part2}, and \eqref{M4} yields that 
\begin{align}
\J_{1}^{(1)}
&= e^{-\frac 12 \| g\|^2_{H^{-1}(\R) } } 
\cdot \frac{  \big| \Xi_\eps \cap   \{  \min_{i \neq j} |\xi_i-\xi_j| \ge  d_\eps   \}     \big|  }{ \big|  \Dl_Q  \cap   \{  \min_{i \neq j} |\xi_i-\xi_j| \ge  d_\eps   \}   \big|  }(1+O(\eps^{0+})).
\label{BJ1}
\end{align}

\noi 
Furthermore,  \eqref{M2}, \eqref{part2}, and \eqref{M4} imply that  
\begin{align}
|\J_{1}^{(2)}| \les  \frac{  \big| \Xi_\eps^c \cap   \{  \min_{i \neq j} |\xi_i-\xi_j| \ge  d_\eps   \}     \big|  }{ \big|  \Dl_Q  \cap   \{  \min_{i \neq j} |\xi_i-\xi_j| \ge  d_\eps   \}   \big|  }(1+O(\eps^{0+})).
\label{BJ0}
\end{align}

\noi
We show that $\J_{1}^{(1)} \to e^{-\frac 12 \| g\|^2_{H^{-1} } }$ and $\J_{1}^{(2)} \to 0 $ as $\eps \to 0$.  The area of the forbidden region, namely the band of length $\cj L_\eps^{Q-1}$ and width $d_\eps$, is 
\begin{align}
&\big| \{ (\xi_1,\dots,\xi_Q) \in \Dl_Q:   \min_{i \neq j} |\xi_i-\xi_j|  <d_\eps  \}     \big| 
\sim O(\cj L_\eps^{Q-1}  d_\eps ).
\label{symp}
\end{align}

\noi
Hence, the separation $\min_{i \neq j}|\xi_i-\xi_j| \ge  d_\eps$ removes only a thin tubular neighbourhood of the diagonal band, of thickness $d_\eps$, from the large cube of side length $\cj L_\eps$,  whose volume is of order $O(\cj L_\eps^{Q-1}  d_\eps )$. Therefore, we have 
 \begin{align}
\big| \Dl_Q  \cap   \{  \min_{i \neq j} |\xi_i-\xi_j| \ge  d_\eps   \} \big| &=\big| \Dl_Q      \big|-O\big(\cj L_\eps^{Q-1} d_\eps\big) \notag \\
&\sim  (2\cj L_\eps)^{Q}\Big(1- O\Big(\frac{d_\eps}{\cj L_\eps}\Big)  \Big) \sim (2\cj L_\eps)^Q,
\label{M6}  
\end{align}

\noi
where in the last line we used $d_\eps=\big|\log (\eps \log \frac 1\eps)\big|$ and $\cj L_\eps\sim \eps^{-\frac 12+\eta}$ from \eqref{cjL}. So the outside band region has almost full area, up to a relative error of order $\tfrac{d_\eps}{\cj L_\eps} \sim \eps^{\frac 12-} $. For the complement of the set $\Xi_\eps$ defined in \eqref{M5}, we have
\begin{align}
|\Xi_\eps^c| &\les Q (\cj L_\eps)^{Q-1} \big| \big\{ \xi_j \in [-\cj L_\eps, \cj L_\eps]: \textup{dist}(\supp g, \xi_j) < \eps^{-\frac 12+2\eta}     \big\}   \big| \notag \\
&\les Q (\cj L_\eps)^{Q-1} (|\supp g|+2\eps^{-\frac 12+2\eta}).
\label{M7}
\end{align}

\noi
Combining \eqref{M6} and \eqref{M7} yields that
\begin{align*}
\frac{  \big| \Xi_\eps^c \cap   \{  \min_{i \neq j} |\xi_i-\xi_j| \ge  d_\eps   \}     \big|  }{ \big|  \Dl_Q  \cap   \{  \min_{i \neq j} |\xi_i-\xi_j| \ge d_\eps   \}   \big|  }
\les  \frac{|\supp g| +2\eps^{-\frac 12+2\eta} }{\cj L_\eps} \les \eps^{\eta},
\end{align*} 

\noi
where we used $\cj L_\eps=L_\eps(1-\eps^\eta-\eps^{\frac 12-\eta} )\sim \eps^{-\frac 12+\eta} $ from \eqref{cjL}. This implies that in \eqref{BJ0} 
$\J_{1}^{(2)} \to 0 $ and so in \eqref{BJ1} $\J_{1}^{(1)} \to e^{-\frac 12 \| g\|^2_{H^{-1} } }$ as $\eps \to 0$.

From \eqref{M9}, \eqref{M2}, and the fact that $\J_{1}^{(1)} \to e^{-\frac 12 \| g\|^2_{H^{-1} } }$, $\J_{1}^{(2)} \to 0 $ as $\eps \to 0$, we conclude that 
\begin{align*}
\lim_{\eps \to 0}\int e^{i\jb{\phi,g }}  (T^\eps)_{\#} \rho_\eps^Q(d\phi)=\exp\Big\{-\frac 12 \jb{g, (-\dx^2+1)^{-1}g } \Big\}.
\end{align*}

\noi
This completes the proof of Theorem \ref{THM:2}.

%\begin{align*}
%\big| \big\{ (\xi_1,\dots,\xi_k) \in [-L_\eps,L_\eps]^k   \big\}  \cap   \{  \min_{i \neq j} |\xi_i-\xi_j| \ge d_\eps   \} \big|=(2L_\eps)^k\Big(1+O\Big(\frac{d_\eps}{L_\eps}\Big)\Big)
%\end{align*}

%In the following, we show that 
%\begin{align*}
%\frac{  \big| \Xi_\eps^c \cap   \{  \min_{i \neq j} |\xi_i-\xi_j| \ge d_\eps   \}     \big|  }{ \big|  \big\{ (\xi_1,\dots,\xi_k) \in [-L_\eps,L_\eps]^k   \big\}  \cap   \{  \min_{i \neq j} |\xi_i-\xi_j| \ge d_\eps   \}   \big|  } \too 0
%\end{align*}

%\noi
%as $\eps \to 0$

\end{proof}

\section{Proof of Theorem \ref{THM:3}}

%The solitons are indistinguishable. Each unordered configuration corresponds to $k!$ identical ordered tuples. 

Recall that  under the coordinate representation $\phi=m^\eps_{\xi_1,\dots,\xi_Q}+\sqrt{\eps}v$,  we take the projection 
$\pi_\eps^T$ onto the tangential directions by $\pi_\eps^T(\phi)=(\xi_1,\dots,\xi_Q)$. The marginal tangential projection is given by
$\pi^T_{j}(\phi)=\xi_{j}$, where $\xi_{j}$ denotes the $j$-th ordered center in the increasing rearrangement $\xi_{1}\le \cdot\le \xi_{Q}$. 

%In the following, we consider the measure
%$\rho_\eps^Q \{ \pi_\eps^{T}(\phi) \in A \}$. 

\begin{proof}[Proof of Theorem \ref{THM:3}]
Using Lemma \ref{LEM:partition}, we decompose the integral into the large-deviation events and the main term $\J_1$ as follows
\begin{align}
\int \ind_{ \{ \pi_\eps^{T}(\phi) \in A \}   }  \rho_\eps^Q(d\phi)=
\J_1+\textup{O}(e^{-c  \log \frac 1\eps }),
\label{RT1}
\end{align}

\noi
where $\J_1=\int_{   \{  \textup{dist}(\phi, \M_Q^{\eps, \, \ge  d_\eps} ) <\dl_\eps        \} } \ind_{ \{ \pi_\eps^{T}(\phi) \in A \}   }  \rho_\eps^Q(d\phi)$.  Here $\dl_\eps= \eta \sqrt{\eps \log \frac 1\eps}$ and $d_\eps= \big|\log(\eps \log \frac 1\eps) \big| $, as defined in \eqref{LDPIN2} and \eqref{LDPIN3}.  By using the coordinate expression $\phi=m^\eps_{\xi_1,\dots,\xi_Q}+\sqrt{\eps}v$ and Proposition \ref{PROP:cha}, we write 
\begin{align}
\J_1&=\cj Z_\eps^{-1}  \int\cdots\int_{\Dl_Q} \mathcal{F}_{\xi_1,\dots,\xi_Q}(1)  \ind_{  \{  (\xi_1,\dots,\xi_Q) \in A    \}   }   \cdot  \ind_{ \{ \min_{i \neq j}|\xi_i-\xi_j|  \ge  d_\eps  \}  } \,
d\xi_1 \dots d\xi_Q,
\label{ORU1}
\end{align}

\noi
where $\Dl_Q=\{-\cj L_\eps\le \xi_1\le \cdots \le \xi_Q \le \cj L_\eps\}$  is defined as in~\eqref{wein1} and
\begin{align*}
\mathcal{F}_{\xi_1,\dots,\xi_Q}(\psi)=\E_{\nu^\perp_{ \xi_1,\dots,\xi_Q} } \Big[ \psi \cdot e^{\frac 1\eps \mathcal{E}_{\xi_1\,\dots,\xi_Q}(\sqrt{\eps}v ) }  
\textup{Det}_{\xi_1,\dots,\xi_Q}(\sqrt{\eps}v), \, \| \sqrt{\eps}v \|_{L^2} < \dl_\eps  \Big].
\end{align*}

\noi
Using \eqref{part1} and Lemma \ref{LEM:AEC}, we decompose the main term and the tail contributions as follows
\begin{align}
\mathcal{F}_{\xi_1,\dots,\xi_Q}(1)&=\mathcal{F}_{\xi_1,\dots,\xi_Q}(\ind_{ A_\eps   } )+\mathcal{F}_{\xi_1,\dots,\xi_Q}(\ind_{ A_\eps^c   } )=1+O(\eps^{0+}),
\label{ORU2}
\end{align}

%%Recall that the multi-soliton profile $m^\eps_{\xi_1,\dots,\xi_Q}=\sum_{j=1}^Q m^\eps(\cdot-\xi_j)$  is indistinguishable. Hence, to avoid overcounting identical configurations, we proceed as follows, using \eqref{ORU1} and \eqref{ORU2} 
%\begin{align}
%\J_3=\cj Z_\eps^{-1}  Q! \mathop{\int\cdots\int}\limits_{-\cj L_\eps<\xi_1<\cdots<\xi_Q< \cj L_\eps}  \ind_{ \{  (\xi_1,\dots,\xi_Q) \in A   \}  }    \cdot  \ind_{ \{ \min_{i \neq j}|\xi_i-\xi_j|  \ge d_\eps  \}  } \, d\xi_1 \dots d\xi_Q  \cdot (1+O(\eps^{0+})),
%\label{ORU3}
%\end{align}

%\noi
%where we used \eqref{perA} and $\ind_{ \{ \min_{i \neq j}|\xi_i-\xi_j|  \ge d_\eps  \}  }$ is symmetric.  

\noi
uniformly in $(\xi_1,\dots,\xi_Q) \in \Dl_Q$. 
By following the same procedure, we can also write the partition function $\cj Z_\eps$ in \eqref{ORU1} as follows
\begin{align}
\cj Z_\eps = \int\cdots\int_{\Dl_Q}   \ind_{ \{ \min_{i \neq j}|\xi_i-\xi_j|  \ge  d_\eps  \} } d\xi_1 \dots d\xi_Q  \cdot (1+O(\eps^{0+})).
\label{ORU4}
\end{align}

\noi
Combining \eqref{ORU1}, \eqref{ORU2}, and \eqref{ORU4} yields that 
\begin{align}
\J_1=\frac{\big|A \cap \{ \min_{i \neq j}|\xi_i-\xi_j|  \ge  d_\eps  \}  \cap \Dl_Q \big|}{ \big|\{ \min_{i \neq j}|\xi_i-\xi_j|  \ge  d_\eps  \} \cap  \Dl_Q \big|  }(1+O(\eps^{0+}) ).
\label{ORU5}
\end{align}

\noi
From \eqref{M6}, we have 
\begin{align}
\big| \Dl_Q  \cap   \{  \min_{i \neq j} |\xi_i-\xi_j| \ge  d_\eps   \} \big|&\sim  (2\cj L_\eps)^{Q}\Big(1- O\Big(\frac{d_\eps}{\cj L_\eps}\Big)  \Big) \sim (2\cj L_\eps)^Q.
\label{ORU6}
\end{align}

\noi
Therefore, the effect of imposing $\min_{i \neq j} |\xi_i-\xi_j| \ge  d_\eps$ is negligible compared to the total volume of the simplex  $\frac{(2\cj L_\eps)^Q}{Q!}$.  Using \eqref{ORU5} and \eqref{ORU6}, we obtain   
\begin{align}
\J_1=\frac{|A   \cap \Dl_Q |}{ | \Dl_Q |  }(1+O(\eps^{0+}) ).
\label{ORU7}
\end{align}

%Therefore, the separation condition $\min_{i \neq j} |\xi_i-\xi_j| \ge d_\eps$ removes only a thin diagonal band of thickness $d_\eps$, which is negligible compared to the total volume of the simplex $\frac{(2L_\eps)^k}{k!}$.

\noi
Combining  \eqref{RT1} and \eqref{ORU7} yields 
\begin{align}
\rho_\eps^Q \{ \pi_\eps^{T}(\phi) \in A \}
= \frac{|A   \cap \Dl_Q |}{ | \Dl_Q |  }(1+O(\eps^{0+}) ).
\label{M11} 
\end{align}

\noi
Therefore, we obtain the desired result.

By following the arguments used to obtain \eqref{M11}, we have
\begin{align}
\rho^Q_\eps\big\{ \pi^T_{j}(\phi) \in B  \big\}=\frac{ | \{\xi_j \in B  \} \cap \Dl_Q |}{ |\Dl_Q |  }(1+O(\eps^{0+}) ).
\label{M15}
\end{align}

In the following, our goal is to find a density function $f_j(x)$
\begin{align*}
\rho_\eps^Q\big\{ \pi^T_{j}(\phi) \in B  \big\}= \int_{B} f_j(x) dx \cdot (1+O(\eps^{0+})).
\end{align*}

\noi
Fix $x \in (-\cj L_\eps, \cj L_\eps)$ and look at the slice with $\xi_j=x$ as follows \begin{align}
S_j(x)=\{  (\xi_1,\dots,\xi_Q) \in \Dl_Q  : \xi_j=x   \}.
\label{M12}
\end{align}

\noi
Then, $(\xi_1,\dots, \xi_{j-1})$ form an ordered simplex $-\cj L_\eps \le \xi_1\le \dots \le \xi_{j-1}\le x$ in the interval $[-\cj L_\eps, x]$, volume
\begin{align}
\frac{(x+\cj L_\eps)^{j-1}}{(j-1)!}.
\label{M13}
\end{align}

\noi
Also, $(\xi_{j+1},\dots, \xi_Q)$ form an ordered simplex $x\le \xi_{j+1}\le \dots \le \xi_{Q} \le \cj L_\eps$ in the interval $[x, \cj L_\eps]$, volume
\begin{align}
\frac{ (\cj L_\eps-x)^{Q-j}  }{(Q-j)! }.
\label{M14}
\end{align}

\noi
Combining \eqref{M12}, \eqref{M13} and \eqref{M14} yields 
\begin{align*}
| \{\xi_j \in B  \} \cap \Dl_Q |=\int_B |S_j(x)| dx=\int_B \frac{(x+\cj L_\eps)^{j-1}}{(j-1)!} \frac{ (\cj L_\eps-x)^{Q-j}  }{(Q-j)! } dx.
\end{align*}

\noi
This, along with \eqref{M15}, implies that
\begin{align*}
\rho_\eps^Q\big\{ \pi^T_{j}(\phi) \in B  \big\}=\int_B f_j(x) dx \cdot (1+O(\eps^{0+}) ),
\end{align*}

\noi
where 
\begin{align*}
f_j(x)= \frac{Q!}{ (2\cj L_\eps)^Q }\frac{(x+\cj L_\eps)^{j-1}}{(j-1)!} \frac{ (\cj L_\eps-x)^{Q-j}  }{(Q-j)! }, \; \; -\cj L_\eps<x<\cj L_\eps.
\end{align*}

\noi
Recall that $\pi^T_{j}(\phi)=\xi_j \in  [-\cj L_\eps, \cj L_\eps]$. Rescaling by $\cj L_\eps$, define $V_j:=\frac{\xi_j+\cj L_\eps}{2\cj L_\eps} \in [0,1]$. Let $x=2\cj L_\eps v-\cj L_\eps$ and $dx=2\cj L_\eps dv$. Then, the density of $V_j$ is 
\begin{align*}
f_{V_j}(v)=f_j(x) \cdot (2\cj L_\eps)&=\frac{Q!}{(2\cj L_\eps)^Q} \frac{(2\cj L_\eps v)^{j-1} (2\cj L_\eps(1-v)  )^{Q-j} }{(j-1)! (Q-j)!  } \cdot (2\cj L_\eps)\\
&=\frac{Q!}{ (j-1)! (Q-j)! } v^{j-1} (1-v)^{Q-j},
\end{align*}

\noi
where $0<v<1$. This shows $V_j \sim \textup{Beta}(j,Q+1-j)$. Hence, 
$\pi^T_{(j)}(\phi)=\xi_{j}=-\cj L_\eps +2\cj L_\eps V_j$ follows a Beta distribution, whose expected location is given by 
\begin{align*}
\E_{\rho_\eps^Q } \big[  \pi^T_{j}(\phi)  \big]=\Big(-\cj L_\eps+\frac{2\cj L_\eps j}{Q+1} \Big) \cdot (1+O(\eps^{0+}) ).
\end{align*}

\noi
This implies that soliton centers $\xi_1,\dots,\xi_{Q}$ are evenly spaced, dividing the interval $[-\cj L_\eps, \cj L_\eps]$ into $Q + 1$ equal parts of length $\tfrac{2\cj L_\eps}{Q + 1}$.

\end{proof}

\begin{ackno}\rm 
K.S. would like to thank  Hendrik Weber for pointing out helpful references. H.S. gratefully acknowledges financial support from an NSF grant (CAREER DMS-2044415). The work of P.S. is partially supported by NSF grants DMS-1811093,  DMS-2154090, and a Simons Fellowship.  
\end{ackno}

\end{document}